\documentclass[11pt]{amsart}
\allowdisplaybreaks
\usepackage[english]{babel}
\usepackage[pdftex,textwidth=400pt,margin=3cm]{geometry}
\usepackage{amsfonts}
\usepackage{amsmath}
\usepackage{amsthm}
\usepackage{amssymb}
\usepackage{bbm}
\usepackage{cancel}
\usepackage{color}
\usepackage[curve]{xypic}
\usepackage{graphicx}
\usepackage{mathabx}
\usepackage{tikz-cd}
\usepackage{epigraph} 
\usepackage{hyperref}
\usepackage{stmaryrd}
\usepackage{enumitem}
\usepackage[mathscr]{eucal}
\usepackage{soul}
\usepackage{mathtools}
\usepackage{xfrac}

\newtheorem{theorem}{Theorem}[section]
\newtheorem{theoremX}{Theorem}

\newtheorem*{theorem*}{Theorem}
\newtheorem{corollary}[theorem]{Corollary}
\newtheorem{corollaryX}[theoremX]{Corollary}
\newtheorem*{corollary*}{Corollary}
\newtheorem{proposition}[theorem]{Proposition}

\newtheorem*{proposition*}{Proposition}
\newtheorem{lemma}[theorem]{Lemma}

\newtheorem{conjecture}[theorem]{Conjecture}

\newtheorem*{conjecture*}{Conjecture}

\newtheorem{assumption}[theorem]{Assumption}
\theoremstyle{definition}
\newtheorem{definition}[theorem]{Definition}

\newtheorem*{definition*}{Definition}

\newtheorem*{variant*}{Variant}
\newtheorem{example}[theorem]{Example}

\newtheorem*{example*}{Example}
\newtheorem{remark}[theorem]{Remark}
\newtheorem*{remark*}{Remark}

\newtheorem*{notation*}{Notation}
\newtheorem{construction}[theorem]{Construction}

\newtheorem*{question*}{Question}
\newtheorem*{acknowledgements}{Acknowledgements}

\usepackage[colorinlistoftodos]{todonotes}

\let\oldtocsection=\tocsection
\let\oldtocsubsection=\tocsubsection
\let\oldtocsubsubsection=\tocsubsubsection
\renewcommand{\tocsection}[2]{\hspace{0em}\oldtocsection{#1}{#2}}
\renewcommand{\tocsubsection}[2]{\hspace{1em}\oldtocsubsection{#1}{#2}}
\renewcommand{\tocsubsubsection}[2]{\hspace{2em}\oldtocsubsubsection{#1}{#2}}

\def\CZ{\mathrm{CZ}}

\def\univ{\mathrm{univ}}

\def\Pic{\mathrm{Pic}}

\def\fC{\mathfrak{C}}
\def\BM{\mathrm{BM}}

\def\un{\mathrm{un}}
\def\ev{\mathrm{ev}}

\def\Ker{\mathrm{Ker}}

\def\Perf{\mathrm{Perf}}

\def\sp{\mathrm{sp}}

\def\lag{\mathrm{lag}}
\def\vir{\mathrm{vir}}
\def\rank{\mathrm{rank}}
\def\ori{\mathrm{or}}

\def\symp{\mathrm{symp}}

\def\pr{\mathrm{pr}}
\def\id{\mathrm{id}}

\def\GG{\mathbb{G}}

\def\cart{\ar@{}[rd]|{\Box}}
\def\Spec{\mathrm{Spec}}

\def\Tot{\mathrm{Tot}}

\def\Crit{\mathrm{Crit}}

\def\dual{^{\vee}}

\def\AA{\mathbb{A}}

\def\Z{\mathbb{Z}}
\def\Q{\mathbb{Q}}
\def\C{\mathbb{C}}

\def\O{\mathcal{O}} 
\def\I{\mathcal{I}} 

\def\LL{\mathbb{L}}
\def\TT{\mathbb{T}}
\def\T{\mathrm{T}}

\def\dSt{\mathrm{dSt}}

\def\QCAlg{\mathrm{QCAlg}}
\def\QCoh{\mathrm{QCoh}}

\def\Map{\mathrm{Map}}
\def\uMap{\underline{\Map}}
\def\fib{\mathrm{fib}}
\def\cof{\mathrm{cof}}

\def\fil{\mathrm{fil}}

\def\DR{\mathrm{DR}}

\def\Gr{\mathrm{Gr}}
\def\Fil{\mathrm{Fil}}
\def\Sym{\mathrm{Sym}}

\def\textin{\quad\textup{in}\quad}
\def\and{\quad\textup{and}\quad}

\def\sA{\mathscr{A}}
\def\lc{\mathrm{lc}}
\def\cl{\mathrm{cl}}

\def\D{\mathrm{D}}

\def\Zero{\mathrm{Z}}

\def\Q{\mathbb{Q}}

\def\ori{\mathrm{or}}

\def\commutes{\ar@{}[rd]|-{\circlearrowleft}}

\def\pt{\mathrm{pt}}
\def\Perv{\mathrm{Perv}}

\def\tE{\widetilde{E}}

\def\tsigma{\widetilde{\sigma}}

\def\tM{\widetilde{M}}

\def\cS{\mathcal{S}}
\def\cD{\mathcal{D}}

\def\beq{\begin{equation} }
\def\eeq{\end{equation} }
\def\lra{\longrightarrow }
\def\bbL{\mathbb{L} }
\def\cE{\mathcal{E} }
\def\fC{\mathfrak{C} }
\def\virt{^{\mathrm{vir}} }
\def\bM{\mathbf{M} }
\def\bbG{\mathbb{G} }
\def\bbA{\mathbb{A} }
\def\CC{\mathbb{C} }
\def\ZZ{\mathbb{Z} }
\def\bsig{\boldsymbol\sigma }
\def\obsig{\overline{\bsig} }
\def\btheta{{\boldsymbol\theta}}

\def\bff{\mathbf{f}}
\def\cC{\mathcal{C}}
\def\PP{\mathbb{P}}
\def\mapright#1{\,\smash{\mathop{\lra}\limits^{#1}}\,}

\def\M{\mathfrak{M}}

\def\tw{\mathrm{tw}}
\def\log{\mathrm{log}}
\def\Res{\mathrm{Res}}
\def\cL{\mathcal{L}}
\def\AKSZ{\mathrm{AKSZ}}

\def\tbsig{\widetilde{\bsig}}

\begin{document}

\title{Cosection localization via shifted symplectic geometry}

\author{Young-Hoon Kiem}
\address{School of Mathematics, Korea Institute for Advanced Study, 85 Hoegiro, Dongdaemun-gu, Seoul 02455, Korea}
\email{kiem@kias.re.kr}

\author{Hyeonjun Park}
\address{June E Huh Center for Mathematical Challenges, Korea Institute for Advanced Study, 85 Hoegiro, Dongdaemun-gu, Seoul 02455, Republic of Korea}
\email{hyeonjunpark@kias.re.kr}

\date{April 28, 2025}

\maketitle
\begin{abstract} 
The purpose of this paper is to shed a new light on classical constructions in enumerative geometry from the view point of derived algebraic geometry. 
We first prove that the cosection localized virtual cycle of a quasi-smooth derived Deligne-Mumford stack with a $(-1)$-shifted closed $1$-form is equal to the virtual Lagrangian cycle of the degeneracy locus which is $(-2)$-shifted symplectic. 
We next establish a shifted analogue of the Lagrange multipliers method 
which gives us the quantum Lefschetz theorems as immediate consequences of the equality of virtual cycles. 
Lastly we study derived algebraic geometry enhancements of gauged linear sigma models 
which lead us to the relative virtual cycles in a general and natural form.  
\end{abstract}

\tableofcontents
\addtocontents{toc}{\protect\setcounter{tocdepth}{1}}

\section*{Introduction}
Throughout history, enumerative geometry has been enriched by influences from various related areas such as topology, representation theory and mathematical physics. 
Recently, derived algebraic geometry opened new avenues of research in enumerative geometry such as the 4-dimensional Donaldson-Thomas (DT for short) theory and the categorified DT theories.  Classical constructions in enumerative geometry often arise naturally from the underlying derived structures which lead us to generalizations 
and to connections of seemingly unrelated structures. 
 
The purpose of this paper is to develop derived algebraic geometry enhancements of 
\begin{enumerate}
\item cosection localization of virtual cycles,
\item Lagrange multipliers method and
\item gauged linear sigma models
\end{enumerate}
which enable us to prove the following pleasant results: 
\begin{enumerate}
\item the equality of the localized virtual cycle by a cosection in \cite{KL} and  the virtual Lagrangian cycle of the degeneracy locus of the cosection by the Oh-Thomas construction in \cite{OT},
\item a generalization of the quantum Lefschetz theorem without curves in \cite{OTq} by the shifted Lagrange multipliers method and
\item a streamlined construction of the virtual cycles of gauged linear sigma models  which may be thought of as a common generalization of those in \cite{FJR, CZ}. 
\end{enumerate}
By (1), the cosection localized virtual cycle can now be thought of as the virtual cycle of the degeneracy locus which is (-2)-shifted symplectic. 
The higher genus quantum Lefschetz theorem for Gromov-Witten invariants in \cite{CL} is an immediate consequence of (2) where the additional field, called the $p$-field, turns out to be just the Lagrange multiplier. 
All these results are established for the relative setting of a quasi-smooth morphism $g:M\to B$ of derived stacks and the classical constructions of (1)-(3) are fully relativized by our results. 

A detailed introduction comes below.


\subsection{Virtual fundamental class} Let $Q$ be a smooth projective variety over the complex number field $\C$. Many problems in algebraic geometry can be reduced to finding suitable objects like curves or sheaves on $Q$, satisfying certain desired properties. Enumerative geometry is about counting those desired objects and has played a fundamental role in algebraic geometry and beyond. 

To solve an enumerative geometry problem, one usually constructs a moduli space $M$ of all objects of given topological type and then finds the intersection of cycles defined by the desired properties. Unfortunately the moduli spaces are often highly singular and it may not be possible to formulate and work out an intersection theory. 

In early 1990s, Kontsevich proposed that one might use a deeper structure of a moduli space to construct a homology class $[M]\virt$, called the \emph{virtual fundamental class} or \emph{virtual cycle}, which behaves nicely like the fundamental class of a smooth variety. 
Subsequently, Li-Tian \cite{LT} and Behrend-Fantechi \cite{BF} showed that if $M$ is equipped with an additional structure, called a \emph{tangent-obstruction complex} or a \emph{perfect obstruction theory}, we have a virtual cycle $[M]\virt$ of $M$. 
Then virtual enumerative invariants like Gromov-Witten (counting curves) or Donaldson-Thomas (counting sheaves) were defined as integrals of cohomology classes against $[M]\virt$. 

These virtual invariants satisfy nice properties like deformation invariance but actual computations are known to be extremely difficult, mainly because there are only a few techniques to handle the virtual cycles. 

The construction of virtual cycles was generalized to the virtual pullback 
$$[M/B]\virt:A_*(B)\lra A_{*+r}(M)$$
of a morphism $g:M\to B$ with a relative perfect obstruction theory of rank $r$ (cf. \cite{Man}).

\subsection{Cosection localization}
The cosection localization in \cite{KL} is a technique for the virtual cycle $[M]\virt$ which 
turned out to be quite useful. 
A perfect obstruction theory on a Deligne-Mumford stack $M$ is a morphism
\beq\label{i0} \phi_M:E_M\lra \bbL_M\eeq
in the derived category of quasi-coherent sheaves on $M$ from a perfect complex of amplitude $[-1,0]$ to the cotangent complex of $M$, such that $h^{-1}(\phi_M)$ is surjective and $h^0(\phi_M)$ is an isomorphism. 
Then the tangent sheaf of $M$ is isomorphic to $h^0(E_M\dual)$ and $h^1(E_M\dual)$ is called the \emph{obstruction sheaf} $Ob_M$ whose vanishing implies the smoothness of $M$. 
The cosection localization in \cite{KL} tells us that a cosection 
\beq\label{i1} \sigma_M:Ob_M\lra \O_M\eeq
of the obstruction sheaf localizes the virtual cycle $[M]\virt$ to the \emph{degeneracy locus} $M(\sigma)$ of $\sigma$, i.e. the locus where $\sigma$ is not surjective. 
Indeed, we have a Chow homology class
\beq\label{i2} [M,\sigma]\virt\in A_r(M(\sigma)), \quad r=\rank\,E_M\eeq 
whose pushforward to $M$ equals $[M]\virt=[M,0]\virt$. 

For example, when $M$ is the moduli space of stable maps to a smooth projective surface $S$ with an effective canonical curve $D$, there is a cosection whose degeneracy locus is the moduli space of stable maps to $D$ and thus the computation of the Gromov-Witten invariants of $S$ is reduced to 
$D$ (cf. \cite{KL4, KL5}). 

The cosection localized virtual pullback $[M/B,\sigma]\virt$ for a morphism $g:M\to B$ with a relative perfect obstruction theory can be constructed when the intrinsic normal cone has support in the kernel of the cosection (Assumption \ref{20}) or when the cosection is \emph{locked} (Definition \ref{25}). See \cite{CKL} or \S\ref{S1.4} and \S\ref{S1.5}.

\subsection{Shifted symplectic structure on the degeneracy locus of a closed cosection}
Many moduli spaces in enumerative geometry such as moduli spaces of stable maps or stable sheaves can be enhanced to derived stacks $M$ (cf. \cite{Toen}) where the perfect obstruction theory \eqref{i0} is in fact the natural morphism
\beq\label{i3} E_{M_\cl}=\bbL_M|_{M_\cl}\lra \bbL_{M_\cl}\eeq
induced by the classical truncation $M_\cl\to M$ when $M$ is \emph{quasi-smooth}, i.e. the cotangent complex $\bbL_M$ has tor-amplitude $[-1,0]$. 
Then a cosection of the obstruction sheaf can be thought of as a map
\[ \sigma:\TT_M|_{M_\cl}[1]\lra \O_{M_\cl}\]
where $\TT_M=\bbL_M\dual$ denotes the tangent complex of $M$. 
For all the known examples, the cosection lifts to a map
\[ \sigma: \TT_M[1] \lra \O_M\]
which is nothing but a $(-1)$-shifted 1-form on $M$.

In most cases, the $(-1)$-shifted 1-form $\sigma$ whose classical truncation is the cosection of the obstruction sheaf can be further enhanced to a \emph{closed} form $\bsig$ in the sense of \cite{PTVV} (see Remark \ref{Rem:2.8}). Then the degeneracy locus $M(\sigma)$ of $\sigma$ is the Lagrangian intersection
\[\xymatrix{
M(\sigma)\ar[r]\ar[d] & M\ar[d]^{\bsig}\\
M\ar[r]^0 & \T^*_M[-1],
}\]
where $\T^*_M[-1]=\mathrm{Spec}\,\mathrm{Sym}(\TT_M[1])$ is the $(-1)$-shifted cotangent bundle of $M$ and the sections are Lagrangian by \cite[Thm.~2.22]{Cal}. 
By \cite[Thm.~2.9]{PTVV}, $M(\sigma)$ admits a \emph{$(-2)$-shifted symplectic structure} $\btheta$. 

If a derived Deligne-Mumford stack is equipped with a $(-2)$-shifted symplectic structure, its classical truncation admits a 3-term symmetric obstruction theory and the intrinsic normal cone is isotropic by \cite{BBJ, BG} (cf.~\cite[Prop.~4.3]{OT}). 
By \cite{OT, Park1, Park2}, we then have the \emph{virtual Lagrangian cycle} (Theorem \ref{60})
which should be thought of as the virtual fundamental class of a Lagrangian (Corollary \ref{Cor:1}). 
In particular, the degeneracy locus $M(\sigma)$ of a cosection $\sigma$ with the canonical $(-2)$-shifted symplectic form $\btheta$ admits the virtual Lagrangian cycle 
\beq\label{i4}
[M(\sigma), \btheta]^\lag\in A_r(M(\sigma)),\quad r=\frac12 \rank \, \bbL_{M(\sigma)}.\eeq

All these constructions are worked out in the relative setting of a quasi-smooth morphism $g:M\to B$ from a derived Deligne-Mumford stack to a derived Artin stack, for a $(-1)$-shifted \emph{locked} 1-form $\bsig$ (Definition \ref{41}). 
The lockedness to a regular function $w:B\to \AA^1$ means that $M$ is locally the zero locus of a section $s$ of a vector bundle $F$ over a smooth stack $V$ over $B$ and the cosection lifts to a map $\sigma_F:F\to \AA^1$ such that $\sigma_F\circ  s=w|_V$ (cf. Definition \ref{25} or Proposition \ref{52}). 
For a $w$-locked $\bsig$, over the zero locus $B(w)=w^{-1}(0)$, the normal cone of $M$ is contained in the kernel of the cosection and thus the cone reduction (Assumption \ref{20}) holds over $B(w)$.  
When $B$ is a point, a locked 1-form is just a closed 1-form and vice versa (Remark \ref{48}).

\subsection{Comparison of virtual cycles}
Let $M$ be a quasi-smooth derived Deligne-Mumford stack equipped with a $(-1)$-shifted closed 1-form $\bsig$.
Then we have two virtual cycles \eqref{i2} and \eqref{i4} of the same dimension 
$$r=\rank\, E_{M_{\cl}}=\rank \, \bbL_M = \frac 12\rank \, \bbL_{M(\sigma)},$$
and it is natural to ask the following.

\medskip
\noindent \textbf{Question}. How are the virtual cycles \eqref{i2} and \eqref{i4} related?

\medskip

The first main result of this paper provides an answer to this question.

\begin{theoremX}[Theorem \ref{Thm:main}] \label{i6}
Let $M$ be a quasi-smooth derived Deligne-Mumford stack
and $\bsig$ be a $(-1)$-shifted closed $1$-form on $M$.
Denote by $M(\sigma)$ the degeneracy locus.
Then we have the equality 
\beq\label{i14} [M,\sigma]^\vir =  [M(\sigma), \btheta]^\lag \in A_r(M(\sigma))\eeq 
of the cosection localized virtual cycle \eqref{i2} and the virtual Lagrangian cycle \eqref{i4}. 
\end{theoremX}
In fact, Theorem \ref{i6} is proved in the relative setting of a quasi-smooth morphism 
$g:M\to B$ and we have the equality 
\beq\label{i7}
[M/B,\bsig]^\vir =  [M(\sigma)/B, \btheta]^\lag:A_*(B(w))\lra A_{*+r}(M(\sigma))\eeq
of the virtual pullbacks for a  $w$-locked $\bsig$ where $w:B\to \AA^1$ is a function and $B(w)=w^{-1}(0)$. See \S\ref{Sec:3.1} for details. 

Theorem \ref{i6} shows that the cosection localized virtual cycle is an {\em intrinsic} object of the degeneracy locus, which is useful for the computations. For instance, see Corollary \ref{i16}.

By Theorem \ref{i6}, various functorial properties proved for virtual Lagrangian cycles automatically hold for cosection localized virtual cycles. In particular,  the virtual pullback formula and torus localization formula in \cite{CKL} follow directly from those in \cite{Park1} when the cosection is enhanced to a $(-1)$-shifted locked 1-form. 

\subsection{Shifted Lagrange multipliers method and quantum Lefschetz}

In calculus, if  a manifold of interest arises as the zero locus $X=s^{-1}(0)$ of a section $s$ (transveral to the zero section) of a vector bundle $E$ on a manifold $M$ and we want to find the critical points of the restriction $f|_X$ of a function $f$ on $M$, we consider the dual bundle $F=E\dual$ (whose fiber coordinates are called Lagrange multipliers) and the function $\varphi=s\dual:F\to \AA^1$ given by pairing with $s$. Then we find that the critical locus of the function $f|_X$ on $X$ is isomorphic to that of $f|_F+\varphi$ on $F$. Needless to say, this Lagrange multipliers method is one of the most useful techniques in mathematics. 

Our second main result establishes a $(-1)$-shifted version of the Lagrange multipliers method which leads us to a far reaching generalization (Theorem \ref{Cor:Lefschetz}) of the quantum Lefschetz formula (Example \ref{g20}). 

Let $M$ be a derived Artin stack and $E$ be a perfect complex with a section $s$. Note that we do not assume transversality with the zero section. Let $X=M\times_{0,E,s}M$ denote the zero locus of $s$ and $f:M\to \bbA^1[-1]$ be a $(-1)$-shifted function.
Let $$F=E\dual[-1]=\mathrm{Spec}\,\mathrm{Sym}(E[1])$$ be the \emph{bundle of Lagrange multipliers} and let $$\varphi=s\dual[-1]:F\lra \AA^1[-1]$$ denote the pairing with $s$.  
\begin{theoremX}[Theorem \ref{Prop:LagMult}]\label{i8}
With the above notation, we have the canonical equivalence
\beq\label{i10} \Crit_X(f|_X)\simeq \Crit_F(f|_F+\varphi)\eeq
of the critical loci of $f|_X$ and $f|_F+\varphi$.
More generally, if $\bsig$ is a $(-1)$-shifted closed 1-form on $M$, we have a canonical equivalence
\beq\label{i9} X(\bsig|_X)\simeq F(\bsig|_F+d\varphi)\eeq
of the degeneracy loci as $(-2)$-shifted symplectic stacks,
where $\bsig|_X$ and $\bsig|_F$ denote the pullbacks of $\bsig$ to $X$ and $F$ respectively.
\end{theoremX}
Note that \eqref{i10} follows from \eqref{i9} by letting $\bsig=df$. In particular, when $f\simeq 0$, we have
$$\T_X^*[-2]\simeq X\times_{0,\T^*_X[-1],0}X\simeq\Crit_X(0)\simeq \Crit_F(\varphi).$$

Combining Theorem \ref{i8} with Theorem \ref{i6}, we have the following. 
\begin{corollaryX}[Theorem \ref{Cor:Lefschetz}]\label{i16} With the above notation, 
if $M$ is a quasi-smooth Deligne-Mumford stack, $E$ is of tor-amplitude $[0,1]$ and $X$ is also quasi-smooth, then we have the equality
\beq\label{i11} [X,\bsig|_X]\virt=(-1)^\ell\cdot [F,\bsig|_F+d\varphi]\virt\in A_r(X(\sigma))\eeq
of cosection localized virtual cycles where $\ell=\rank(E)$ and $r=\rank\, \TT_X$. In particular, letting $\bsig\simeq 0$, we have the equality 
\beq\label{i12} [X]\virt =(-1)^\ell \cdot [F,d\varphi]\virt\in A_r(X).\eeq   
\end{corollaryX}

The equality \eqref{i11} is an immediate consequence of \eqref{i9} and \eqref{i14} (cf. \eqref{i15} with $B=\mathrm{pt}$). 

All the above constructions and theorems are worked out for the relative setting $g:M\to B$ below.  Note that the equality \eqref{i12} is the \emph{quantum Lefschetz without curves} proved in \cite{OTq} by an independent method. 

\subsection{Gauged linear sigma models}
The theory of gauged linear sigma models (GLSM for short) in \cite{FJR} aims at constructing 
\emph{cohomological field theories} for geometric invariant theory (GIT) quotients $V/\!/G$ of a vector space $V$ acted on linearly by a reductive group $G$. It further comes with a function $w:V/G\to \AA^1$, called the \emph{superpotential} whose critical locus $\Crit(w)$ plays a major role. 

The theory of GLSM in fact uses a larger group $\widehat{G}$ which acts on $V$ and contains $G$ as the kernel of a surjective homomorphism $\chi:\widehat{G}\to \GG_m$ such that $w$ is of weight 1 with respect to the residual action of $\GG_m$.   
The moduli stack $\M_{g,n}^{\omega^{\log}}(V/G) $ for a GLSM parameterizes 
\beq\label{108} (C,P\to V, P/G\cong \mathring{\omega}_C^\log)\eeq
where $C\in \M^\tw_{g,n}$ is a twisted prestable curve of genus $g$ with $n$ marked points, $P$ is a principal $\widehat{G}$-bundle over $C$, $P\to V$ is a $\widehat{G}$-equivariant map and $\mathring{\omega}_C^\log$ denotes the log dualizing line bundle (Example \ref{Ex:Dualizing}) minus the zero section. 

More generally, we can define such a moduli stack $\M^{\omega^{\log}}_{g,n}(U)$ for any Artin stack $U$ with a $\GG_m$-action (Definition \ref{Def:TwMaps}).
Given a $\GG_m$-equivariant function $w:U \to \AA^1$ of weight $1$, the critical locus $\Crit(w)$ has an induced $\GG_m$-action and we can also consider the moduli stack $\M^{\omega^{\log}}_{g,n}(\Crit(w))$.
We then have the evaluation maps (Definition \ref{Def:ev})
\beq\label{100}\M_{g,n}^{\omega^\log}(U)\lra I(U)^{\times n}\times \M_{g,n}^\tw,\eeq
\beq\label{101}  \M_{g,n}^{\omega^\log}(\Crit(w))\lra I(\Crit(w))^{\times n}\times \M_{g,n}^\tw\eeq 
where $I(U)$ and $I(\Crit(w))$ denote the (rigidified cyclotomic) inertia stacks (Definition \ref{Def:Inertia}) of $U$ and $\Crit(w)$ respectively.
Denote by $Iw : I(U) \to \AA^1$ the induced function.

In this paper, we construct a $(-1)$-shifted $\underline{w}$-locked 1-form $\bsig_w$ for the evaluation map \eqref{100} (Lemma \ref{Lem:GLMSCosection}) 
so that we have the cosection localized virtual cycle (Construction \ref{Const:3})
\beq\label{103}
[\M_{g,n}^{\omega^\log}(U)^\alpha,\sigma_w]\virt: A_*(\underline{w}^{-1}(0))\lra 
 A_*(\M_{g,n}^{\omega^\log}(U)^\alpha(\sigma_w))
\eeq
where $\underline{w}:I(U)^{\times n}\times \M_{g,n}^\tw\to \AA^1$ denotes the sum $\sum_{i=1}^n\pr_i^*(Iw)$, $\M_{g,n}^{\omega^\log}(U)^\alpha$ is any open Deligne-Mumford substack of $\M_{g,n}^{\omega^\log}(U)$ and $\M_{g,n}^{\omega^\log}(U)^\alpha(\sigma_w)$ is the degeneracy locus of the cosection in $\M_{g,n}^{\omega^\log}(U)^\alpha$.

We furthermore construct a $(-1)$-shifted exact Lagrangian structure on \eqref{101} over $\M^\tw_{g,n}$ (Proposition \ref{Prop:AKSZ}) which induces a $(-2)$-shifted $\underline{w}$-locked symplectic structure on the composition 
\beq\label{102}  \M_{g,n}^{\omega^\log}(\Crit(w))\lra I(\Crit(w))^{\times n}\times \M_{g,n}^\tw\lra I(U)^{\times n}\times \M_{g,n}^\tw.\eeq  
By Theorem \ref{60}, we then have the virtual Lagrangian cycle 
\beq\label{104}
[\M_{g,n}^{\omega^\log}(\Crit(w))^\alpha]^\lag: A_*(\underline{w}^{-1}(0))\lra 
 A_*(\M_{g,n}^{\omega^\log}(\Crit(w))^\alpha)
\eeq
where $\M_{g,n}^{\omega^\log}(\Crit(w))^\alpha=\M_{g,n}^{\omega^\log}(U)^\alpha\cap \M_{g,n}^{\omega^\log}(\Crit(w))$.

On top of these, our third main result enables us to compare \eqref{103} with \eqref{104}. 
\begin{theoremX}[Theorem \ref{Prop:GLSMmain}] \label{105}
Let $U$ be a smooth Artin stack with a $\GG_m$-action and a $\GG_m$-equivariant function $w:U \to \AA^1$ of weight $1$.
Then we have a canonical equivalence
\beq\label{106} 
\M_{g,n}^{\omega^\log}(U)(\sigma_w)\simeq  \M_{g,n}^{\omega^\log}(\Crit(w))
\eeq
as $\underline{w}$-locked $(-2)$-shifted symplectic fibrations over $I(U)^{\times n}\times \M_{g,n}^\tw$. 
\end{theoremX}

Now Theorem \ref{i6} in the form of \eqref{i7} together with Theorem \ref{105} immediately implies the following.
\begin{corollaryX}[Corollary \ref{Cor:2}]\label{107} With the above notation, we have the equality
\[ [\M_{g,n}^{\omega^\log}(U)^\alpha,\sigma_w]\virt = [\M_{g,n}^{\omega^\log}(\Crit(w))^\alpha]^\lag\] 
of the two virtual cycles \eqref{103} and \eqref{104}. 
\end{corollaryX}

In \cite[Definition 6.1.2]{FJR}, Fan, Jarvis and Ruan assume the factorization  
\beq\label{109} \xymatrix{
&\omega_C\ar[d]\\
P\times_{\widehat{G}}V\ar@{.>}[ur]\ar[r]\ar[d] & \omega^\log_C\ar[d]\\
P\times_{\widehat{G}}V|_\Sigma \ar[r] & \omega_\Sigma
}\eeq
using the notation of \eqref{108} where $\Sigma$ denotes the union of marked points. 
Then they construct an absolute cosection and obtain the cosection localized virtual cycle \cite[Definition 6.1.6]{FJR} which in turn gives the correlators \cite[Definition 6.2.1]{FJR}.
In \eqref{109}, the relative tangent bundle of the left vertical arrow maps to $\omega_C$ which is in fact the relative cosection $\sigma_w$ above.   
With the factorization in \eqref{109}, the relative cosection $\sigma_w$ descends to the absolute cosection in \cite{FJR} and hence the virtual cycle in \cite{FJR} is a special case of ours. 
If we assume the Joyce conjecture (Conjecture \ref{Conj:Joyce}), 
the shifted Lagrangian structure on \eqref{101} gives rise to a variant of \eqref{104} (Corollary \ref{111}) which may lead us to cohomological field theories for all GLSMs with minimal assumptions. 

In \cite{CZ}, Cao and Zhao also deal with derived algebraic geometry enhancement of GLSM and construct a map similar to (17) with additional assumptions. See Remark 7.14 for more details and comparison with ours. 
They also ask if it is related to the virtual cycle in \cite{FJR}. 
Corollary \ref{107} is our answer to this question.

\subsection{The layout}
In \S\ref{S1}, we recall the cosection localization and its relative version in the classical setting. 
In \S\ref{Sec:1}, we construct the cosection localized virtual cycle from derived algebraic geometry point of view. For the relative setting, the cone reduction is handled by the notion of locked forms and the derived deformation space. In \S\ref{Sec:2}, we recall the construction of virtual Lagrangian cycles and review their properties. 
In \S\ref{Sec:3}, we state and prove our first main result, Theorem \ref{i6}. 
In \S\ref{S5}, we formulate the shifted Lagrange multpliers method and prove Theorem \ref{i8} and Corollary \ref{i16}. 
In \S\ref{S6}, we apply Corollary \ref{i16} to Gromov-Witten theory for quantum Lefschetz. 
In \S\ref{S4}, we prove Theorem \ref{105} and Corollary \ref{107} for gauged linear sigma models. In the Appendix, we prove a technical result for Weil restrictions that we need for a proof of Theorem \ref{105}.

\subsection*{Notation and conventions}
\begin{itemize}
\item We use the language of $\infty$-categories \cite{LurHTT,LurHA}.
\item All derived Artin stacks are assumed to be quasi-separated, $1$-Artin, of finite type over $\C$, and have affine stabilizers. DM stands for Deligne-Mumford. 
\item All derived DM stacks are assumed to be separated of finite type over $\C$.
\item The Chow group $A_*(M)$ of a derived Artin stack $M$ is defined to be Kresch's Chow group of its classical truncation $M_\cl$ with rational coefficients \cite{Kre}.
\item For a function $w$ on a stack $B$, its zero locus is denoted by 
$$B(w)=\Zero_B(w)=\Zero(w)=w^{-1}(0).$$
Likewise, the zero locus of a differential form $\sigma$ is denoted by $M(\sigma)=\Zero_M(\sigma)=\Zero(\sigma).$
\item The pullback of a function $f$ or a differential form $\sigma$ on a stack $Y$ by a morphism $X\to Y$ is denoted by $f|_X$ or $\sigma|_X$.
\end{itemize}

\begin{acknowledgements}
We thank Jeongseok Oh, Richard Thomas, and Renata Picciotto for sharing their drafts on related works.
The second named author would like to thank 
Dhyan Aranha, Younghan Bae, Yalong Cao, Adeel Khan, Tasuki Kinjo, Martijn Kool, Alexei Latyntsev, 
Chiu-Chu Melissa Liu, Charanya Ravi, Yukinobu Toda and Gufang Zhao for various discussions on revisiting cosection localization via derived algebraic geometry. 
\end{acknowledgements}

\bigskip
\section{Cosection localization revisited} \label{S1}

In this section, we review the construction of {\em cosection localized virtual cycles} in \cite{KL}.
For its relative version, we introduce the notion of a \emph{locked} cosection (Definition \ref{25}) and define the cosection localized virtual pullback (Construction \ref{28}).

\subsection{Virtual fundamental class}\label{S1.1}
Let $M$ be a DM stack with a \emph{perfect obstruction theory}
\beq\label{1} \phi_M:E_M\lra \bbL_M\quad\text{in }\ \  D^b\mathrm{Coh}(M)\eeq
where $E_M$ is a perfect complex of amplitude $[-1,0]$ and $\bbL_M$ is the cotangent complex of $M$. By the definition in \cite{BF},
$h^{-1}(\phi_M)$ is surjective and $h^0(\phi_M)$ is an isomorphism.  
By a standard argument (cf. \cite[p.1037]{KL}, \cite[\S6.1]{Kis}),  
there is an open cover $\{M_\imath\}$ of $M$ and a diagram
\beq\label{70} \xymatrix{
& F_\imath\ar[d]\\
M_\imath=s_\imath^{-1}(0)\ar@{^(->}[r]  &V_\imath\ar@/_/[u]_{s_\imath}
}\eeq 
where $F_\imath$ is a vector bundle over a smooth variety $V_\imath$ such that 
$\tau^{\ge -1}\phi_M|_{M_\imath}$ is isomorphic to the standard perfect obstruction theory
\beq\label{69}\xymatrix{
E_M|_{M_\imath}\ar[d]_{\tau^{\ge -1}\phi_M} & [F_\imath^\vee|_{M_\imath}\ar[r]^{ds_\imath} \ar[d]^{s_\imath} & \Omega_{V_\imath}|_{M_\imath}]\ar@{=}[d]\\
\tau^{\ge -1}\bbL_{M_\imath} &[\I_\imath/\I_\imath^2 \ar[r]^d& \Omega_{V_\imath}|_{M_\imath}]
}\eeq
where $\I_\imath$ denotes the ideal of $M_\imath\subset V_\imath$.  On this chart $M_\imath$,
the virtual fundamental class is 
$$[M_\imath]^\vir=s_\imath^![V_\imath] =0^!_{F_\imath|_{M_\imath}}[C_{M_\imath/V_\imath}]\ \ \in \ \ A_*(M_\imath)$$
where 
\beq\label{85} C_{M_\imath/V_\imath}\subset F_\imath|_{M_\imath}\eeq denotes the normal cone of $M_\imath$ in $V_\imath$. 
The virtual fundamental class of $M$ is a class $[M]^\vir\in A_*(M)$ whose restriction to $M_\imath$ is $[M_\imath]^\vir$ for all $\imath$. 

By \cite{BF}, $h^1/h^0(\phi_M^\vee)$ induces an inclusion
\beq\label{2} \mathfrak{C}_M\hookrightarrow h^1/h^0(E_M^\vee)=:\cE_M\eeq
of the intrinsic normal cone $\fC_M$ of $M$ into the vector bundle stack $\cE_M$.
In the open chart \eqref{70}, \eqref{2} is the quotient 
$$\mathfrak{C}_{M_\imath}=C_{M_\imath/V_\imath}/\T_{V_\imath}|_{M_\imath} \hookrightarrow F_\imath|_{M_\imath}/\T_{V_\imath}|_{M_\imath}=\cE_M|_{M_\imath}$$
of the embedding \eqref{85} by the action of $\T_{V_\imath}|_{M_\imath}$.

The \emph{virtual fundamental class} or the \emph{virtual cycle} of $M$ is now defined by 
\beq\label{3} [M]\virt=0^!_{\cE_M}[\fC_M]\ \  \in\  A_r(M), \quad r=\mathrm{rank}(E_M)\eeq 
by \cite{LT, BF} where $0^!_{\cE_M}$ is the Gysin map of the vector bundle stack $\cE_M$ in \cite{Kre}.

\subsection{Virtual pullback}\label{S1.2}
The construction of a virtual fundamental class is relativized by the virtual pullback as follows (cf. \cite{Man}). Let $g:M\to B$ be a morphism of a DM stack $M$ to an Artin stack $B$, equipped with a perfect obstruction theory 
\beq\label{4} \phi_{M/B}:E_{M/B}\lra \bbL_{M/B}\eeq
where $E_{M/B}$ is perfect of amplitude $[-1,0]$. By \cite{BF}, we have a natural embedding
\beq\label{5} \nu:\fC_{M/B}\hookrightarrow h^1/h^0(E^\vee_{M/B})=:\cE_{M/B}\eeq
of the intrinsic normal cone $\fC_{M/B}$ of the map $g$ into the vector bundle stack $\cE_{M/B}$. 
For the vector bundle stack $\cE_{M/B}$, we have the Gysin map (cf. \cite{Kre}) 
\beq\label{6} 
0^!_{\cE_{M/B}}:A_*(\cE_{M/B})\lra A_{*+r}(M).\eeq

On the other hand, the deformation $\bM^\circ_{M/B}$ to the normal cone, constructed in \cite{Ful, Kre} fits into the Cartesian diagram
\beq\label{7}\xymatrix{
\fC_{M/B} \ar@{^(->}[r]^i\ar[d] & \bM^\circ_{M/B} \ar[d] & B\times \bbG_m\ar@{_(->}[l]_j \ar[d] ^q\\
0\ar@{^(->}[r] & \bbA^1 &\bbG_m\ar@{_(->}[l]
}\eeq
which gives us the specialization map 
\beq\label{8} \sp: A_*(B)\lra A_*(\fC_{M/B})\xrightarrow{\nu_*} A_*(\cE_{M/B})\eeq
as the composition 
\beq\label{9}\xymatrix{
A_*(B)\ar[r]^-{q^*} & A_{*+1}(B\times \bbG_m)\ar[dr] & A_{*+1}(\bM^\circ_{M/B})\ar@{->>}[l]_-{j^*}\ar[d]^{i^!}\\ 
&&A_*(\fC_{M/B})\ar[r]^{\nu_*} & A_*(\cE_{M/B})
}\eeq
by using $i^!i_*=0$ and the embedding \eqref{5}. 
Note that the first map in \eqref{8} is classically called the specialization map but its composition with $\nu_*$ is better suited for comparison with the derived approach in \S\ref{ss:1.3}. 

Upon composing \eqref{6} and \eqref{8}, we obtain the \emph{virtual pullback} or the (relative) \emph{virtual cycle}
\beq\label{10} [M/B]\virt=g^!:A_*(B)\lra A_{*+r}(M),\quad r=\mathrm{rank}(E_{M/B})\eeq
in \cite{Man}, which reduces to \eqref{3} when $B$ is a point.

\subsection{Cosection localized virtual fundamental class}\label{S1.3}
Given a perfect obstruction theory \eqref{1} on a DM stack $M$, its \emph{obstruction sheaf} is defined as $Ob_M=h^1(E_M^\vee)$ whose vanishing implies the smoothness of $M$. Let
\beq\label{11} \sigma_M:Ob_M\lra \O_M\eeq
be a \emph{cosection} of the obstruction sheaf. Equivalently, a cosection of $Ob_M$ is a homomorphism
\beq\label{11a} \sigma_M:E_M^\vee[1]\lra \O_M\quad\text{in }\ \  D^b\mathrm{Coh}(M)\eeq
which induces a morphism 
\beq\label{12} h^1/h^0(\sigma):\cE_M\lra \bbA^1.\eeq
Let $\cE_M(\sigma)$ denote the kernel of \eqref{12}. 
In the local chart \eqref{70},  the restriction of \eqref{11} to $M_\imath$ extends to a homomorphism $\sigma_\imath:F_\imath\to \O_{V_\imath}$ so that we have 
a diagram
\beq\label{72} \xymatrix{
& F_\imath\ar[d]\ar[rd]^{\sigma_\imath}\\
M_\imath=s_\imath^{-1}(0)\ar@{^(->}[r]  &V_\imath\ar@/_/[u]_{s_\imath} \ar[r]^{\sigma_\imath\circ s_\imath}&\AA^1
}\eeq
and 
$$Ob_M|_{M_\imath}\cong \mathrm{coker}(ds_\imath:\T_{V_\imath}|_{M_\imath}\to F_\imath|_{M_\imath}), \ \ 
\cE_M(\sigma)|_{M_\imath}=[F_\imath|_{M_\imath}(\sigma)/\T_{V_\imath}|_{M_\imath}]$$
where $F_\imath|_{M_\imath}(\sigma)=\mathrm{ker}(F_\imath|_{M_\imath}\twoheadrightarrow Ob_M|_{M_\imath}\xrightarrow{\sigma} \O_{M_\imath}).$ 


For the cosection localization, it is important to know the relative position of $\mathfrak{C}_M$ and $
\cE_M(\sigma)$ in $\cE_M$.  In nice circumstances, the former is contained in the latter. 
\begin{definition}\label{71}
A cosection $\sigma_M:Ob_M\to \O_M$ is called \emph{closed} if there is an open cover $\{M_\imath\}$ of $M$ with local charts \eqref{72} satisfying $\sigma_\imath \circ s_\imath =0$ for all $\imath$.
\end{definition} 
By the graph construction of the normal cone $C_{M_\imath/V_\imath}$ in \cite{Ful}, if $\sigma_M$ is closed,  since $$\sigma_\imath \circ ts_\imath=t(\sigma_\imath\circ s_\imath)=0\ \ \ \text{as} \ \ t\to \infty,$$ we find that 
$$C_{M_\imath/V_\imath}\subset F_\imath(\sigma_\imath)=\mathrm{ker}(F_\imath\xrightarrow{\sigma_\imath} \O_{M_\imath}).$$
We thus have the inclusion 
$$\cE_M(\sigma)|_{M_\imath}=F_\imath(\sigma_\imath)|_{M_\imath}/\T_{V_\imath}|_{M_\imath} \supset  C_{M_\imath/V_\imath}/\T_{V_\imath}|_{M_\imath} =\mathfrak{C}_{M}|_{M_\imath}$$ for all $\imath$ and obtain the following.  
\begin{lemma}\label{73}
 If $\sigma_M:Ob_M\to \O_M$ is a closed cosection,   
$\mathfrak{C}_M\subset \cE_M(\sigma).$ 
\end{lemma} 
We will see in Remark \ref{Rem:2.8} that in most known examples, the cosections are closed by  Corollary \ref{89}. 

\medskip

Without the closedness, we still have the inclusion of the \emph{reduced support} of the intrinsic normal cone $\mathfrak{C}_M$ into the kernel $\cE_M(\sigma)$ of $\sigma_M$ by the following \emph{cone reduction}. 
\begin{proposition} \label{14a}\cite[Proposition 4.3]{KL} 
Let $M$ be a DM stack equipped with a pefect obstruction theory \eqref{1} and a cosection \eqref{11}. Then the reduced support of $\fC_M$ is contained in $\cE_M(\sigma)$ so that we have a class 
\beq\label{15} [\fC_M]\in A_0(\cE_M(\sigma)).\eeq 
\end{proposition}

\medskip 

Let $M(\sigma)$ denote the \emph{degeneracy locus} of $\sigma_M:Ob_M\to \O_M$, i.e. the closed substack of $M$ defined by the image of $\sigma_M$. In \cite[Proposition 3.3]{KL}, the Gysin map 
$0^!_{\cE_M}$ in \cite{Kre} is localized to a map 
\beq\label{13} 0^!_{\cE_M,\sigma}:A_*(\cE_M(\sigma))\lra A_{*+r}(M(\sigma)), \quad r=\rank \, E_M\eeq
which fits into the commutative diagram
\beq\label{14}\xymatrix{
A_*(\cE_M(\sigma))\ar[r]\ar[d]_{0^!_{\cE_M,\sigma}} & A_*(\cE_M)\ar[d]^{0^!_{\cE_M}}\\
A_{*+r}(M(\sigma)) \ar[r]^{i_*} & A_{*+r}(M)}\eeq
whose horizontal arrows are pushforwards by the inclusion maps. 
We will review the construction of $ 0^!_{\cE_M,\sigma}$ in Construction \ref{Const:2b} below. 


The \emph{cosection localized virtual fundamental class} is now defined as 
\beq\label{16} [M,\sigma]\virt=[M]\virt_\sigma =0^!_{\cE_M,\sigma}[\fC_M]\in A_r(M(\sigma))\eeq
by using \eqref{13} and \eqref{15}. 
By \eqref{14}, it is clear that  $i_*[M,\sigma]^\vir=[M]\virt.$

\subsection{Cosection localized virtual pullback}\label{S1.4} 
In this subsection, we relativize \eqref{16} to construct the cosection localized virtual pullback. Let 
$g:M\to B$ be a morphism of a DM stack to an Artin stack equipped with a perfect obstruction theory \eqref{4} and a cosection 
\beq\label{17}  \sigma_{M/B}:Ob_{M/B}\lra \O_M\eeq
where $Ob_{M/B}=h^1(E^\vee_{M/B}).$ 

As in \S\ref{S1.2}, we have the natural embedding \eqref{5} and as in \S\ref{S1.3}, we let 
$\cE_{M/B}(\sigma)$ denote the kernel of $\sigma_{M/B}$ while $M(\sigma)$ denotes the degeneracy locus of $\sigma_{M/B}$. 
By \cite[Proposition 3.3]{KL} again, the Gysin map 
$0^!_{\cE_{M/B}}$ in \cite{Kre} is localized to a map 
\beq\label{18} 0^!_{\cE_{M/B},\sigma}:A_*(\cE_{M/B}(\sigma))\lra A_{*+r}(M(\sigma))\eeq
which fits into the commutative diagram
\beq\label{19}\xymatrix{
A_*(\cE_{M/B}(\sigma))\ar[r]\ar[d]_{0^!_{\cE_{M/B},\sigma}} & A_*(\cE_{M/B})\ar[d]^{0^!_{\cE_{M/B}}}\\
A_{*+r}(M(\sigma)) \ar[r]^{i_*} & A_{*+r}(M)}\eeq
whose horizontal arrows are pushforwards by the inclusion maps. 

Unlike the absolute case in Proposition \ref{14a}, the cone reduction does not hold in general for the relative case (cf. Example \ref{23}). So we make the following. 
\begin{assumption}[Cone reduction]\label{20}  The reduced support of  $\fC_{M/B}$ is contained in $\cE_{M/B}(\sigma)$ so that we have the pushforward map 
\beq\label{21} \nu_*:A_*(\fC_{M/B})\lra A_*(\cE_{M/B}(\sigma))\eeq
and the specialization map
\beq\label{54} \sp:A_*(B)\lra A_*(\fC_{M/B})\xrightarrow{\nu_*} A_*(\cE_{M/B}(\sigma))\eeq 
which comes from \eqref{9} with $\cE_{M/B}$ replaced by $\cE_{M/B}(\sigma).$
\end{assumption}
\begin{construction}\label{21a} \cite[Definition 2.9]{CKL}
Under Assumption \ref{20}, the \emph{cosection localized virtual pullback} is defined as the composition of \eqref{54} and \eqref{18}:  
\beq\label{22} [M/B,\sigma]^\vir=g^!_\sigma: A_*(B)\xrightarrow{\sp}  A_*(\cE_{M/B}(\sigma))\xrightarrow{0^!_{\cE_{M/B},\sigma}} A_{*+r}(M(\sigma)).
\eeq 
\end{construction}
By definition, when $B$ is a point, $[M/B,\sigma]^\vir=[M,\sigma]^\vir$. 

\medskip

A particularly nice case where Assumption \ref{20} holds is the following. 
\begin{lemma} \cite[Lemma 2.8]{CKL}\label{v3} 
Suppose furthermore that $M$ is equipped with a perfect obstruction theory \eqref{1} that fits into a commutative diagram
\[\xymatrix{
& E_M\ar[d]_{\phi_M} \ar[r]^\eta & E_{M/B}\ar[d]^{\phi_{M/B}}\\
\bbL_B|_M\ar[r] & \bbL_M\ar[r] & \bbL_{M/B}
}\]
whose bottom row is the natural exact triangle of cotangent complexes for the morphism $g:M\to B$. 
If the cosection $\sigma_{M/B}:Ob_{M/B}\to \O_M$ lifts to a cosection 
$\sigma_M:Ob_M\to \O_M$ via the map $$h^1(\eta^\vee):Ob_{M/B}\lra Ob_M,$$ Assumption \ref{20} holds. 
\end{lemma} 
In this case, the cone reduction for $M\to B$ follows from that for $M$.  

\medskip

In general, there is an obstruction to lifting $\sigma_{M/B}$ to $\sigma_M$ as we can see in the example below which arises naturally in the Fan-Jarvis-Ruan-Witten theory (cf. \cite[(4.3)]{KLq}).
\begin{example}\label{23} 
Let $B$ be a smooth variety and $q:V\to B$ be a smooth morphism. Let $F\to V$ be a vector bundle and $s$ be a section of $F$ whose zero locus is denoted by $M$. Suppose we have a regular function $w:B\to \bbA^1$ and a cosection $\sigma:F\to \O_V$ that satisfy 
$$w\circ q=\sigma\circ s:V\lra \bbA^1.$$
In summary, we have a diagram 
\beq\label{24}\xymatrix{
& F\ar[dr]^\sigma\ar[d]\\
M=s^{-1}(0)\ar@{^(->}[r] \ar[dr]_g &V\ar@/_/[u]_s\ar[d]^q\ar[r]^{\sigma\circ s}_{w\circ q} & \bbA^1\\
& B.\ar[ur]_w
}\eeq
Since $B$, $V$ and $q$ are smooth, we have perfect obstruction theories
$$E_M=[F^\vee|_M\xrightarrow{ds} \Omega_V|_M], \quad E_{M/B}=[F^\vee|_M\xrightarrow{ds} \Omega_{V/B}|_{M}]$$
which fit into the morphism of exact triangles 
$$\xymatrix{
\Omega_B|_M\ar[r]\ar@{=}[d] & E_M \ar[r]\ar[d]_{\phi_M} & E_{M/B}\ar[d]^{\phi_{M/B}}\\
\Omega_B|_M\ar[r] & \bbL_M\ar[r] & \bbL_{M/B}.
}$$  
From the long exact sequence of the dual of the top row, we have a commutative diagram 
$$\xymatrix{
\T_B|_M\ar[r]^\delta\ar[dr]_{\sigma_{M/B}\circ \delta} & Ob_{M/B}\ar[d]^{\sigma_{M/B}}\ar[r] & Ob_M\ar@{.>}[dl]^{\sigma_M}\ar[r] & 0\\
&\O_M }$$
where $\sigma_{M/B}$ is induced from $\sigma$. 
As the horizontal row is exact, the cosection $\sigma_{M/B}$ of $Ob_{M/B}$ lifts to a cosection  $\sigma_M$ of $Ob_M$ if and only if $$\sigma_{M/B}\circ \delta=0\in H^0(M,\Omega_B|_M).$$ It is straightforward to see that  
$\sigma_{M/B}\circ \delta$ is in fact the pullback of the 1-form $dw$ which is nonzero for homogeneous polynomials like $\sum_{i=1}^5z_i^5$ on $B=\CC^5$. 
Moreover, the support of the normal cone $C_{M/V}$ of $M$ in $V$ is not contained in the kernel $F(\sigma)$ of $\sigma$. 
 \end{example}

\subsection{Locked cosections and localized virtual pullback}\label{S1.5}  
We will see by Remark \ref{Rem:2.8} and Corollary \ref{89} that in most known examples where we have a cosection $\sigma_{M/B}$, the morphism $g:M\to B$ locally fits into a diagram like \eqref{24} as follows. 
\begin{definition}\label{25}
Let $g:M\to B$ be a morphism of a DM stack $M$ to an Artin stack $B$. Let $\phi_{M/B}:E_{M/B}\to \bbL_{M/B}$ be a perfect obstruction theory. Let $w:B\to \bbA^1$ be a regular function. A cosection $\sigma_{M/B}:Ob_{M/B}=h^1(E^\vee_{M/B})\to \O_M$ is called $w$-\emph{locked} if there is an open cover $\{M_\imath\}$ of $M$ and a commutative diagram
\beq\label{26} \xymatrix{
& F_\imath\ar[dr]^{\sigma_\imath}\ar[d]\\
M_\imath=s_\imath^{-1}(0)\ar@{^(->}[r] \ar[dr]_{g|_{M_\imath}} &V_\imath\ar@/_/[u]_{s_\imath}\ar[d]^{q_\imath}\ar[r]^{\sigma_\imath\circ s_\imath}_{w\circ q_\imath} & \bbA^1\\
& B\ar[ur]_w
}\eeq
such that \begin{enumerate}
\item[(i)] $q_\imath$ is smooth and the perfect obstruction theory $\phi_{M/B}$ restricted to $M_\imath$ is isomorphic to the standard one with  $$E_{M/B}|_{M_\imath}\cong [F_\imath^\vee|_{M_\imath}\xrightarrow{ds_\imath} \Omega_{V_\imath/B}|_{M_\imath}];$$
\item[(ii)] $\sigma_\imath\circ s_\imath=w\circ q_\imath$;
\item[(iii)] $\sigma_{M/B}|_{M_\imath}\circ (F_\imath|_{M_\imath}\twoheadrightarrow Ob_{M/B}|_{M_\imath})=\sigma_\imath|_{M_\imath}.$
\end{enumerate}
A cosection is called \emph{exact} if it is $w$-locked for $w=0:B\to \AA^1$. 
\end{definition}
\begin{remark}
(1) When $B$ is a point, $w=0$ unless $M= \emptyset$ and hence the lockedness (Definition \ref{25}) is the same as the closedness (Definition \ref{71}) as well as the exactness. 

(2) By (ii) above, $g:M\to B$ factors through $B(w)=w^{-1}(0)\subset B$ if $\sigma_{M/B}$ is $w$-locked. 
\end{remark}

\medskip

By pulling \eqref{26} to $B(w)=w^{-1}(0)$, we have a commutative diagram 
\beq\label{27} \xymatrix{
& F_\imath|_{V_\imath(w)}\ar[drr]^{\sigma_\imath}\ar[d]\\
M_\imath=s_\imath^{-1}(0)\ar@{^(->}[r] \ar[dr]_{g|_{M_\imath}} &V_\imath(w)\ar@/_/[u]_{s_\imath}\ar[d]^{q_\imath}\ar[rr]^{\sigma_\imath\circ s_\imath=0} && \bbA^1\\
& B(w)
}\eeq
where $V_\imath(w)=V_\imath\times_BB(w).$ 
As in Lemma \ref{73}, 
we then have the inclusion 
\beq\label{30}\fC_{M/B(w)}\subset \cE_{M/B}(\sigma)\eeq 
of the intrinsic normal cone into the kernel of $\sigma_{M/B}$ in $\cE_{M/B}$. 
Using this, we can now define the cosection localized virtual pullback as follows.
\begin{construction}\label{28}
Let $\sigma_{M/B}$ be a $w$-locked cosection under the hypothesis of Definition \ref{25}. 
Let $\bsig=(\sigma_{M/B},w)$. 
The \emph{localized virtual pullback} is the composition
\beq\label{29} [M/B,\bsig]^\vir=g^!_{\bsig}:A_*(B(w))\xrightarrow{\sp} A_*(\cE_{M/B}(\sigma))\xrightarrow{0^!_{\cE_{M/B},\sigma}} A_{*+r}(M(\sigma))\eeq 
of \eqref{18} and the specialization map \eqref{54} for $M\to B(w)$ using \eqref{30}.  
\end{construction} 
Note that $g^!_{\bsig}=[M/B,\bsig]^\vir$ is not defined over $A_*(B)$ but over $A_*(B(w))$. When $\sigma_{M/B}$ is exact, \eqref{29} is the same as \eqref{22}. 
\begin{remark}\label{31}
In \cite{KLk}, the K-theoretic cosection localization was worked out and in \cite{KP1}, it was further lifted to algebraic cobordism and more. It is straightforward to check that Construction \ref{28} works for these refined homology theories. Moreover, in \cite{Kis, KS}, the cosection localization was extended to the setting of semi-perfect obstruction theory and almost perfect obstruction theory. Construction \ref{28} works for these extended settings as well. 
\end{remark}

\bigskip
\section{Cosection localization, derived} \label{Sec:1}

In this section, we carefully work out the construction of {\em cosection localized virtual cycles} in \S\ref{S1} from the view point of derived algebraic geometry. See \cite{Toen} for a quick introduction to derived algebraic geometry and \cite{Park3} for virtual cycles from the derived perspective. 

In the derived setting, we consider a finitely presented morphism $g:M\to B$ from a derived DM stack $M$ to a derived Artin stack $B$. 
The classical truncation $M_{\cl}$ of $M$ is a classical DM stack and the inclusion map $M_{\cl}\hookrightarrow M$ induces an adjunction homomorphism 
\beq\label{40}\bbL_{M/B}|_{M_{\cl}}\lra \bbL_{M_{\cl}/B_{\cl}}\eeq
of cotangent complexes.
When $g$ is quasi-smooth,
$\bbL_{M/B}|_{M_{\cl}}$ is perfect of amplitude $[-1,0]$ and \eqref{40} is a perfect obstruction theory for the classical truncation of $g$ which gives us the virtual pullback \eqref{10} in \S\ref{S1} (or the virtual fundamental class \eqref{3} when $B$ is a point). 

A cosection of the obstruction sheaf $Ob_{M_{\cl}/B_{\cl}}=h^1(\bbL_{M/B}|_{M_{cl}}^\vee)$ in this case is a map 
\beq\label{78} \sigma_{M/B}:\bbL_{M/B}^\vee|_{M_{\cl}}[1] \lra \O_{M_{\cl}}.\eeq 
For all the moduli spaces where cosections were discovered, $\sigma_{M/B}$ lifts to a homomorphism  
\beq\label{79} \sigma:\bbL_{M/B}^\vee[1]=\TT_{M/B}[1]\lra \O_M\eeq 
which is nothing but a \emph{$(-1)$-shifted 1-form} for $g:M\to B$. 
Conversely, the restriction of a $(-1)$-shifted 1-form \eqref{79} to the classical truncation $M_\cl$ is a cosection \eqref{78}. 

In \S\ref{S1}, we found that the cosection localized virtual pullback requires the cone reduction (Assumption \ref{20}) and Construction \ref{21a} works for a $w$-locked cosection (Definition \ref{25})  where $w$ is a function on $B$.  If we enhance a cosection $\sigma_{M/B}$ to a $(-1)$-shifted 1-form $\sigma$ in the previous paragraph, the notion of a $w$-locked cosection is enhanced to a \emph{$w$-locked $(-1)$-shifted 1-form} $\bsig$ (Definition \ref{41}) which was already introduced in \cite{Park2} to resolve the deformation invariance issue of 4-dimensional Donaldson-Thomas invariants. In Corollary \ref{89}, we show that a $w$-locked $(-1)$-shifted 1-form always gives us a $w$-locked cosection on the classical truncation. 

In \S\ref{ss:1.1} and \S\ref{ss:1.2}, we review the construction of (cosection) localized Gysin map \eqref{13} or \eqref{18} for a perfect complex on a DM stack. In \S\ref{ss:1.3}, we use the deformation space to define the specialization map and the localized virtual pullback for a locked $(-1)$-shifted 1-form. 
This step-by-step construction will be compared with that for the virtual Lagrangian cycle in \S\ref{Sec:2} and it will lead us to the comparison theorem (Theorem \ref{Thm:main}) in \S\ref{Sec:3}. 



\subsection{Localized Gysin maps for vector bundles}\label{ss:1.1}

Let $E$ be a vector bundle of rank $r$ on a classical scheme $M$.
A cosection $\sigma : E \to \O_M$ gives rise to a {\em cosection localized Gysin map}  
\beq\label{44} 0_{E,\sigma}^! : A_*(E(\sigma)) \to A_{*-r}(M(\sigma)),\eeq
in \cite[Cor.~2.9]{KL} where 
\begin{itemize}
\item (kernel cone) $E(\sigma) \subseteq E$ is the zero locus of the linear function $\sigma : E \to \AA^1$, and
\item (degeneracy locus) $M(\sigma)\subseteq M$ is the zero locus of the section $\sigma\dual : M \to E\dual$.
\end{itemize}

The localized Gysin map $0_{E,\sigma}^!$ can be constructed via the {\em blowup} along the degeneracy locus.
Indeed, let $\tM$ be the (classical) blowup of $M$ along $M(\sigma)$ and $D$ be the exceptional divisor.
Then $\sigma|_{\tM} : E|_{\tM} \to \O_{\tM}$ factors through $\O_{\tM}(-D)$,
and its kernel $K:=\Ker(E|_{\tM}\twoheadrightarrow \O_{\tM}(-D))$ is a vector bundle over $\tM$.
%
We can form the {\em generalized blowup squares}
\begin{equation}\label{Eq:2}
\xymatrix{
D \ar@{^{(}->}[r]^{\widetilde{\iota}} \ar[d]^{p_{\sigma}} & \tM \ar[d]^{p} \\
M(\sigma) \ar@{^{(}->}[r]^-{\iota} & M,
}\qquad\xymatrix{
K|_{D} \ar@{^{(}->}[r] \ar[d] & K \ar[d]^{a_{\sigma}} \\
E|_{M(\sigma)} \ar@{^{(}->}[r]^-{b_{\sigma}} & E(\sigma),
}\end{equation}
that is, $\iota, b_\sigma$ are closed embeddings and $p,a_\sigma$ are proper morphisms whose restrictions over the complements of $\iota, b_\sigma$ are isomorphisms.
By \cite[Ex.~1.8.1]{Ful},
we have a right exact sequence
\begin{equation}\label{Eq:1}
\xymatrix@C+2pc{
A_*(K|_{D}) \ar[r] & A_*(K) \oplus A_*(E|_{M(\sigma)}) \ar[r]^-{(a_{\sigma,*},b_{\sigma,*})} & A_*(E(\sigma)) \ar[r]& 0.
}\end{equation}

\begin{construction}
\label{Const:1}
Let $\sigma:E \to \AA^1$ be a cosection of a vector bundle $E$ over a scheme $M$.
We define the {\em localized Gysin map} \eqref{44} 
as the unique map 
satisfying
\[0_{E,\sigma}^! \circ a_{\sigma,*} = - p_{\sigma,*} \circ \widetilde{\iota}^! \circ 0_K^!, \and 0_{E,\sigma}^! \circ b_{\sigma,*} = 0_{E|_{M(\sigma)}}^! .\]
\end{construction}

It is straightforward to check that $0_{E,\sigma}^!$ is well defined via \eqref{Eq:1}.
See \cite{KLk, KLq, KP1} for variations of this construction in K-theory, topology and algebraic cobordism. 


\subsection{Localized Gysin maps for perfect complexes}\label{ss:1.2}

Let $E$ be a perfect complex on a (classical) DM stack $M$.
The {\em total space} of $E$ is the derived stack given by the functor
\[\Tot_M(E) : (T \to M) \mapsto \Map_{\QCoh(T)}(\O_T,E|_T).\]
When $E$ is of tor-amplitude $\geq -1$, $\Tot_M(E)$ is a derived Artin stack of finite presentation \cite[Lem.~3.9]{ToVa}.
By abuse of notation, we use the same symbol $E$ to denote the total space.

Consider a {\em cosection} of a perfect complex $E$, that is, a $\GG_m$-equivariant function
\[\sigma : E \lra \AA^1,\]
with respect to the canonical $\GG_m$-actions (of weight $1$) on $E$ and $\AA^1$ given by 
\[\GG_m \times E \lra E : (t,v) \mapsto t\cdot v,\and \GG_m \times \AA^1 \lra \AA^1 : (t,w) \mapsto t\cdot w.\]
The {\em kernel cone} $E(\sigma)$ and the {\em degeneracy locus} $M(\sigma)$ 
 are the (derived) zero loci
\begin{equation}\label{Eq:6}
\xymatrix{
E(\sigma) \ar[r] \ar[d] \cart & \pt \ar[d]^0 \\
E \ar[r]^-{\sigma} & \AA^1 ,}
\qquad
\xymatrix{
M(\sigma) \ar[r] \ar[d] \cart & M \ar[d]^0 \\
M \ar[r]^{\sigma\dual} & E\dual,
}\end{equation}
where the cosection $\sigma :E \to \AA^1$ is equivalent to a section $\sigma\dual :M \to E\dual$ by \cite[Thm.~2.5]{Mon}.

\begin{remark} \label{43}
For a perfect complex $E$ on a derived stack $M$, we have $$\Map_{\dSt}(E,\AA^1) \not \simeq \Map_{\QCoh(M)}(\O_M,\Sym \,E\dual)$$ in general. However, by \cite[Prop.~2.11]{Mon}, we always have $$\Map_{\dSt}^{\GG_m}(E,\AA^1(-p))  \simeq \Map_{\QCoh(M)}(\O_M,\Sym^pE\dual)$$ where $\AA^1(-p)$ is the weight $p$ representation. This can be extended to (bounded above) quasi-coherent complexes by \cite[Lem.~1.2.2]{Park2}.\end{remark}

Construction \ref{Const:1} is generalized in two steps as follows. 

\subsubsection*{Case 1}

If $M$ is a quasi-projective scheme,
we can extend the localized Gysin map \eqref{44} (Construction \ref{Const:1})
to perfect complexes via {\em resolutions}.
Indeed, when $E$ is of tor-amplitude $[-1,0]$, we can find
\begin{equation}\label{Eq:3}
\text{a linear map $c:\tE \to E$ from a vector bundle $\tE$ such that $\fib(c)$ is a vector bundle.}
\end{equation}
The cosection $\sigma : E \to \AA^1 $ lifts to a cosection $\tsigma:=\sigma \circ c : \tE \to \AA^1$,
and we can form fiber squares between the kernel cones/degeneracy loci,
\begin{equation}\label{Eq:8}
\xymatrix{
\tE(\tsigma) \ar[r] \ar[d]_{c_{\sigma}} \cart & \tE \ar[d]^c\\ 
E(\sigma) \ar[r]^-{} & E,
}\qquad\xymatrix{
M(\sigma) \ar[r] \ar[d]_{} \cart & M \ar[d]^{\mathrm{cl.eq}} \\
M(\tsigma) \ar[r] & \fib(c)\dual[-1]
}\end{equation}
where the last vertical arrow is an isomorphism in classical truncation. 
Since $c_{\sigma}:\tE(\tsigma) \to E(\sigma)$ is an affine bundle
and $M(\sigma)_\cl \simeq M(\tsigma)_\cl$, by \cite[Cor.~2.5.7]{Kre}, we have natural isomorphisms 
\[  c_{\sigma}^*:A_*(E(\sigma)) \xrightarrow{\simeq} A_{*+\widetilde{r}-r}(\tE(\tsigma)) \and  A_*(M(\sigma)) {\cong} A_*(M(\tsigma)),\]
where $r:=\rank(E)$ and $\widetilde{r}:=\rank(\tE)$. 
Using $0_{\tE,\tsigma}^!$ in Construction \ref{Const:1}, we define the localized Gysin map for a perfect complex on a quasi-projective scheme as follows.

\begin{construction} 
\label{Const:2}
Let $\sigma:E \to \AA^1$ be a cosection on a perfect complex $E$ of amplitude $[-1,0]$ over a quasi-projective scheme $M$.
We define the {\em localized Gysin map} as
\[0_{E,\sigma}^! = 0_{\tE,\tsigma}^! \circ c_{\sigma}^* : A_*(E(\sigma)) \lra A_{*-r}(M(\sigma)),\]
for a resolution $c:\tE \to E$ in \eqref{Eq:3}. 
\end{construction}

It is straightforward to check that $0_{E,\sigma}^!$ is independent of the choice of the resolution \eqref{Eq:3} 
(see e.g.~\cite[Prop.~5.6]{KP1}).

\subsubsection*{Case 2} 
If $M$ is a separated DM stack, we can further extend the localized Gysin map (Construction \ref{Const:2}) via the {\em Chow lemma}.
Indeed, we can find
\begin{equation}\label{Eq:Chow}
\text{ a projective surjective map $d:M_1 \twoheadrightarrow M$ from a quasi-projective scheme $M_1$}\end{equation}
by \cite[Thm.~16.6.1]{LM}.
Denote by $(E_1,\sigma_1)$ and $(E_2,\sigma_2)$ the pullbacks of $(E,\sigma)$ to $M_1$ and $M_2:=M_1\times_M M_1$ respectively.
We can form the fiber squares
\[\xymatrix{
E_2(\sigma_2) \ar[r]^{\pr_1} \ar[d]^{\pr_2} \cart & E_1(\sigma_1) \ar[d]^{e_\sigma}\\
E_1(\sigma_1) \ar[r]^-{e_\sigma} & E(\sigma),
}\qquad\xymatrix{
M_2(\sigma_2) \ar[r] \ar[d]_{} \cart & M_1(\sigma_1) \ar[d]^-{d_\sigma} \\
M_1(\sigma_1) \ar[r]^-{d_\sigma} & M(\sigma)
}\]
where $E_2(\sigma_2) \simeq E_1(\sigma_1)\times_{E(\sigma)}E_1(\sigma_1) $.
Then we have an exact sequence
\begin{equation}\label{Eq:Kimura}
\xymatrix@C+2pc{
A_*(E_2(\sigma_2)) \ar[r]^-{\pr_{1,*} - \pr_{2,*}} & A_*(E_1(\sigma_1)) \ar[r]^-{e_{\sigma,*}} & A_*(E(\sigma)) \ar[r]& 0,
}\end{equation}
by \cite[Thm.~A.1.1]{Park0} or \cite{BP}.
Using $0_{E_1,\sigma_1}^!$ in Construction \ref{Const:2}, we define the localized Gysin map for a perfect complex on a DM stack as follows.

\begin{construction} 
\label{Const:2b}
Let $\sigma : E \to \AA^1$ be a cosection of a perfect complex $E$ of amplitude $[-1,0]$ over a DM stack $M$.
We define the {\em localized Gysin map}
\beq\label{46} 0_{E,\sigma}^!:A_*(E(\sigma))\lra A_{*-r}(M(\sigma))\eeq 
as the unique map satisfying:
\[0_{E,\sigma}^! \circ e_{\sigma,*} = d_{\sigma,*} \circ 0_{E_1,\sigma_1}^!,\]
for a projective cover $M_1 \to M$ in \eqref{Eq:Chow}.
\end{construction}

It is straightforward to check that $0_{E,\sigma}^!$ is well defined via \eqref{Eq:Kimura} and is independent of the choice of the projective cover \eqref{Eq:Chow}.
These are direct consequences of the bivariance of the localized Gysin maps in Construction \ref{Const:2} (see e.g. \cite[Prop.~3.3.4]{Park0}).

Construction \ref{Const:2b} extends to derived DM stacks since Chow groups are by definition classical.

\subsection{Localized virtual cycles}\label{ss:1.3}
Having defined the localized Gysin map \eqref{46}, we next construct a derived version of the specialization map \eqref{54} and then the localized virtual pullback by composing the specialization map and the localized Gysin map. As we saw in \S\ref{S1.4}, we need the cone reduction (Assumption \ref{20}) which is automatic for locked cosections. 

Let $g:M \to B$ be a finitely presented morphism from a derived DM stack $M$ to a derived Artin stack $B$.
The {\em deformation space} constructed in \cite{HKR} is the mapping stack
\[\D_{M/B}:=\uMap_{B\times\AA^1}(B\times 0,M \times \AA^1),\]
which is a derived Artin stack by \cite[Thm.~5.1.1]{HP}. 
We have a fiber diagram
\begin{equation}\label{Eq:12}
\xymatrix{
\T_{M/B}[1] \ar@{^{(}->}[r]^-{i} \ar[d] \cart & \D_{M/B} \ar[d] \cart & B\times \GG_m \ar@{_{(}->}[l]_-{j} \ar@{=}[d] \\
B\times 0 \ar@{^{(}->}[r]|-{|} & B\times \AA^1 & B\times \GG_m \ar@{_{(}->}[l]|-{\circ}
}\end{equation}
where $\T_{M/B}[1]=\Tot_M(\TT_{M/B}[1]) \simeq \uMap_B(\AA^1_B[-1],M)$ is the {\em $1$-shifted tangent bundle}.
Moreover, the deformation space has a $\GG_m$-action such that the diagram \eqref{Eq:12} is equivariant.
We note that the restrictions of the $\GG_m$-action to  $\T_{M/B}[1]$ and $B\times\GG_m$ have {\em different} weights;
\begin{align*}
\text{(special fiber) }\quad &\GG_m \times \T_{M/B}[1] \lra \T_{M/B}[1] : (t,v) \mapsto t^{-1}\cdot v ,\\
\text{(generic fiber) }\quad &\GG_m \times (B\times \GG_m) \lra B\times \GG_m : (t,b,w) \mapsto (b,t \cdot w).
\end{align*}

Observe that {\em differential forms} on $g:M \to B$
are equivalent to functions on $\T_{M/B}[1]$ by 
\[\sA^{p}(M/B,-p)= \Map_{\QCoh(M)}(\O_M,\wedge^p\LL_{M/B}[-p]) \simeq \Map^{\GG_m}_{\dSt}(\T_{M/B}[1],\AA^1(p)),\]
in \cite[Prop.~2.11]{Mon},
where $\AA^1(p)$ denotes the affine line with the $\GG_m$-action of weight $(-p)$.
In particular, a $(-1)$-shifted 1-form is a $\GG_m$-equivariant function $\T_{M/B}[1]\to \AA^1(1)$ whose restriction to $M_\cl$ is a cosection of the obstruction sheaf $h^1(\TT_{M/B}|_{M_\cl})$.  
The counterpart of the notion of a locked cosection (Definition \ref{25}) in derived geometry is the following. 
\begin{definition}[{Locked forms, \cite[\S1]{Park2}}] \label{41} 
The space of {\em $(-p)$-shifted locked $p$-forms} is 
\[\sA^{p,\lc}(M/B,-p) := \Map^{\GG_m}(\D_{M/B},\AA^1(p)). 
\]
\end{definition}
When $B$ is a point, the space $\sA^{p,\lc}(M/B,-p)$ coincides with the space $\sA^{p,\cl}(M/B,-p)$ of closed $(-p)$-shifted $p$-forms defined in \cite{PTVV,CPTVV} by \cite[Prop.~6.1.1]{Park2}.

The {\em underlying $p$-form} (resp. {\em underlying function}) of a locked form is the restriction to the special fiber (resp. generic fiber) via \eqref{Eq:12},
\beq\label{45}
\xymatrix@+2pc{
\sA^p(M/B,-p)& \sA^{p,\lc}(M/B,-p) \ar[l]_-{\overline{(-)}:=i^*} \ar[r]^-{[-]:=j^*}& \sA^0(B,0).
}\eeq
More precisely, the underlying function of $\bsig \in \sA^{p,\lc}(M/B,-p)$ is given by the formula
$$j^*(\bsig) \simeq [\bsig] \boxtimes t^{-p} : B \times \GG_m \to \GG_m,$$ 
where $t\in \Gamma(\GG_m,\O_{\GG_m})$ is the coordinate function. 
We say $\bsig$ is \emph{$w$-locked} if $w\simeq[\bsig]:B \to \AA^1$.
In particular, $\bsig$ is called \emph{exact} if $0\simeq [\bsig]:B\to \AA^1$, i.e. $\bsig$ is $0$-locked. 

For any locked $1$-form $\bsig\in \sA^{1,\lc}(M/B,-1)$, the induced cosection on the classical truncation $M_\cl$ is $w$-locked in the sense of Definition \ref{25} with $w=[\bsig]:B \to \AA^1$, by Corollary \ref{89}.

\begin{example} 
Let $g:M \hookrightarrow B$ be a closed embedding of smooth affine schemes. 
Then $\sA^p(M/B,-p) \simeq \I_{M/B}^p$ where $\I_{M/B} \subseteq \O_B$ is the ideal sheaf defining $M$. 
Thus $(-p)$-shifted locked $p$-forms are functions on $B$ that vanish on the $p$-th order infinitesimal thickening of $M$ in $B$.
In particular, $(-1)$-shifted locked $1$-forms are functions on $B$ that vanish on $M$.
\end{example}

\begin{remark}[Alternative description of $(-1)$-shifted locked $1$-forms]\label{Rem:n1}
Let $g:M \to B$ be a finitely presented morphism of derived Artin stacks.
Then we have a canonical equivalence
\[\sA^{1,\lc}(M/B,-1) \xrightarrow{\simeq} \fib\left(\sA^0(B,0) \xrightarrow{g^*} \sA^0(M,0), 0\right).\]
Hence a $(-1)$-shifted locked $1$-form on $g:M \to B$ is equivalent to a pair $(w,h)$ of
\begin{itemize}
\item a function $w: B\to \AA^1$ and
\item a null-homotopy $h:0\xrightarrow{} w \circ g$ between functions on $M$.
\end{itemize}
\end{remark}


\begin{remark}\label{Rem:2.8}
Most of the cosections that arise in enumerative geometry come from $(-1)$-shifted locked $1$-forms.
See Theorem \ref{Prop:GLSMmain} for the cosections in the theory of gauged linear sigma models in \cite{FJR} and 
Proposition \ref{g10} for the cosections on the $p$-field moduli spaces in \cite{CL}.
Moreover, since closed forms are locked when the $B=\pt$ by \cite[Prop.~6.1.1]{Park2}, we have more examples of $(-1)$-shifted locked $1$-forms as follows.
See \cite[Ex.~7.7]{AKLPR} for the cosections in the theory of Seiberg-Witten invariants in \cite{CK}.
See \cite[Ex.~7.8]{AKLPR} for the cosections in the Gromov-Witten theory of surfaces with holomophic $2$-forms in \cite{KL,KT}.
See \cite[Prop.~A.13]{BKP} for the cosections in the theory of Donaldson-Thomas invariants of Calabi-Yau $4$-folds in \cite{KP2,BKP}.
\end{remark}

For a locked form, we now construct the {\em localized specialization map}.


\begin{lemma}[Localized specialization maps]\label{Lem:LocSp}
Let $\bsig$ be a $(-p)$-shifted locked $p$-form on $g:M \to B$. 
Denote by $w=[\bsig]$ and $\sigma=\obsig$.
Then we have the  localized specialization map  
\beq\label{51}\sp_{M/B}^{\bsig}: A_*(B(w))\lra A_*(\T_{M/B}[1](\sigma)) \eeq 
as the unique map that fits into the commutative triangle
\[\xymatrix{
 & A_{*+1}(\D_{M/B}(\bsig)) \ar[ld]_{0^!} \ar@{->>}[rd]^-{1^!} & \\
A_*(\T_{M/B}[1](\sigma)) &&  A_*(B(w)) \ar@{.>}[ll]_{\sp_{M/B}^{\bsig}} ,
}\]
where $\D_{M/B}(\bsig)=\Zero_{\D_{M/B}}(\bsig)$, $\T_{M/B}[1](\sigma)=\Zero_{\T_{M/B}[1]}(\sigma)$ and $B(w)=\Zero_B(w)$ are the zero loci
and $0^!,1^!$ are the Gysin maps of the fibers of $\D_{M/B}(\bsig) \to \AA^1$ over $0,1 \in\AA^1$, respectively. 
\end{lemma}

\begin{proof}
Consider the zero loci of $\bsig$ in the upper row of \eqref{Eq:12},
\begin{equation}\label{Eq:13}
\xymatrix{
\T_{M/B}[1](\sigma) \ar@{^{(}->}[r]^-{i_{\bsig}}   & \D_{M/B}(\bsig)  & B(w)\times \GG_m \ar@{_{(}->}[l]_-{j_{\bsig}} .
}\end{equation}
We then define $\sp_{M/B}^{\bsig}$ as the unique map that fits into the commutative diagram
\[\xymatrix{
A_{*+1}(\T_{M/B}[1](\sigma)) \ar[r]^-{i_{\bsig,*}} \ar[rd]_{c_1(\O)=0} & A_{*+1}(\D_{M/B}(\bsig)) \ar@{->>}[r]^-{j_{\bsig}^*} \ar[d]^-{i_{\bsig}^!} & A_{*+1}(B(w) \times \GG_m) \ar[r]  \ar@{.>}[ld]|{\exists!} & 0 \\
& A_*(\T_{M/B}[1](\sigma)) & \ar@{.>}[l]^-{\sp_{M/B}^{\bsig} } A_*(B(w)) ,\ar[u]_{\pr_1^*}^{\simeq}
}\]
where the upper row is right exact by the excision sequence \cite[Prop.~2.3.6]{Kre},
$i_{\bsig}^! \circ i_{\bsig,*}=c_1(\O_{\T_{M/B}[1](\bsig)})=0$ by the self-intersection formula,
and $\pr_1^*$ is an isomorphism by the excision sequence and the self-intersection formula.
\end{proof}

\begin{remark}[Intrinsic normal cones] \label{53} 
In the situation of Lemma \ref{Lem:LocSp}, assume that $B$ is an equi-dimensional classical scheme and $w \simeq 0$. Then we have
\[\sp_{M/B}^{\bsig}([B]) = [\fC_{M/B}] \in A_*(\T_{M/B}[1](\sigma)),\]
where $\fC_{M/B}$ is the {\em intrinsic normal cone} \cite{BF,Kre}.
Here we automatically have the {\em cone reduction property} (or the {\em isotropic condition}), i.e. $(\fC_{M/B} \hookrightarrow \T_{M/B}[1] \xrightarrow{\overline{\sigma}} \AA^1) =0 $, 
since $\sigma$ extends to $\bsig : \D_{M/B} \to \AA^1$ and $\fC_{M/B}$ is the flat limit of $B$ inside $\D_{M/B}$.
Thus, if we use the notion of locked forms, it is not necessary to describe intrinsic normal cones or to check the cone reduction property.
These are embedded in the localized specialization map. 
\end{remark}

Recall that $g: M \to B$ is {\em quasi-smooth} if $\T_{M/B}[1]$ is of tor-amplitude $\leq 0$.
Composing \eqref{51} and \eqref{46} with $E=\T_{M/B}[1](\sigma)$, we obtain the virtual pullback localized by a locked $(-1)$-shifted 1-form $\bsig$ as follows. 


\begin{construction}[Localized virtual cycles]\label{Const:3}
Let $g:M \to B$ be a quasi-smooth morphism from a derived DM stack $M$ to a derived Artin stack $B$.
Let $\bsig$ be a $(-1)$-shifted locked $1$-form on $g:M \to B$.
We define the {\em localized virtual cycle} (or  the \emph{localized virtual pullback}) as 
\beq\label{55}
[M/B,\bsig]^\vir=g^!_{\bsig} := 0^!_{\T_{M/B}[1],{\sigma}} \circ \sp_{M/B}^{\bsig} : A_*(B(w)) \lra A_{*+r}(M(\sigma)),\eeq 
where 
$B(w)$ is the zero locus of the underlying function $w:=[\bsig] : B \to \AA^1$, 
$M(\sigma)$ is the zero locus of the underlying $1$-form $\sigma:=\overline{\bsig} : M \to \T^*_{M/B}[-1]$,
and $r:=\rank(\TT_{M/B})$.
\end{construction}

The localized virtual cycle $[M/B,\bsig]^\vir$ indeed localizes the ordinary virtual fundamental cycle $[M/B]^\vir=[M/B,0]^\vir$ in the sense that the diagram
\[\xymatrix@C+3pc{
A_*(B(w)) \ar[r]^-{[M/B,\bsig]^\vir} \ar[d] \commutes & A_{*+r}(M(\sigma)) \ar[d] \\
A_*(B) \ar[r]^-{[M/B]^{\vir}} & A_{*+r}(M)
}\]
commutes, where the vertical arrows are the proper pushforwards for the inclusion maps.

Construction \ref{Const:3} is the derived version of the cosection localization. 
See Corollary \ref{89} below for the comparison with Construction \ref{28}.

\subsection{Localized Gysin maps are localized virtual cycles}\label{ss:1.4}
Here we provide a technical lemma claiming that the localized Gysin map \eqref{46} is a particular instance of the localized virtual cycle \eqref{55}. This will be used in \S\ref{Sec:3}.

\begin{lemma}[Locked $1$-forms on zero sections]\label{Lem:sigmazerosection}
Let $\sigma :E \to \AA^1$ be a cosection of a perfect complex $E$ on a derived Artin stack $M$.
Then we have a canonical $(-1)$-shifted locked $1$-form
\[\bsig \in \sA^{1,\lc}(M/E,-1)\]
on the zero section $0_E:M \to E$ such that 
\begin{itemize}
\item (underlying form) $\overline{\bsig}\simeq \sigma : \T_{M/E}[1] \simeq E \lra \AA^1$;
\item (underlying function) $[\bsig] \simeq \sigma : E\lra \AA^1$.
\end{itemize}
\end{lemma}

\begin{proof}
By \cite[Lem.~3.2.2]{Park2}, we have a canonical equivalence of derived stacks 
\begin{equation}\label{Eq:9}
\D_{M/E} \simeq E \times \AA^1,
\end{equation}
where the canonical $\GG_m$-action on the deformation space $\D_{M/E}$ is 
\[\GG_m \times (E\times \AA^1) \lra E\times \AA^1 : (t,v,w) \mapsto (t^{-1}v,tw).\]
We define $\bsig : \D_{M/E} \to \AA^1$ as the composition $\sigma \circ \pr_1 : E\times \AA^1 \to \AA^1$ under \eqref{Eq:9}.    
\end{proof}

\begin{lemma}[Gysin maps are virtual cycles]\label{Lem:LocGysin}
Let $\sigma:E \to \AA^1$ be a cosection on a perfect complex of tor-amplitude $[-1,0]$ over a derived DM stack $M$.
Let $\bsig \in \sA^{1,\lc}(M/E,-1)$ be the induced locked $1$-form in Lemma \ref{Lem:sigmazerosection}.
Then we have the equality 
\begin{equation}\label{Eq:LocGysin}
0_{E,\sigma}^! = [M/E,\bsig]^\vir : A_*(E(\sigma)) \to A_{*-r}(M(\sigma)),
\end{equation}
where the kernel cone and the degeneracy of the cosection $\sigma$ in \eqref{Eq:6}
are equivalent to the zero loci of the underlying function and the underlying $1$-form of $\bsig$ in Construction \ref{Const:3}.
\end{lemma}

\begin{proof}
Since $\bsig:\D_{M/E} \to \AA^1$ is defined as the {\em constant} function $\sigma \circ \pr_1 : E\times \AA^1 \to \AA^1$,
\[\sp^{\bsig}_{M/E}=\id : A_*(E(\sigma)) \to 	A_*(E(\sigma))\]	
by 
Lemma \ref{Lem:LocSp}.
Then \eqref{Eq:LocGysin} follows from the definition of $[M/E,\bsig]^\vir$.
\end{proof}


\subsection{Local structure of locked $1$-forms}\label{42}
In this section, we work out a local model for a morphism $g:M\to B$ equipped with a locked $(-1)$-shifted 1-form.

\begin{lemma}[Locked $1$-forms on zero loci]
Let $q:V \to B$ be a finitely presented morphism of derived Artin stacks,
$E$ be a perfect complex on $V$,
and $s:V \to E$ be a section.
Let $\sigma :E \to \AA^1$ be a cosection  together with an equivalence 
\beq\label{57} h:\sigma \circ s \simeq w\circ q \textin  \sA^0(V,0)\eeq 
for a function $w:B \to \AA^1$.
Then there exists a canonical $w$-locked $(-1)$-shifted $1$-form 
\[\bsig_{\Zero(s)/B} \in \sA^{1,\lc}(\Zero(s)/B,-1)\]
on the zero locus $\Zero(s)$ of $s:V \to E$
such that $\overline{\bsig_{\Zero(s)/B}} \simeq \sigma|_{\Zero(s)} : \T_{\Zero(s)/V}[1] \simeq E|_{\Zero(s)} \to \AA^1$.
\end{lemma}

\begin{proof}
Note that we have a canonical fiber diagram
\begin{equation}\label{Eq:n2}
\xymatrix{
\sA^{1,\lc}(\Zero(s)/B,-1) \ar[r]^-{(-)_{/V}} \ar[d]^{[-]} \cart & \sA^{1,\lc}(\Zero(s)/V,-1) \ar[r] \ar[d]^{[-]} \cart & \pt \ar[d]^0\\
\sA^0(B,0) \ar[r]^{(-)|_V} & \sA^0(V,0) \ar[r]^{(-)|_{\Zero(s)}} & \sA^0(\Zero(s),0)
}\end{equation}
induced by the cofiber sequence
$\O_{\D_{\Zero(s)/-}} \xrightarrow{}  \O_{\D_{\Zero(s)/-}}  \to i_* \O_{\T_{\Zero(s)/-}[1]}$.
Then the $(-1)$-shifted locked 1-form 
 $\bsig_{V/E} \in \sA^{1,\lc}(V/E,-1)$ for the zero section $0:V\to E$  in Lemma \ref{Lem:sigmazerosection} induces
\[s^*\bsig_{V/E} \in \sA^{1,\lc}(\Zero(s)/V,-1) \text{ such that } [s^*\sigma] \simeq s^*[\sigma] \simeq \sigma \circ s \xrightarrow[\simeq]{h} w|_V \text{ in }\sA^0(V,0),\]
which is equivalent to a locked $1$-form $\bsig_{\Zero(s)/B} \in \sA^{1,\lc}(\Zero(s)/B,-1)$.    
\end{proof}

%
%

Conversely, if a quasi-smooth morphism $g:M\to B$ is equipped with a $(-1)$-shifted $w$-locked 1-form $\bsig$, locally it looks like $\Zero(s)\to B$ for a smooth $q:V\to B$ with $E$ a vector bundle admitting a cosection $\sigma$ satisfying \eqref{57}.
\begin{proposition} 
\label{52} 
Let $g: M \to B$ be a quasi-smooth morphism of derived Artin stacks
and $\bsig$ be a $w$-locked $(-1)$-shifted $1$-form for a function $w:B\to\AA^1$.
Then there exist
\begin{itemize}
\item  a smooth morphism $q : V \to B$ from a derived affine scheme $V$,
\item a section $s:V \to F$ of a vector bundle $F$ over $V$,
\item a cosection $\sigma : F \to \AA^1$ together with an equivalence $\sigma \circ s \simeq w\circ q$ in $\sA^0(V,0)$, and
\item a smooth surjective map $c:\Zero(s) \to M$ such that $ c^*(\bsig) \simeq \bsig_{\Zero(s)/B} \in \sA^{1,\lc}(\Zero(s)/B,-1).$
\end{itemize}
Moreover, if $M$ is a derived DM stack, then we can further assume that $c$ is {\'e}tale. 
\end{proposition}

\begin{proof}
We can form a commutative square
\[\xymatrix{
N \ar@{^{(}->}[r]^i \ar[d]_c & V \ar[d]^q \\
M \ar[r]^g & B
}\]
where the vertical arrows $c,q$ are smooth, the horizontal arrow $i$ is a closed embedding, and $N,V$ are derived affine schemes.
By the arguments in \cite[Thm.~7.4.3.18]{LurHA},
we may assume that there is a vector bundle $F$ over $V$ and a section $s:V \to F$ such that $N\simeq \Zero(s)$ is the zero locus.
By \cite[Lem.~3.2.6]{Park2}, there exists a cosection $\sigma:F \to \AA^1$ such that 
\[\ \sA^{1,\lc}(M/B,-1) \xrightarrow{(-)|_{\Zero(s)/V}} \sA^{1,\lc}(\Zero(s)/V,-1) : \bsig \mapsto s^*\bsig_{V/F},\]
where $\bsig_{V/F}\in \sA^{1,\lc}(V/F,-1)$ is the locked $1$-form induced by $\sigma$ via Lemma \ref{Lem:sigmazerosection}.
Then $\sigma \circ s \simeq [s^*\bsig_{V/F}] \simeq [\bsig]|_V$ 
and hence $ c^*(\bsig) \simeq \bsig_{\Zero(s)/B}  \in \sA^{1,\lc}(\Zero(s)/B,-1)$ under \eqref{Eq:n2}.

If $M$ is a derived DM stack, then we can choose $c:N \to M$ to be {\'e}tale.
Indeed, it suffices to show that a smooth morphism of derived affine schemes has an {\'e}tale local section.
For classical schemes, this is shown in \cite[Cor.~17.16.3(ii)]{Gro}.
Then the claim follows from the infinitesimal lifting property of smooth morphisms and the convergence of the Postnikov truncations in \cite[Prop.~7.1.3.19]{LurHA}.
\end{proof}

\begin{corollary} \label{89} 
Let $g:M \to B$ be a quasi-smooth morphism from a derived DM stack $M$ to a derived Artin stack $B$
and $\bsig$ be a locked $(-1)$-shifted $1$-form with $\sigma:=\overline{\bsig}$ and $w:=[\bsig]$.
Then the induced cosection
\[h^1(\sigma|_{M_{\cl}}\dual): Ob_{M_{\cl}/B_{\cl}}:=h^1(\LL_{M/B}|_{M_{\cl}}\dual) \lra \O_{M_{\cl}}\]
is a $w$-locked cosection in the sense of Definition \ref{25}.
Moreover, the localized virtual pullback in Construction \ref{28} coincides with its derived version in Construction \ref{Const:3}.

In particular, if $B$ is a point, the induced cosection is closed in the sense of Definition  \ref{71} and the localized virtual cycle in \eqref{16} coincides with its derived version in Construction \ref{Const:3}. 
\end{corollary}
\begin{proof}
The first statement is a direct consequence of Proposition \ref{52} and the second is immediate from the constructions (cf. Remark \ref{53}).    
\end{proof}

\bigskip
\section{Virtual Lagrangian cycles}\label{Sec:2}
When a DM stack $M$ is equipped with a $(-2)$-shifted symplectic structure $\btheta$, we have the \emph{virtual Langrangian cycle} $[M,\btheta]^\lag$ by the Oh-Thomas construction in \cite{OT}. This construction was generalized to the relative setting of a finitely presented morphism $g:M\to B$ equipped with a $(-2)$-shifted locked symplectic form by the second named author in \cite{Park1}.
In this section, we collect basic properties of the  virtual Lagrangian cycle (or virtual Lagrangian pullback) from \cite{Park1}.

\subsection{Virtual Lagrangian cycles}\label{S2.1}

Let $g:M \to B$ be a finitely presented morphism of derived Artin stacks.

\begin{definition}[Symplectic forms]
A {\em $(-2)$-shifted locked symplectic form} on $g:M \to B$ is a $(-2)$-shifted locked $2$-form $\btheta$ such that the map
\[{\theta}^{\#}:\TT_{M/B}[1] \xrightarrow{\simeq} \LL_{M/B}[-1]\]
induced by the underlying $2$-form $\theta:=\overline{\btheta}  : \O_M \to \wedge^2\LL_{M/B}[-2]$ is an equivalence.
\end{definition}

An {\em orientation} of a $(-2)$-shifted locked symplectic form $\btheta$ on $g:M \to B$ is an equivalence
\[\ori : \O_M \xrightarrow{\simeq} \det(\TT_{M/B}[1])\]
whose square is 
\[\det(\TT_{M/B}[1]) \otimes \det(\TT_{M/B}[1]) \xrightarrow[\simeq]{1\otimes \det({\theta}^{\#})} \det(\TT_{M/B}[1]) \otimes \det(\LL_{M/B}[-1])\xrightarrow[\simeq]{p_{\TT_{M/B}[1]}} \O_M,\]
where $p_{\TT_{M/B}[1]}$ is the pairing in \cite[Eq.~(8), (56)]{OT} for $\TT_{M/B}[1]$.

\begin{remark}[Determinants and sign conventions]
The determinants for classical schemes in \cite{KM} can be extended to derived stacks as in \cite[Def.~3.1]{STV} using \cite[Cor.~8.1.3]{HS}.
Following \cite{KM}, we use the Koszul sign rule for $\det(E)\otimes \det(F) \simeq \det(F) \otimes \det(E)$.
Note that our convention of orientation is slightly different from \cite{OT}. 
After multiplying $(\sqrt{-1})^{\frac{r(r-1)}{2}}$ to our orientation $\ori$, we get an orientation in the sense of \cite[Def.~2.1, Eq.~(58)]{OT}. 
\end{remark}

\begin{theorem}[{Virtual Lagrangian cycles, \cite{OT,Park1,Park2}}] \label{60}
Let $g:M \to B$ be a finitely presented morphism from a derived DM stack $M$ to a derived Artin stack $B$.
Let $\btheta$ be an oriented $(-2)$-shifted locked symplectic form.
Then there exists a map
\beq\label{58}
[M/B,\btheta]^\lag : A_*(B(w)) \lra A_{*+ r}(M),\eeq 
satisfying the bivariance (Proposition \ref{Prop:bivariance}) and functoriality (Proposition \ref{Prop:functoriality}) below.
Here, $B(w)$ is the zero locus of $w:=[\btheta]:B\to \AA^1$
and $r=\frac12\rank(\TT_{M/B})$.
\end{theorem}
When $\btheta$ is obvious from the context, we will write simply $[M/B]^\lag$ instead of  $[M/B,\btheta]^\lag$. 
\begin{remark}\label{65}
(1) In \cite{OT}, the virtual Lagrangian cycle was constructed for the absolute case $B=\mathrm{pt}$. 
In this case, the lockedness  is the same as the closedness (Remark \ref{48}) which is automatic for a symplectic form. Therefore, \eqref{58} generalizes the Oh-Thomas cycle in \cite{OT}. 

(2) The relative virtual Lagrangian cycle or the virtual Lagrangian pullback \eqref{58} does not exist in general if we delete the lockedness condition. 

(3) In \cite{Park1}, a relative construction \eqref{58} was provided and functorial properties were established. In \cite{Park2}, all these were lifted to the derived setting and a uniqueness theorem was proved. 
The construction of $[M/B,\btheta]^\lag$ for general DM stacks $M$ using the Kimura sequence in \cite{BP} is explained in \cite[Ap.~A.2]{Park1}. This generality was not presented in the derived interpretation in \cite{Park2} for simplicity but it works the same.
\end{remark} 

The construction of virtual Lagrangian cycles is quite similar to that of (cosection) localized virtual cycles in \S\ref{Sec:1}.
Starting from the square-root Gysin maps for vector bundles with non-degenerate quadratic functions,
we can extend them to perfect complexes of tor-amplitude $[-1,1]$ using symmetric resolutions,
and the general case follows from the localized specialization maps (Lemma \ref{Lem:LocSp}).
We will not repeat the construction here.
See \cite{Park3} for a quick survey.

\medskip 

In the remainder of this section, we present two basic properties, the bivariance and functoriality, which will be useful in \S\ref{Sec:3}. 
In fact, these two properties uniquely determine the virtual Lagrangian cycles \eqref{58} by \cite[Thm.~C]{Park2}.

\subsection{Bivariance}\label{S2.2}

Consider a fiber square of derived Artin stacks 
\[\xymatrix{
N \ar[r]^-{p_M} \ar[d] \cart & M \ar[d]  \\
U \ar[r]^-{p} & B,
}\]
where the vertical arrows are of finite presentation.
Let $\btheta$ be an oriented $(-2)$-shifted locked symplectic form on $g:M \to B$.
Then we have an induced oriented $(-2)$-shifted locked symplectic form $p^*\btheta$ on $N \to U$ since $\D_{N/U}\simeq \D_{M/B}\times_B U$.

\begin{proposition}[Bivariance]\label{Prop:bivariance}
Suppose that $M$ is  DM and let $\btheta$ be an oriented $(-2)$-shifted locked symplectic form on $g:M\to B$. 
\begin{enumerate}
\item If the classical truncation $p_\cl:U_\cl \to B_\cl$ of $p$ is a projective morphism, then we have
\[[M/B,\btheta]^\lag \circ (p_{w})_* =p_{M,*} \circ [N/U,p^*\btheta]^\lag 
,\] 
where $p_{w} : U(p^*w) \to B(w)$ is the restriction of $p$ to the zero loci of $w:=[\btheta]$. 
\item If $p:U \to B$ is a quasi-smooth morphism (and $p_\cl$ is quasi-projective), then we have
\[p_{M}^! \circ [M/B,\btheta]^\lag  =  [N/U,p^*\btheta]^\lag \circ p_{w}^! 
,\]
where $(-)^!$ denotes the quasi-smooth (virtual) pullback \cite{BF,Man}.
\end{enumerate}
\end{proposition}

Proposition \ref{Prop:bivariance} is a derived version of \cite[Prop.~1.15]{Park1} (see also \cite[Prop.~4.3.10]{Park0} for the DM case).
The statement here is slightly more general since we consider quasi-smooth morphisms instead of lci morphisms, but it also follows from the same arguments.
These can be further extended to the case where the morphism $p:U \to B$ is DM (not necessarily quasi-projective) by using the proper pushforwards for Artin stacks developed in \cite{BSS}. 

\subsection{Functoriality}\label{S2.3}

Consider an oriented locked $(-2)$-shifted {\em Lagrangian correspondence}
\[\xymatrix{
& L \ar[ld]_s \ar[rd]^t & \\
M && N
}\]
over a derived Artin stack $B$,
that is, a correspondence of derived Artin stacks
together with 
\begin{itemize}
\item a $(-2)$-shifted locked sympletic form $\btheta_M \in \sA^{2,\lc}(M/B,-2)$ and an orientation $\ori_M$;
\item a $(-2)$-shifted locked sympletic form $\btheta_N \in \sA^{2,\lc}(M/B,-2)$ and an orientation $\ori_N$;
\item an equivalence $s^*(\btheta_M) \simeq t^*(\btheta_N) \in \sA^{2,\lc}(L/B,-2)$ such that 
\begin{enumerate} 
\item[(i)] the induced square
\[
\xymatrix{
\TT_{L/B}[1] \ar[r]^-{b:=t_*} \ar[d]^{a:=s_*} \cart & \TT_{N/B}[1]|_L \simeq \LL_{N/B}[-1]|_L \ar[d]^{b\dual\simeq t^*} \\
\TT_{M/B}[1]|_L \simeq \LL_{M/B}[-1]|_L \ar[r]^-{a\dual\simeq s^*} & \LL_{L/B}[-1]
}\]
is a fiber square and 
\item[(ii)] $\ori_M|_L$ is sent to $\ori_N|_L$ under the canonical equivalence
\begin{align*}
\det(\TT_{M/B}[1]|_L) &\xrightarrow{\simeq} \det(\TT_{L/B}[1])\otimes \det(\cof(a)) \\
&\xrightarrow{\simeq} \det(\TT_{N/B}[1]|_L)\otimes \det(\fib(b)) \otimes \det(\cof(b\dual))   \xrightarrow[p_{\fib(b)}]{\simeq} \det(\TT_{N/B}[1]|_L),
\end{align*}
where $p_{\fib(b)}$ is the pairing in \cite[Eq.~(56)]{OT} for $\fib(b)$.
\end{enumerate} 
\end{itemize}

\begin{proposition}[Functoriality]\label{Prop:functoriality}
Assume that $M$, $N$ are DM, the classical truncation $t_\cl$ of $t$ is quasi-projective, and $[\btheta_M] \simeq [\btheta_N] : B \to \AA^1$.
If $s_\cl$ is an isomorphism and $t$ is quasi-smooth, then we have
\[[M/B,\btheta_M]^\lag = [L/N]^\vir \circ [N/B,\btheta_N]^\lag 
.\]
\end{proposition}

Proposition \ref{Prop:functoriality} is a derived interpretation of \cite[Thm.~2.2, Thm.~A.4]{Park1}. See \cite[Rem.~5.2.5]{Park2} for the precise connection.
The assumption on $t_\cl$ being quasi-projective is necessary in the proof of functoriality in \cite{Park1}. However the authors expect that such assumptions can be removed using the symplectic deformations in \cite[\S5.1]{Park2}.

We end this section with an immediate consequence of Proposition \ref{Prop:functoriality} which tells us that the virtual Lagrangian cycle is the \emph{virtual cycle of a Lagrangian}. 

\begin{corollary}\label{Cor:1}
In the situation of Proposition \ref{Prop:functoriality},
if $N \simeq B$ so that we have a diagram
\[\xymatrix{ 
L\ar[dr]_{\mathrm{qsm}}\ar[rr]^{\mathrm{Lag}} && M\ar[dl]^{(-2)\mathrm{-{symp}}}\\
&B
}\]
and $\ori_N=1$,
then we have
\[[M/B]^\lag = [L/B]^\vir.\]
%
\end{corollary}

\bigskip
\section{Comparison Theorem}\label{Sec:3} 
In this section, we state and prove the first main result of this paper (Theorem \ref{i6}). 

Let $g:M\to B$ be a quasi-smooth morphism from a derived DM stack and let $\bsig$ be a locked $(-1)$-shifted 1-form with $w=[\bsig]:B\to \AA^1$ and $\sigma=\obsig\in \sA^1(M/B,-1)$, using the notation of \eqref{45}. By Construction \ref{Const:3}, we have the (cosection) localized virtual cycle (pullback)
\beq\label{61} 
[M/B,\bsig]^\vir :A_*(B(w))\lra A_{*+r}(M(\sigma)), \quad r=\rank\, \TT_{M/B}.\eeq
In this section, we first show that the restriction of $g$ to the zero locus $M(\sigma)$ of $\sigma$ admits a canonical $(-2)$-shifted locked symplectic form $\btheta$ with a canonical orientation $\mathrm{or}_{M(\sigma)/B}$ (Proposition \ref{Prop:TwCot}) which gives us the virtual Lagrangian cycle 
\beq\label{62}
[M(\sigma)/B,\btheta]^\lag :A_*(B(w))\lra A_{*+r}(M(\sigma)), \quad r=\frac12\rank\, \TT_{M(\sigma)/B}\eeq 
by Theorem \ref{60}. Hence, there arises a natural question asking how \eqref{61} is related to \eqref{62}.  
Theorem \ref{Thm:main} below tells us that the two pullbacks \eqref{61} and \eqref{62} are in fact identical.

\subsection{Symplectic structures on zero loci and comparison of virtual cycles}\label{Sec:3.1}

By definition, the {\em zero locus} of a $(-1)$-shifted 1-form $\sigma$ on $M$ is the fiber product 
\beq\label{68}\xymatrix{
M(\sigma) \ar[r] \ar[d] \cart & M \ar[d]^0\\
M \ar[r]^-{\sigma} & \T^*_{M/B}[-1],
}\eeq
where $\T^*_{M/B}[-1]=\Tot_M(\LL_{M/B}[-1])$.
We first observe that if $\bsig$ is a locked $1$-form with $\obsig\simeq\sigma$, its zero locus $M(\sigma)$ admits a natural 
$(-2)$-shifted locked symplectic form $\btheta$.

\begin{proposition}
\label{Prop:TwCot}
Let $g:M \to B$ be a finitely presented morphism of derived Artin stacks.
Given 
$\bsig \in \sA^{1,\lc}(M/B,-1)$, there exists a canonical $(-2)$-shifted locked symplectic form
\[\btheta \in \sA^{2,\lc}(M(\sigma)/B,-2)\text{ such that $[\btheta] \simeq [\bsig] : B \lra \AA^1$},\]
where $M(\sigma)$ is the zero locus of $\sigma=\obsig$. 
Moreover,  $\btheta$ admits a canonical orientation $\ori_{M(\sigma)/B}$.
\end{proposition}

In the literature, $M(\sigma)$ equipped with $\btheta$ is called the {\em $\bsig$-twisted $(-2)$-shifted cotangent bundle} (see e.g.~\cite{BG} or \cite[Def.~1.16]{Saf}).
It is well known that the twisted cotangent bundle is shifted symplectic by \cite[Thm.~2.22]{Cal}.
Moreover, there is a straightforward extension to the locked case in \cite[\S3.1]{Park2}.
We briefly recall the proof here.

\begin{proof}[Proof of Proposition \ref{Prop:TwCot}]

The commutative square \eqref{68} gives us a null-homotopy
\[h_{\bsig} : 0 \xrightarrow{\simeq} \sigma|_{M(\sigma)} \simeq \overline{(\bsig|_{M(\sigma)})} \textin \sA^{1}(M(\sigma)/B,-1).\]

We will use the fiber square of spaces
\[\xymatrix{
\sA^{2,\lc}(M(\sigma)/B,-2) \ar[r]^-{(-)^{\geq 1}} \ar[d] 
\cart & \sA^{1,\lc}(M(\sigma)/B,-1) \ar[d]^{\overline{(-)}} 
\\
\pt \ar[r]^-0 & \sA^1(M(\sigma)/B,-1),
}\]
induced by the canonical cofiber sequence
\[\O_{\D_{M(\sigma)/B}} \xrightarrow{t\cdot}  \O_{\D_{M(\sigma)/B}}  \to i_* \O_{\T_{M(\sigma)/B}[1]}, \]
where $t:\D_{M(\sigma)/B} \to B\times \AA^1 \xrightarrow{\pr_2} \AA^1$ is the composition and $i :\T_{M(\sigma)/B}[1] \hookrightarrow \D_{M(\sigma)/B}$ is the inclusion. 
The null-homotopy $h_{\bsig}$ is equivalent to a locked $2$-form
$\btheta \in \sA^{2,\lc}(M(\sigma)/B,-2)$
with $\btheta^{\geq1} \simeq \bsig|_{M(\sigma)}$ via the fiber square.
We can easily observe that $[\btheta] \simeq [\bsig] : B \to \AA^1$.

The nondegeneracy of $\btheta$ follows from the Lagrangian intersection theorem \cite[Thm.~2.9]{PTVV} 
since
the $(-1)$-shifted cotangent bundle $\T^*_{M/B}[-1]$ is $(-1)$-shifted symplectic over $B$ by  \cite[Thm.~2.4]{Cal}
and the sections $\bsig,0 : M \to \T^*_{M/B}[-1]$ are Lagrangians by \cite[Thm.~2.22]{Cal}.

The orientability follows from the canonical self-dual cofiber sequence
\[\xymatrix{
\LL_{M/B}[-1]|_{M(\sigma)} \ar[r] & \LL_{M(\sigma)/B}[-1]\simeq\TT_{M(\sigma)/B}[1] \ar[r] & \TT_{M/B}[1]|_{M(\sigma)},
}\]
induced by the Lagrangian fibration structure of $M(\sigma) \to M$ (cf. \cite[Thm.~3.5]{Gra}). 
We choose the orientation $\ori_{M(\sigma)/B}$ as the inverse of the induced isomorphism
\[\det(\TT_{M(\sigma)/B}[1]) \simeq \det(\TT_{M/B}[1]|_{M(\sigma)}) \otimes \det(\TT_{M/B}[1]|_{M(\sigma)}\dual) \xrightarrow[p_{\TT_{M/B}[1]}]{\simeq}\O_{\TT_{M(\sigma)/B}[1]},\]
where $p_{\TT_{M/B}[1]}$ is the pairing in \cite[Eq.~(56)]{OT} for $\TT_{M/B}[1]$.
\end{proof}

From now on, we will sometimes write 
\beq\label{911} M(\bsig)=(M(\sigma),\btheta)\eeq  
which is an oriented locked $(-2)$-shifted symplectic fibration over $B$.

\medskip

We can now state the main theorem of this section. 


\begin{theorem}\label{Thm:main}
Let $g:M \to B$ be a quasi-smooth morphism from a derived DM stack to a derived Artin stack.
Let $\bsig$ be a $(-1)$-shifted locked $1$-form on $g: M \to B$ with $w:=[\bsig]$ and $\sigma:=\obsig$. Let $\btheta$ be the $(-2)$-shifted locked symplectic form on $M(\sigma)\to B$ with the canonical orientation $\ori_{M(\sigma)/B}$ by Proposition \ref{Prop:TwCot}. 
Then we have the identity 
\begin{equation}\label{Eq:Main}
[M/B,\bsig]^\vir = [M(\sigma)/B, \btheta]^\lag : A_*(B(w)) \lra A_{*+r}(M(\sigma)),\quad r=\rank\, \TT_{M/B}
\end{equation}
of the localized virtual cycle from Construction \ref{Const:3} and the virtual Lagrangian cycle from Theorem \ref{60}.  
\end{theorem}


\begin{remark}[Quotient stacks] 
Theorem \ref{Thm:main} also works for some classes of Artin stacks.
Given a quasi-smooth morphism $g:M \to B$ of derived Artin stacks and a $(-1)$-shifted locked $1$-form $\bsig \in \sA^{1,\lc}(M/B,-1)$,
instead of assuming that $M$ is a derived DM stack,
assume that 
\begin{itemize}
\item [A1)] $M_\cl$ is a {\em global quotient stack}, i.e. $M_\cl$ is the quotient stack $[P/G]$ of a quasi-projective scheme $P$ by a linear action of a linear algebraic group $G$;
\item [A2)] $g:M \to B$ is a DM type 
morphism, i.e. $\LL_{M/B}$ is of amplitude $\leq 0$.
\end{itemize}
Then the results in \S\ref{Sec:1} and \S\ref{Sec:2} extend so that we still have the virtual cycles $[M/B,\bsig]^\vir$ and $[M(\sigma)/B]^\lag$ as well as the formula \eqref{Eq:Main}.
Indeed, these results follow from the following two facts:
\begin{enumerate}
\item \cite[Lem.~2.6]{Tho} $M_\cl$ has the {\em resolution property}, i.e. all perfect complexes are quasi-isomorphic to bounded chain complexes of vector bundles;
\item \cite{Tot} for each integer $k$, there exists a smooth morphism $a:\tM \to M$ from a quasi-projective scheme $\tM$ such that $a^*:A_k(M) \to A_{k+\dim(a)}(\tM)$ is an isomorphism. 
\end{enumerate}

\end{remark}

\begin{remark}[Locked versus~Closed]\label{48} 
The locked forms are {\em closed} in the sense of \cite{PTVV},
but not all closed forms are locked.
Nevertheless, in the absolute case (i.e. $B=\Spec\,\C$),  
the lockedness coincides with the closedness as well as the exactness by \cite[Prop.~6.1.1]{Park2}. 
\end{remark}

By Remark \ref{48} and Theorem \ref{Thm:main}, we obtain the following comparison result which was announced in \cite{KP3}.  
\begin{corollary}\label{67} 
Let $M$ be a quasi-smooth derived DM stack equipped with a $(-1)$-shifted \emph{closed} 1-form $\bsig$. Let $\btheta$ be the induced $(-2)$-shifted symplectic form on the twisted $(-2)$-shifted cotangent bundle $M(\sigma)$ defined by the fiber square \eqref{68}.  
Then we have the equality 
\beq\label{66}
[M,\bsig]^\vir=[M(\sigma),\btheta]^\lag\ \ \ \in \ \  A_r(M(\sigma)), \quad r=\rank\, \TT_M\eeq
of the (cosection) localized virtual cycle in \cite{KL} and the virtual Lagrangian cycle in \cite{OT}. 
\end{corollary}


\medskip

The remainder of this section is devoted to a proof of Theorem \ref{Thm:main}.
Our strategy is as follows. 
Recall from \S\ref{Sec:1} that the localized virtual cycles are constructed through four steps
\[\text{Construction \ref{Const:1}} \leadsto \text{Construction \ref{Const:2}}\leadsto \text{Construction \ref{Const:2b}}\leadsto\text{Construction \ref{Const:3}},\]
gradually increasing the generality of the situation.
For each step, there are formulas that uniquely determine the localized virtual cycles.
We will conversely go through these steps and show that the virtual Lagrangian cycles satisfy analogous formulas as consequences of the bivariance (Proposition \ref{Prop:bivariance}) and the functoriality (Proposition \ref{Prop:functoriality}).
This will prove that the two virtual cycles are equal.

\subsection{Deformation to normal cones}

The key step in proving Theorem \ref{Thm:main} is to reduce the statement to the ``linear'' case via deformation to the normal cone
\[(M \xrightarrow{g} B) \leadsto (M \xrightarrow{0} \T_{M/B}[1]).\]

Recall that a locked form on $g:M \to B$ induces a locked form on $0:M \to \T_{M/B}[1]$,
\[\bsig \in \sA^{1,\lc}(M/B,-1) \mapsto \cS(\bsig) \in \sA^{1,\lc}(M/\T_{M/B}[1],-1),\]
as follows. 
The underlying $1$-form $\overline{\bsig} : \T_{M/B}[1] \to \AA^1$ is a cosection on $\T_{M/B}[1]$,
and it induces a locked form $\cS(\bsig)$ on the zero section $0:M \to \T_{M/B}[1]$ by Lemma \ref{Lem:sigmazerosection}.

We also recall from Construction \ref{Const:3} that the localized virtual cycle is defined as:
\[[M/B,\bsig]^\vir = [M/\T_{M/B}[1],\cS(\bsig)]^\vir \circ \sp_{M/B}^{\bsig},\]
where we replaced the localized Gysin map by the localized virtual cycle via Lemma \ref{Lem:LocGysin}.

The virtual Lagrangian cycles satisfy an analogous formula.
\begin{lemma}\label{Lem:1}
In the situation of Theorem \ref{Thm:main}, we have
\[[M(\bsig)/B]^\lag = [M({\cS(\bsig)})/\T_{M/B}[1]]^\lag \circ \sp_{M/B}^{\bsig},\]
where $M(\obsig) \simeq M(\overline{\cS(\bsig)})$ as derived DM stacks (without the symplectic forms).
\end{lemma}

We will see that the composition
\[M(\sigma) \times \AA^1 \lra M \times \AA^1 \lra \D_{M/B}\]
has a canonical $(-2)$-shifted locked symplectic form.
Then Lemma \ref{Lem:1} will follow from the bivariance of virtual Lagrangian cycles (Proposition \ref{Prop:bivariance}).

\begin{proof}[Proof of Lemma \ref{Lem:1}]
We note that the locked form $\bsig \in \sA^{1,\lc}(M/B,-1)$ induces a locked form
\[\cD(\bsig) \in \sA^{1,\lc}(M\times \AA^1/\D_{M/B},-1)\]
on the canonical map $M \times \AA^1 \to \D_{M/B}$, 
such that the fibers over $t\in \AA^1$ are
\begin{equation}\label{Eq:3.2.6}
\cD(\bsig)|_{t} \simeq \begin{cases}
	t^{-1} \cdot\bsig \in \sA^{1,\lc}(M/B,-1) & \text{for $t\neq 0$}, \\ 
	\cS(\bsig) \in \sA^{1,\lc}(M/\T_{M/B}[1],-1) & \text{for $t = 0$},
\end{cases}
\end{equation}
and the underlying form and function are 
\begin{align}
\overline{\cD(\bsig)} &\simeq \overline{\bsig} \circ \pr_1 : \T_{M\times \AA^1/\D_{M/B}}[1] \simeq \T_{M/B}[1]\times \AA^1 \to \AA^1, \label{Eq:3.2.7}\\
[\cD(\bsig)] & \simeq \bsig : \D_{M/B} \to \AA^1. \label{Eq:3.2.8}
\end{align}
Indeed, we have a canonical fiber square of derived stacks (cf. \cite[Lem.~5.1.3]{Park2})
\[\xymatrix@C+4pc{
\D_{M\times\AA^1/\D_{M/B}} \ar@{.>}[r]^{\cD} \ar[d] \cart & \D_{M/B} \ar[d] \\
B\times \AA^1 \times \AA^1 \ar[r]^{(b,x,y) \mapsto (b,xy)} & B \times \AA^1,
}\]
and we define $\cD(\bsig) :=\bsig \circ \cD: \D_{M\times\AA^1/\D_{M/B}} \to \D_{M/B} \to \AA^1$.
The equivalences \eqref{Eq:3.2.6}, \eqref{Eq:3.2.7} and \eqref{Eq:3.2.8} follow by computing the fibers.

We can then form a locked $(-2)$-shifted symplectic fibration 
\[(M\times \AA^1)(\overline{\cD(\bsig)}) \lra \D_{M/B}\]
by Proposition \ref{Prop:TwCot}.
The fibers over $t\in \AA^1$ are
\[\begin{cases}
	M(\sigma) \to B & \text{for $t\neq 0$}, \\ 
	M(\overline{\cS(\bsig)}) \to \T_{M/B}[1] & \text{for $t = 0$},
\end{cases}\]
as locked symplectic fibrations by \eqref{Eq:3.2.6} since $\btheta$ in Proposition \ref{Prop:TwCot} is stable under base change.
Moreover, $M(\overline{\cS(\bsig)}) \simeq M(\sigma) \times \AA^1$ as derived schemes by \eqref{Eq:3.2.7}.

By the bivariance of virtual Lagrangian cycles (Proposition \ref{Prop:bivariance}), the square in
\[\xymatrix@C+6pc{
 A_{*+1}(\D_{M/B}(\bsig)) \ar[d]_{0^!} \ar@{->>}[r]^-{1^!} & A_*(B(w)) \ar@{.>}[ld]|{\sp_{M/B}^{\bsig}} \ar[d]^{[M(\bsig)/B]^\lag} \\
A_*(\T_{M/B}[1](\sigma)) \ar[r]_-{[M({\cS(\bsig)})/\T_{M/B}[1]]^\lag} & A_{*+r} (M(\sigma))   ,
}\]
commutes, where $w:=[\bsig]$, $\D_{M/B}(\cD(\bsig)) \simeq \D_{M/B}(\bsig)$ by \eqref{Eq:3.2.8} and $0^!=1^! : A_*(M(\sigma)\times \AA^1) \to A_*(M(\sigma))$ by the homotopy property of Chow groups.
Since the upper left triangle commutes (by Lemma \ref{Lem:LocSp}) and the upper horizontal arrow is surjective, the lower right triangle also commutes. This proves the desired equality.
\end{proof}

\begin{remark}[Classical versus Derived]
Lemma \ref{Lem:1} is closely related to the rational equivalence 
\[[\fC_{M(\sigma)/B}] \simeq [\fC_{M(\sigma)/\fC_{M/B}}] \in A_{\dim\,B}(\fC_{M(\sigma)\times \AA^1/\bM^{\circ}_{M/B}})\]
of the classical intrinsic normal cones in \cite{KKP}, 
where $\bM^{\circ}_{M/B}$ is the classical deformation space \cite{Ful,Kre}.
Here we assumed that $B$ is equi-dimensional for simplicity. 
To show a classical version of Lemma \ref{Lem:1} using symmetric obstruction theories and the above rational equivalence,
we need some isotropic condition for $\fC_{M(\sigma)\times \AA^1/\bM^{\circ}_{M/B}}$
which seems difficult to achieve classically since the obstruction theory for $M(\sigma)_\cl \times \AA^1 \to \bM^\circ_{M/B}$ in \cite{KKP} is {\em not} canonical. 
In fact, it is given by a map between the cones in the triangulated category which is {\em not} unique (cf. \cite[Prop.~2.7]{Park1} or \cite[Rem.~2.3.21]{Park0}).

This subtlety is minor for the virtual fundamental cycles in \cite{LT,BF} for $2$-term perfect obstruction theories $\phi_{M/B}:{E}_{M/B} \to \LL_{M/B}$ since they do not depend on the maps $\phi_{M/B}$, but only on the $K$-theory classes of the complexes ${E}_{M/B}$. 
However, for virtual Lagrangian cycles in \cite{OT} or \cite{Park1}, the necessary isotropic condition {\em depends} on the maps $\phi_{M/B} : {E}_{M/B} \to \LL_{M/B}$ as well.
This is the main reason why we use the {\em derived} deformation space $\D_{M/B}$ in this paper,
which gives us a {\em canonical} obstruction theory. It is not clear how to get such a canonical theory  from classical deformation spaces and triangulated categories.
\end{remark}

Note that Lemma \ref{Lem:1} reduces the proof of 
Theorem \ref{Thm:main} to the case of the zero section $0:M\to E$ of a perfect complex $E$ of tor-amplitude $[-1,0]$ on a derived DM stack $M$ with the locked $1$-form $\bsig$ given by a cosection $\sigma:E\to \AA^1$.

\subsection{Projective covers by quasi-projective schemes}

The second step is to reduce the statement for DM stacks to quasi-projective schemes via projective covers.

Let $E$ be a perfect complex of tor-amplitude $[-1,0]$ over a DM stack $M$ and $\sigma : E\to \AA^1$ be a cosection.
By the Chow lemma \cite[Thm.~16.6.1]{LM}, we can choose a projective surjective map $d:M_1 \twoheadrightarrow M$ from a quasi-projective scheme as in \eqref{Eq:Chow}.
Denote by $E_1$ and $\sigma_1$ the pullbacks of $E$ and $\sigma$ to $M_1$ respectively.
Denote by $d_\sigma : M_1(\sigma_1) \to M(\sigma)$ and $e_\sigma : E_1(\sigma_1) \to E(\sigma)$ the pullbacks of $d$.

Recall from Construction \ref{Const:2b} that the localized virtual cycle is determined by the formula
\[[M/E,\bsig]^\vir \circ e_{\sigma,*}  = d_{\sigma,*} \circ [M_1/E_1,\bsig_1 ]^\vir,\]
where we replaced the localized Gysin maps by the localized virtual cycles (via Lemma \ref{Lem:LocGysin}) and $\bsig$, $\bsig_1$ are the $(-1)$-shifted locked 1-forms induced by $\sigma,\sigma_1$ (via Lemma \ref{Lem:sigmazerosection}).

We observe that the virtual Lagrangian cycles satisfy an analogous formula
\[[M(\bsig)/E]^\lag \circ e_{\sigma,*} = d_{\sigma,*} \circ  [M_1(\bsig_1)/E_1]^\lag.\]
Indeed, this is a direct consequence of the bivariance (Proposition \ref{Prop:bivariance}).

Consequently, we may assume that $M$ is a quasi-projective scheme.

\subsection{Resolution by vector bundles}

The next step is to reduce the statement for perfect complexes to vector bundles via resolutions.

Let $E$ be a perfect complex of tor-amplitude $[-1,0]$ on a quasi-projective scheme $M$.
Let $\sigma : E \to \AA^1$ be a cosection.
Choose a resolution $c:\tE \to E$ by a vector bundle $\tE$ as in \eqref{Eq:3}.
Let $\tsigma=\sigma \circ c: \tE \to \AA^1$ and
denote by $c_{\sigma}: \tE(\tsigma) \to E(\sigma)$ the restriction of $c$.

Recall from Construction \ref{Const:2} that the localized virtual cycle in this case is defined as:
\[\big[M/E,\bsig \big]^\vir = \big[M/\tE,\tbsig\big]^\vir \circ c_{\sigma}^*\]
where we replaced the localized Gysin maps by the localized virtual cycles (via Lemma \ref{Lem:LocGysin})
and $\bsig$ and $\tbsig$ are the $(-1)$-shifted locked $1$-forms induced by $\sigma$ and $\tsigma$ (via Lemma \ref{Lem:sigmazerosection}).

The virtual Lagrangian cycles satisfy an analogous formula.

\begin{lemma}\label{Lem:2}
With the above notation,
we have
\[\big[M(\bsig)/E\big]^\lag = \big[M(\tbsig)/\tE\big]^\lag \circ c_{\sigma}^*.\]
\end{lemma}

\begin{proof} 
The lemma is obtained by applying the functoriality of virtual Lagrangian cycles (Proposition \ref{Prop:functoriality}) as follows.

Observe from the construction in \S\ref{ss:1.4} that the underlying classical scheme of $M(\sigma)$ is the degeneracy locus of $\sigma$ defined in \S\ref{ss:1.2}. 
By Proposition \ref{Prop:TwCot}, $M(\sigma)$ is equipped with the oriented $(-2)$-shifted locked symplectic form $\btheta$ over $E$.

%
By \cite[Rem.~3.1.5(b)]{Park2},
we have a canonical Lagrangian correspondence
\[\xymatrix{
& M(\bsig) \ar[ld]_s^{\mathrm{qsm}} \ar[rd]^t_{\mathrm{cl.eq}} & \\
c^*M(\bsig) && M(\tbsig)
}\]
of $(-2)$-shifted locked symplectic fibrations over $\tE$,
where $c^*M(\bsig) \to \tE$ is the pullback of $M(\bsig) \to E$ along $c:\tE \to E$.
Then $s:M(\sigma) \to c^*M(\sigma)$ is quasi-smooth since it is a section of the smoooth morphism $c_M:c^*M(\sigma) \to M(\sigma)$.
Also $t_\cl:M(\sigma)_\cl \to M(\tsigma)_\cl$ is an isomorphism as observed in \eqref{Eq:8}.
Hence the functoriality gives us
\[[M(\tbsig)/\tE]^\lag = s^! \circ [c^*M(\bsig)/\tE]^\lag.\]
Taking the compositions with $c_{\sigma}^*$, we obtain the desired formula
\[[M(\tbsig)/\tE]^\lag \circ c_{\sigma}^* = s^! \circ [c^*M(\bsig)/\tE]^\lag \circ c_{\sigma}^* = s^! \circ c_M^* \circ [M(\bsig)/E]^\lag = [M(\bsig)/E]^\lag,\]
by the bivariance (Proposition \ref{Prop:bivariance}) and the functoriality of quasi-smooth pullbacks.
\end{proof}

Consequently, it suffices to prove Theorem \ref{Thm:main} for the zero sections of vector bundles.

\subsection{Blowing up the degeneracy loci}

Let $E$ be a vector bundle on a scheme $M$ and $\sigma:E\to \AA^1$ be a cosection.
Let $\tM$ be the classical blowup of $M$ along $M(\sigma)_\cl$,
$D$ be the exceptional divisor
and $K:=\Ker(E|_{\tM} \twoheadrightarrow \O_{\tM}(-D))$.
We will use the notation in \S\ref{ss:1.1}.

Recall from Construction \ref{Const:1} that the localized virtual cycle is defined via the formulas
\[[M/E,\bsig]^\vir \circ a_{\sigma,*} = -p_{\sigma,*} \circ [D/K]^\vir \and [M/E,\bsig]^\vir \circ b_{\sigma,*} = [M(\sigma)/E|_{M(\sigma)}]^\vir,\]
where 
we replaced the Gysin maps for regular embeddings by the virtual fundamental cycles,
the localized Gysin map by the localized virtual cycle via Lemma \ref{Lem:LocGysin},
and $\bsig$ is the locked $1$-form induced by $\sigma$ via Lemma \ref{Lem:sigmazerosection}.

The virtual Lagrangian cycle satisfies an analogous formula.
\begin{lemma} \label{Lem:3}
With the above notation, we have
\[[M(\bsig)/E]^\lag \circ a_{\sigma,*} = -p_{\sigma,*} \circ [D/K]^\vir, \and
[M(\bsig)/E]^\lag \circ b_{\sigma,*} = [M(\sigma)/E|_{M(\sigma)}]^\vir. \]
\end{lemma}

\begin{proof} 
The lemma was essentially proved in \cite[Lem.~5.5]{KP2}.
Here we provide a proof using the properties of virtual Lagrangians cycles in Corollary \ref{Cor:1} and Proposition \ref{Prop:bivariance}.


Form a commutative diagram
\beq\label{76}\xymatrix@C+1pc{
D \ar[r]^-{p_{\sigma}} \ar[d]_{0_K \circ \widetilde{\iota}} & M(\bsig) \ar[r]^-{\iota} \ar[d]_{0_E \circ \iota} \cart & M \ar[d]^{0_{E\oplus E\dual}} \\
K \ar[r]^a & E \ar[r]^-{(1,\sigma)} & E\oplus E\dual,
}\eeq
where $a: K \hookrightarrow E|_{\tM} \to E$ is the composition.
We claim that the left square is a relative {\em Lagrangian square},
i.e. the induced map
\[D \to a^*M(\bsig)\]
is a Lagrangian of the pullback symplectic fibration $a^*M(\bsig) \to K$ of $M(\bsig) \to E$ along $a$.
Note that the right square is a symplectic {\em pullback square}, i.e. 
the induced map
\[M(\bsig) \xrightarrow{\simeq} (1,\sigma)^* (0_{E\oplus E\dual},\mathbf{q})\]
is an equivalence of symplectic fibrations over $E$ by \cite[Prop.~3.2.4, Eq.~(21)]{Park2}.
Here the zero section $0_{E\oplus E\dual} : M \to E\oplus E\dual$ is equipped with the $(-2)$-shifted locked symplectic form $\mathbf{q}$ given by the quadratic pairing $q:E \oplus E\dual \to \AA^1$.
Hence it suffices to show that the total square is a Lagrangian square.

Form another commutative diagram whose total square is the same as \eqref{76} 
\[\xymatrix@C+1pc{
D \ar[r]^-{\widetilde{\iota}} \ar[d]_{0_K \circ \widetilde{\iota}} \cart & \tM \ar[r]^-{p} \ar[d]^{(0,s)} & M \ar[d]^{0_{E\oplus E\dual}} \\
K \ar[r]^-{(1,s)} & K \oplus L \ar[r]^-{(a,e)} & E \oplus E \dual.
}\]
Here $L:=\O_{\tM}(D)$, $s\in \Gamma(\tM,L)$ is the section of the divisor $D\subseteq \tM$, and $e:L \hookrightarrow E|_{\tM}\dual \to E\dual$ is the composition.
The right square is a Lagrangian square by \cite[Eq.~(22)]{Park2} since 
\begin{equation}\label{Eq:3.4.1}
(K\oplus L) \subseteq (E\oplus E\dual)|_{\tM}
\end{equation}
is a maximal isotropic subbundle.
The left square is a pullback square.
Hence the total square is also a Lagrangian square, which proves the claim.

Since the classical truncation of the Lagrangian $D \to a^*M(\bsig)$ is an isomorphism,
\[[a^*M(\sigma)/K]^\lag = -[D/K]^\vir\]
by Corollary \ref{Cor:1}.
Here the sign $(-1)$ is given by the fact that \eqref{Eq:3.4.1} is a {\em negative} maximal isotropic subbundle (for the orientation on $E\oplus E\dual$ making $E\oplus 0$ positive).
We then have
\[[M(\bsig)/E]^\lag \circ a_{\sigma,*} = p_{\sigma,*}\circ [a^*M(\bsig)/K]^\lag = -p_{\sigma,*}\circ [D/K]^\vir\]
by Proposition \ref{Prop:bivariance}.

Similarly, we can form commutative diagrams having the same total square
\[\xymatrix@C+1pc{
M(\bsig) \ar@{=}[r] \ar[d]_{0_{E|_{M(\sigma)}}} & M(\bsig) \ar[r]^-{\iota} \ar[d] |{0_E \circ \iota} \cart & M \ar[d]^{0_{E\oplus E\dual}}  \\
E|_{M(\sigma)} \ar[r]^-{b} & E \ar[r]^-{(1,\sigma)} & E \oplus E\dual,
}\quad
\xymatrix@C+1pc{
M(\bsig) \ar[r]^-{\iota} \ar[d]_{0_{E|_{M(\sigma)}}} \cart & M \ar@{=}[r] \ar[d]|{0_E} \ar@{}[rd]|{\lag} & M \ar[d]^{0_{E\oplus E\dual}}  \\
E|_{M(\sigma)} \ar[r]^-{b} & E \ar[r]^-{(1,0)} & E \oplus E\dual.
}
\]
Since the second square is a symplectic pullback square, the third square is a pullback square, and the fourth square is a Lagrangian square, the first square is also a Lagrangian square.
Hence Corollary \ref{Cor:1} and Proposition \ref{Prop:bivariance} give us
\[[M(\bsig)/E]^\lag \circ b_{\sigma,*} = [b^*M(\bsig)/E|_{M(\sigma)}]^\lag = [M(\sigma)/E|_{M(\sigma)}]^\vir,\]
as desired.
\end{proof}

\begin{proof}[Proof of Theorem \ref{Thm:main}]
By Lemma \ref{Lem:1}, we may assume that $B$ is the total space of a perfect complex $E$ of tor-amplitude $[-1,0]$, $g:M \to B$ is the zero section $0_E: M \to E$, and $\bsig$ is induced by $\sigma: E\to \AA^1$ under Lemma \ref{Lem:sigmazerosection}.
By Lemma \ref{Lem:1}, we may assume that $E$ is a vector bundle.
By Lemma \ref{Lem:3} and Construction \ref{Const:1} with Lemma \ref{Lem:LocGysin}, we have
\[[M(\bsig)/E]^\lag = 0_{E,\sigma}^! = [M/E,\bsig]^\vir,\]
as desired.
\end{proof}

\bigskip
\section{Shifted Lagrange multipliers method}\label{S5}

In this section, we formulate the shifted Lagrange multipliers method and 
prove a formula comparing the localized virtual cycles of a stack defined by constraints and the stack enlarged by Lagrange multipliers. 
An immediate corollary is a generalized form of \emph{quantum Lefschetz without curves} in \cite{OTq} which in turn gives us the usual quantum Lefschetz principle in \cite{CL} in Gromov-Witten theory when applied to the moduli spaces of stable maps.  
In the subsequent section, we will look into applications in curve counting. 

\subsection{Classical Lagrange multipliers method}\label{S5.1}
The classical Lagrange multipliers method is typically about a diagram
\beq\label{91}\xymatrix{
& E\ar[d]\\
X=s^{-1}(0)\ar@{^(->}[r]\ar[drr]_{f|_X} & M\ar@/_/[u]_{s}\ar[dr]^f\\ 
&& \AA^1
}\eeq
where $M$ is smooth and the obstruction sheaf $Ob_X=\mathrm{coker}(\T_M|_X\xrightarrow{ds} E|_X)$
for $X$ is trivial so that $X$ is smooth. We are interested in the problem of finding the critical locus
$$\Crit_X(f|_X)=\Zero(df|_X)=X(df|_X)$$
of $f|_X$. Typically $M$ is much nicer than $X$ and we want to reduce the problem to finding the critical locus of a function on $M$ or a similarly nice space. 

In the Lagrange multipliers method, we let $F=E^\vee$ be the dual bundle of $E$. 
We call the fiber coordinates of the bundle projection $\pi:F\to M$ the \emph{Lagrange mulipliers}. 
By computing the derivatives, we find that  $\pi$ induces an isomorphism
\beq\label{92} \Crit_X(f|_X)\cong \Crit_F(f|_F+\varphi)\eeq
of critical loci where $\varphi=s\dual:F\to \AA^1$ is the map defined by pairing with $s\in H^0(F^\vee)$. 
Often it is much easier to compute $\Crit_F(f|_F+\varphi)$ than $\Crit_X(f|_X)$ and 
this method of Lagrange is extremely useful for finding extrema of a function subject to constraints. 


\medskip

In \S\ref{S5.2}, we generalize this Lagrange multipliers method to 
shifted functions. 
We find that \eqref{92} holds in the shifted setting as well (Corollary \ref{93}) even without any assumption about the smoothness of $M$ or $X$. 
Combining this with Theorem \ref{Thm:main}, we find that the localized virtual cycles of $X$ and $F$ coincide (Theorem \ref{Cor:Lefschetz}). 
The quantum Lefschetz without curves in \cite{OTq} is obtained as a special case.
In \S\ref{S5.3}, we prove the key result (Theorem \ref{Prop:LagMult}). 

\subsection{Shifted Lagrange multipliers}\label{S5.2}

Let $g:M \to B$ be a finitely presented morphism of derived Artin stacks.
Recall from Proposition \ref{Prop:TwCot} that the degeneracy locus $M(\bsig)$ of a locked $1$-form $\bsig \in \sA^{1,\lc}(M/B,-1)$ is $(-2)$-shifted locked symplectic.


\begin{example}[Critical loci] \label{Ex:Crit} 
Consider the {\em de Rham differential}
\[d : \sA^0(M,-1) \lra \sA^{1,\lc}(M/B,-1),\]
given by the boundary map in the cofiber sequence $\O_{\D_{M/B}} \xrightarrow{t} \O_{\D_{M/B}} \to i_*\O_{\T_{M/B}[1]}$.
\begin{itemize}
\item (Critical locus) Let $\varphi: M \to \AA^1[-1]$ be a $(-1)$-shifted function.
The {\em critical locus} of $\varphi$ is the zero locus of the differential $d\varphi \in \sA^{1,\lc}(M/B,-1)$, i.e. the fiber product
\[\xymatrix{
\Crit_{M/B}(\varphi) \ar[r] \ar[d] \cart & M \ar[d]^0 \\
M \ar[r]^-{d\varphi} & T^*_{M/B}[-1].
}\]
\item (Shifted cotangent bundle) Observe that the {\em $(-2)$-shifted cotangent bundle} $\T^*_{M/B}[-2]$ of $g:M\to B$ is the critical locus of the zero function $0:M \to \AA^1[-1]$, that is,
\[\T^*_{M/B}[-2]=\Tot(\LL_{M/B}[-2]) \simeq \Crit_{M/B}(0).\]
\end{itemize}
By Proposition \ref{Prop:TwCot}, $\Crit_{M/B}(\varphi)$ and $\T^*_{M/B}[-2]$ have oriented $(-2)$-shifted locked symplectic forms over $B$.
These symplectic forms are {\em exact} (i.e. the underlying functions are zero). 
\end{example}

The main result in this section is the following.

\begin{theorem}\label{Prop:LagMult}
Let $g:M \to B$ be a finitely presented morphism of derived Artin stacks,
$s:M \to E$ be a section of a perfect complex $E$
and $\bsig \in \sA^{1,\lc}(M/B,-1)$ be a $(-1)$-shifted locked $1$-form.
Denote by $X=s^{-1}(0)=\Zero(s)$ the zero locus of $s$ in $M$,
 $F=E\dual[-1]$,
 and $\varphi=s\dual[-1]:F \to \AA^1[-1]$.
Then there exists a canonical equivalence
\beq\label{81} X(\bsig|_X) \simeq F(\bsig|_F + d\varphi)\eeq
as $(-2)$-shifted locked symplectic fibrations over $B$. 
Moreover, the canonical orientations on the left and right sides of \eqref{81} differ by a multiplication factor $(-1)^{\rank(E)}$.
\end{theorem}
We will call $F=E\dual[-1]$ the \emph{bundle of Lagrange multipliers} over $M$.

For instance, if $\bsig \simeq df$ for a $(-1)$-shifted function $f$, 
then we have the following
generalization of \eqref{92}. 
%

\begin{corollary}[Lagrange multipliers method]\label{93} 
Let $g:M \to B$ be a finitely presented morphism of derived Artin stacks,
$s:M \to E$ be a section of a perfect complex $E$
and $f:M \to \AA^1[-1]$ be a shifted function.
Denote by $X=s^{-1}(0)$,
 $F=E\dual[-1]$,
 and $\varphi=s\dual[-1]$.
Then we have a canonical equivalence
\[\Crit_{X/B}(f|_X ) \simeq \Crit_{F/B}(f|_{F}+\varphi)\]
as $(-2)$-shifted exact symplectic fibrations over $B$.
In particular, letting $f\simeq 0$, we have a canonical equivalence
$\T^*_{X/B}[-2] \simeq \Crit_{F/B}(\varphi).$
\end{corollary}

As an immediate consequence of Theorem \ref{Prop:LagMult} and Theorem \ref{Thm:main},
we obtain a Lefschetz principle for cosection localized virtual cycles.

\begin{theorem}[Quantum Lefschetz principle]\label{Cor:Lefschetz}
Let $g:M \to B$ be a quasi-smooth morphism from a derived DM stack $M$ to a derived Artin stack $B$,
$s:M \to E$ be a section of a perfect complex $E$ of amplitude $[0,1]$
and $\bsig \in \sA^{1,\lc}(M/B,-1)$ be a $(-1)$-shifted locked $1$-form.
Denote by $X=s^{-1}(0)$, $F=E\dual[-1]$ and $\varphi=s\dual[-1]$.
If $X$ is quasi-smooth over $B$, then
\[[X/B,\bsig|_{X}]^\vir = (-1)^\ell \cdot [F/B, \bsig|_F+d\varphi]^\vir : A_*(B) \lra A_{*+r}(X(\sigma))\]
where $\ell=\rank(E)$, $r=\rank(\TT_{M/B})$ and $X(\sigma)$ is the degeneracy locus of the cosection $\sigma_X=\overline{\bsig|_X}$.
\end{theorem}
Indeed, we have
\beq\label{i15}
[X/B,\bsig|_{X}]^\vir =[X(\sigma)/B]^\lag = (-1)^\ell \cdot [F(\bsig|_F + d\varphi)]^\lag =(-1)^\ell\cdot [F/B, \bsig|_F+d\varphi]^\vir\eeq 
where the first and last equalities come from Theorem \ref{Thm:main} and the middle comes from Theorem \ref{Prop:LagMult}.

By letting $\bsig\simeq 0$ in Theorem \ref{Cor:Lefschetz}, we obtain the following. 
\begin{corollary}\label{94}
Let $g:M \to B$ be a quasi-smooth morphism from a derived DM stack $M$ to a derived Artin stack $B$ and $s:M \to E$ be a section of a perfect complex $E$ of tor-amplitude $[0,1]$. 
Let $F=E\dual[-1]$, $\ell=\rank \, F$ and $\varphi=s\dual[-1]$.
If the zero locus $X=s^{-1}(0)$ of $s$ is quasi-smooth over $B$, then we have the equality
\beq\label{82} [X/B]^\vir = (-1)^\ell \cdot [F/B,d\varphi]^\vir : A_*(B) \lra A_{*+r}(X).\eeq
\end{corollary}

If furthermore $B$ is a point, \eqref{82} gives us the \emph{quantum Lefschetz without curves}
\beq\label{95} [X]^\vir=(-1)^\ell \cdot [F,d\varphi]^\vir\ \ \in\ \  A_r(X)\eeq
which was proved  in \cite{OTq} by an independent method under the assumption that $E$ admits a global resolution by vector bundles.

\subsection{Symplectic zero loci}\label{S5.3}

In this subsection, we prove Theorem \ref{Prop:LagMult} using the notion of {\em symplectic zero loci} in \cite[\S3.2]{Park2}.
Let $g:M \to B$ be a finitely presented morphism of derived Artin stacks 
and $w:B \to \AA^1$ be a function.

A {\em symmetric complex} $S$ on a derived Artin stack $M$ is a perfect complex $S$ on $M$ together with a symmetric form $\beta :\O_M \to \Sym^2(S\dual)$ that induces an equivalence
\[\beta^{\#} : S \xrightarrow{\simeq} S\dual.\]


Given a section $s:M \to S$ of a symmetric complex $S$ and an equivalence $h : \beta(s,s) \to w|_M$ in $\sA^0(M,0)$,
we can form a $(-2)$-shifted locked symplectic fibration (cf. \cite[Def.~3.2.3]{Park2})
\[\Zero^\symp_{M/B}(S,s) \lra B,\]
called the {\em symplectic zero locus}.
We refer to \cite[\S3.2]{Park2} for the construction.
In this paper, we will only use the following property of symplectic zero loci.

\begin{lemma} \cite[Prop.~3.2.4, Eq.~(21)]{Park2} \label{Lem:SympZero}
Let $E$ be a perfect complex on $M$.
Let $s: M \to E$ be a section, $t : E \to \AA^1$ be a cosection, and $h : t \circ s \to w|_M $ be a homotopy in $\sA^0(M,0)$.
Denote by $t_* \in \sA^{1,\lc}(\Zero(s)/B,-1)$ the locked $1$-form induced by the path
\[ 0\xrightarrow{t \circ y}  t \circ s|_{\Zero(s)}\xrightarrow{h|_{\Zero(s)}} w|_{\Zero(s)}  \textin \sA^0(\Zero(s),0)\] by Remark \ref{Rem:n1}, where $y : 0 \to s|_{\Zero(s)}: \Zero(s) \to E$ is the canonical path for the zero locus.
Then there exists a canonical equivalence
\[\Zero^\symp_{M/B}(E\oplus E \dual,(s,t)) \simeq \Zero(s)(t_*)\]
as $(-2)$-shifted locked symplectic fibrations over $B$.
\end{lemma}

Consequently, if we change the role of $s$ and $t$, then we have
\[ \Zero(s)(t_*) \simeq \Zero(t)(s_*),\]
as locked $(-2)$-shifted symplectic fibrations over $B$.

\begin{proof}[Proof of Theorem \ref{Prop:LagMult}] 
By Remark \ref{Rem:n1}, 
$\bsig \in \sA^{1,\lc}(M/B,-1)$ is equivalent to a pair of 
\begin{itemize}
\item a function $w:B \to \AA^1$ and
\item a null-homotopy $n: 0 \to w|_M$ in $\sA^0(M,0)$.
\end{itemize}
We will apply Lemma \ref{Lem:SympZero} for $t=0$ and $h: t\circ s \simeq 0 \xrightarrow{n} w|_M$.
Since we have $\Zero(t)\simeq E\dual[-1]$,
\[s_* \simeq \bsig|_{E\dual[-1]} + s\dual[-1] \in \sA^{1,\lc}(E\dual[-1]/B,-1) \and t_* \simeq \bsig|_{\Zero(s)} \in \sA^{1,\lc}(\Zero(t)/B,-1), \]
we obtain the desired symplectic equivalence
\begin{equation}\label{Eq:4.1}
\Zero(s)(\bsig|_{\Zero(s)}) \simeq E\dual[-1](\bsig|_{E\dual[-1]}+ds\dual[-1]).
\end{equation}



For the comparison of orientations, we may assume $M \simeq B$
since we have Lagrangian correspondences for base changes \cite[Rem.~3.1.5(b)]{Park2}.
Consider the fiber square
\[\xymatrix{
W:=\Zero^{\symp}_{M/M}((E\oplus E\dual),(s,0)) \ar[r] \ar[d] \cart & \Zero(t)\simeq E\dual[-1]= F \ar[d] \\
X= \Zero(s) \ar[r] & M.
}\]
Then $\T_{W/M}[1]\simeq E\oplus E\dual$ and the two orientations are given by
\[p_E, p_{E\dual} : \det(E \oplus E\dual)|_W \simeq \det(E)|_W \otimes \det(E\dual)|_W \lra \O_W,\]
where $p_{(-)}$ are the pairings from \cite[Eq.~(56)]{OT}.
We thus have $p_{E\dual} = (-1)^{\rank(E)} \cdot p_E$.
\end{proof}

\bigskip
\section{Gromov-Witten invariants of complete intersections and branched covers}\label{S6}

Let $Q$ be a smooth quasi-projective variety. Fix $g,n\in \ZZ_{\ge 0}$ and $\beta\in H_2(Q,\Z)$. Let 
\beq\label{g0} M(Q):=\overline{M}_{g,n}(Q,\beta)\eeq
denote the moduli space of stable maps $f:C\to Q$ from a prestable curve $C\in \mathfrak{M}_{g,n}$ of genus $g$ with $n$ marked points such that $f_*[C]=\beta$. 
There is a natural perfect obstruction theory (cf. \cite{BF}) on $M(Q)$ which gives us the virtual fundamental class $[M(Q)]\virt\in A_*(M(Q))$. 
The Gromov-Witten invariants of $Q$ are defined as the integrals
\beq\label{g8}  \int_{[M(Q)]\virt}\eta\quad\quad\text{for }\eta\in H^*(M(Q),\Q)\eeq
which turned out to be quite useful in algebraic geometry, symplectic geometry and mathematical physics. 

Since its beginning, a fundamental idea in Gromov-Witten theory is to push the computation for a variety $Q$ to a simpler space $P$, when there is a proper morphism $\rho:Q\to P$. 
Let $M(P)=\overline{M}_{g,n}(P,\rho_*\beta)$. 
In lucky circumstances, the virtual pushforward formula
\beq\label{g1} \hat{\rho}_*[M(Q)]\virt=[M(P)]\virt\cap \gamma\quad\quad \text{for some }\gamma\in A^*(M(P))\eeq 
holds and thus the Gromov-Witten invariants of $Q$ (for insertions from $P$) may be computed by the formula
\beq\label{g4} \int_{[M(Q)]\virt}\eta|_{M(Q)}=\int_{[M(P)]\virt}(\eta\cup \gamma)\quad \text{for }\eta\in H^*(M(P),\Q)\eeq
where $\eta|_{M(Q)}$ denotes the pullback of $\eta$ by the natural map 
$$\hat{\rho}:M(Q)\lra M(P)$$ 
which composes any stable map to $Q$ with $\rho$ and contracts destabilizing components of the domain curve.   
\begin{example}[Quantum Lefschetz]\label{g20} 
Let $g=0$ and $\rho:Q\hookrightarrow P=\PP^m$ be a smooth hypersurface of degree $d$. Then we have the virtual pushforward formula
\beq\label{g3} \hat{\rho}_*[M(Q)]\virt=[M(P)]\virt\cap e(\pi_*\bff^*\O_{\PP^m}(d))\eeq 
where $e(\pi_*\bff^*\O_{\PP^m}(d))$ denotes the Euler class of the vector bundle $\pi_*\bff^*\O_{\PP^m}(d)$ and 
\beq\label{g2}\xymatrix{
\cC\ar[r]^{\bff}\ar[d]_{\pi} & P\\
M(P)}\eeq
is the universal family of stable maps to $P$. Hence we have the formula
\beq\label{g5} \int_{[M(Q)]\virt}\eta|_{M(Q)}=\int_{[M(P)]\virt} \eta\cup e(\pi_*\bff^*\O_{\PP^m}(d))\quad \text{for }\eta\in H^*(M(P),\Q)\eeq
whose right hand side may be effectively computed by a torus action on $\PP^m$.
\end{example} 
The quantum Lefschetz formula \eqref{g5} plays a key role in computation of genus $0$ Gromov-Witten invariants of 
hypersurfaces in toric varieties. 
For a higher genus, however, the virtual pushforward formula \eqref{g1} or \eqref{g3} does not hold even for hypersurfaces in a projective space and we cannot push the computation of Gromov-Witten invariants \eqref{g8} of $Q$ to $P$ by a formula like \eqref{g4} or \eqref{g5}. 

A breakthrough idea for higher genus quantum Lefschetz first appeared in \cite{CL} where Chang and Li discovered that if we enlarge $M(P)$ to the moduli space $M(P)^p$ of stable maps with $p$-fields, then
we have the identity
\beq\label{g6} \int_{[M(Q)]\virt}1=(-1)^\ell \cdot\int_{ [M(P)^p,\sigma]\virt}1,\quad\quad
\ell=\rank\,R\pi_*\bff^*\O_{\PP^4}(5),\eeq 
for a cosection $\sigma$ of $Ob_{M(P)^p}$, when $n=0$ and $Q$ is the quintic hypersurface in $P=\PP^4$.
The higher genus quantum Lefschetz formula \eqref{g6} turned out to be quite useful in Gromov-Witten theory. 

In this section, we will see that the enlarged space $M(P)^p$ above is in fact the 
bundle of Lagrange multipliers over $M(P)$ and \eqref{g6} is an immediate consequence of 
Corollary \ref{94}.   
In fact, 
the shifted Lagrange multipliers method in \S\ref{S5.2} works for all complete intersections (cf. \S\ref{S7.1}) and branched covers (cf. \S\ref{S7.2}).

The results in this section obviously hold for quasimap invariants as well but  we confine ourselves to Gromov-Witten invariants for clarity of presentation.

\subsection{Complete intersections}\label{S7.1}

Let $L$ be a vector bundle over a smooth quasi-projective variety $P$ and let $h$ be a section of $L$ whose zero locus is $Q=h^{-1}(0)$. 
In this  subsection, we always assume the following.
\begin{assumption}\label{g9}
$Q$ is smooth and the natural sequence
$$0\lra  \T_Q\lra T_P|_Q\mapright{dh} L|_Q\lra 0$$
is exact. 
\end{assumption}
To make the discussion simpler, let us also assume that  
$\rho_*:H_2(Q,\Z)\to H_2(P,Z)$ is injective where $\rho:Q\to P$ denotes the inclusion map. 

Note that a stable map $f:C\to P$ factors through $Q$ if and only if $f^*h\in H^0(C,f^*L)$ is zero. Globally over $M(P)$, using the notation of \eqref{g2}, we have a complex 
$$E=R\pi_*\bff^*L$$
of tor-amplitude $[0,1]$ and a section $s=\pi_*\bff^*h$ of $E$ whose vanishing locus is $M(Q)$. 
Letting $X=M(Q)$ and $M=M(P)$, we are in the situation of the diagram \eqref{91} (with $f=0$). 
The bundle of Lagrange multipliers is 
$$F=E\dual[-1]=R\pi_*(\bff^*L\dual\otimes \omega_{\cC/M(P)})$$
by Serre duality where $\omega_{\cC/M(P)}$ denotes the relative dualizing sheaf of $\pi$ in \eqref{g2}. 
The classical truncation of $F$ is the moduli space
$$\pi_*(\bff^*L\dual\otimes \omega_{\cC/M(P)})=\{(f\in M(P),p\in H^0(C,f^*L\dual\otimes\omega_C))\}=:M(P)^p$$
of stable maps $f:C\to P$ with a $p$-field in \cite{CL}. 
Thus the following \emph{quantum Lefschetz formula} is as an immediate consequence of Corollary \ref{94}.
\begin{proposition}\label{g10} We have the equality of (localized) virtual fundamental classes 
$$[M(Q)]\virt=(-1)^{\ell}\cdot [M(P)^p,d\varphi]\virt, \quad \ell=\rank\, R\pi_*\bff^*L$$
where 
$\varphi=s\dual[-1]:F\to \AA^1[-1]$.  
\end{proposition}

\subsection{Branched covers}\label{S7.2}
In this subsection, we consider Gromov-Witten invariants of branched covers. 
Let $P$ be a line bundle over a smooth quasi-projective variety $S$ and let 
$$h=y^m+h_1y^{m-1}+h_2y^{m-2}+\cdots+h_m\ \ \in H^0(P, L), \quad h_i\in H^0(S,P^{\otimes i})$$ 
be a section of the pullback $$L:=P|_P^{\otimes m}$$ of $P^{\otimes m}$ to $P$ by the bundle projection $q:P\to S$ where $y$ denotes the coordinate of the fibers of $q$. 
Let $Q\subset P$ denote the zero locus of $h$ so that we have a diagram
\beq\label{g11} \xymatrix{
Q\ar@{^(->}[r]^\imath\ar[dr]_\rho & P\ar[d]^q\\
& S.}\eeq
Note that $\rho:Q\to S$ is an $m$-fold cover. Once again, we suppose that Assumption \ref{g9} holds and $\rho_*:H_2(Q,\Z)\to H_2(S,\Z)$ is injective. 
Our goal is to express the virtual fundamental class $[M(Q)]\virt$ of the moduli space $M(Q)$ of stable maps to $Q$ by a localized virtual fundamental class of the moduli space of stable maps to $S$ with additional fields. 

By \eqref{g11}, we have the diagram 
\beq\label{g13} \xymatrix{
M(Q)\ar@{^(->}[r]^{\hat{\imath}}\ar[dr]_{\hat{\rho}} & M(P)\ar[d]^{\hat{q}}\\
& M(S)}\eeq
of moduli spaces of stable maps. 
As $q:P\to S$ is a line bundle, $M(P)$ is the moduli space 
\beq\label{g12} M(S)^y=\{(f\in M(S),y\in H^0(C,f^*P))\}\eeq
of stable maps $f:C\to S$ with a section $y$ of the line bundle $f^*P$. 
By \eqref{g12} and Proposition \ref{g10}, we have 
\begin{align*} 
[M(Q)]\virt &=(-1)^{\ell}\cdot [M(P)^p,d\varphi]\virt\\
&= (-1)^\ell\cdot [M(S)^{y,p},d\varphi]\virt
\end{align*}
where $\ell=\rank\, R\pi_*\bff^*P$ and 
\beq\label{g15}
M(S)^{y,p}=\{(f\in M(S),y\in H^0(C,f^*P), p\in H^0(C,f^*P^{\otimes -m}\otimes\omega_C))\}\eeq
is the moduli space of stable maps $f:C\to S$ to $S$ with a $y$-field (a section of $f^*P$) and a $p$-field (a section of $f^*P^{\otimes -m}\otimes\omega_C$) while 
\[
\xymatrix{
\cC\ar[r]^{\bff}\ar[d]_{\pi} & S\\
M(S)}
\]
is the universal family. So we have the following.
\begin{corollary}[Quantum Lefschetz for branched covers]\label{g14}
With the notation and assumptions above, we have 
$$
[M(Q)]\virt = (-1)^\ell\cdot [M(S)^{y,p},d\varphi]\virt, \quad \ell=\rank\, R\pi_*\bff^*P$$
where $M(S)^{y,p}$ is \eqref{g15}. 
\end{corollary}

\bigskip
\section{Gauged linear sigma models}\label{S4}

In this section, we show that the moduli stack for a GLSM comes equipped with a $(-1)$-shifted locked 1-form and hence its degeneracy locus has a $(-2)$-shifted locked symplectic form. 
We then prove Theorem \ref{105} which identifies the degeneracy locus with the GLSM moduli stack for the critical locus. Applying Theorem \ref{Thm:main}, we immediately obtain  
Corollary \ref{107} which compares the cosection localized virtual cycle with the virtual Lagrangian cycle.

\subsection{Input data}
The theory of GLSMs intends to construct cohomological field theories (cf. \cite{Manin}) from the following. 
\begin{definition}\cite{FJR} \label{Def:GLSMInput}
A {\em GLSM input datum} consists of
\begin{itemize}
\item [(D1)] a finite-dimensional vector space $V$ with an action of a reductive group $\widehat{G}$;
\item [(D2)] (Gauge group) $G=\ker(\chi)$ for a surjective group homomorphism $\chi : \widehat{G} \twoheadrightarrow \GG_m$;
\item [(D3)] (GIT quotient) $U=[V^{ss}(\theta)/G]$ for a character $\theta : \widehat{G} \to \GG_m$ such that $V^s(\theta)=V^{ss}(\theta)=V^{ss}(\theta|_G)$;
\item [(D4)] (Superpotential) a $\widehat{G}$-equivariant function $\widehat{w}: V \to \AA^1(\chi)$.
\end{itemize}
\end{definition}

Let $\M^{\mathrm{tw}}_{g,n}$ be the moduli stack of $n$-marked twisted prestable curves of genus $g$, which is a smooth Artin stack by \cite[Thm.~1.10]{Ols}.
The classical moduli stack $\M_{g,n}^{\omega^\log}(V/G)$ for a GLSM is an Artin stack parameterizing triples \eqref{108} and we consider open DM substacks $\M_{g,n}^{\omega^\log}(V/G)^\alpha$ (given by stability conditions) to obtain virtual invariants.

\subsection{Derived GLSM moduli}
We consider derived structures on GLSM moduli stacks and evaluation maps.   
Consider the universal family
\[\xymatrix{
\Sigma_{g,n,k} \ar@{^{(}->}[r]^-{} & \mathcal{C}_{g,n} \ar[d]^{} \\ & \M_{g,n}^{\tw} 
}\]
such that 
\begin{itemize}
\item (universal curve) $\mathcal{C}_{g,n} \to \M_{g,n}^{\tw}$ is a proper, flat, local complete intersection morphism;
\item (gerbe markings) $\Sigma_{g,n,k} \subseteq \mathcal{C}_{g,n}$ are effective Cartier divisors for $k \in \{1,\cdots, n\}$ such that the compositions $\Sigma_{g,n,k}\hookrightarrow \mathcal{C}_{g,n} \to \M_{g,n}^{\tw}$ are gerbes banded by finite cyclic groups.
\end{itemize}

\begin{definition}[Twisted maps]\label{Def:TwMaps}
Let $X$ be a derived stack with a $\GG_m$-action
and $\mathcal{L} \in \Pic(\mathcal{C}_{g,n})$.
We define the {\em moduli stack of $\mathcal{L}$-twisted maps to $X$} as the $\GG_m$-equivariant mapping stack
\[\M^{\mathcal{L}}_{g,n}(X):=\uMap^{\GG_m}_{\M_{g,n}^{\tw}}(\mathcal{L}^\circ,X\times \M_{g,n}^{\tw}),\]
where $\mathcal{L}^{\circ}\to \mathcal{C}_{g,n}$ is the principal $\GG_m$-bundle associated to the line bundle $\mathcal{L} \to \mathcal{C}_{g,n}$.
\end{definition}
More explicitly, a $\mathcal{L}$-twisted map to $X$ is a triple $(C,f,\xi)$ that consists of
\begin{itemize}
\item a twisted curve $C \in \M^{\tw}_{g,n}$ of genus $g$ with $n$-markings;
\item a morphism of derived stacks $f:C \to X/\GG_m$;
\item an isomorphism of line bundles $\xi : \mathcal{L}\simeq f^*\O_{X}(1)$, where $\O_X(1)$ is the pullback of the weight $1$ representation of $\GG_m$ by the projection map $X/\GG_m \to B\GG_m$.
\end{itemize}
Universally, we have morphisms 
\[\xymatrix{
&\M^{\mathcal{L}}_{g,n}(X) \times_{\M^{\tw}_{g,n}} \mathcal{C}_{g,n} \ar[rd]^-{\univ}\ar[ld]_{\pr_1} \ar[d]^{\pr_2} &   \\
\M^{\mathcal{L}}_{g,n}(X) & \mathcal{C}_{g,n} &X/\GG_m, & \pr_2^* (\cL) \simeq \univ^*(\O_{X}(1)).
}\]
If $X$ is a derived Artin stack of finite presentation, then so is $\M^{\mathcal{L}}_{g,n}(X)$ (Example \ref{Ex:TwMaps/Inertia}).
The cotangent complex of the forgetful map $\un:\M_{g,n}^{\mathcal{L}}(X) \to \M^{\tw}_{g,n}$ is 
\begin{equation}\label{Eq:GLSMCot}
\LL_{\M^{\mathcal{L}}_{g,n}(X)/\M^{\tw}_{g,n}} \simeq \pr_{1,*} \left( \univ^* (\LL_{[X/\GG_m]/B\GG_m}) \otimes \pr_2^*( \omega) \right)[1], \end{equation} 
by \eqref{Eq:WeilTangent}, where 
$\omega$ is the canonical line bundle of $\mathcal{C}_{g,n} \to \M^{\tw}_{g,n}$ 
(Example \ref{Ex:Dualizing}).

\begin{remark}[$\O$-twisted maps]
The usual twisted maps in the orbifold Gromov-Witten theory \cite{AGV} are the representable $\O$-twisted maps. 
More precisely,
when $\mathcal{L} \simeq \O \in \Pic(\mathcal{C}_{g,n})$, 
\[\M^{\O}_{g,n}(X) \simeq \uMap_{\M_{g,n}^{\tw}}(\mathcal{C}_{g,n},X\times {\M_{g,n}^{\tw}})\]
consists of maps $f:C \to X$ from twisted curves $C \in \M^{\tw}_{g,n}$.
Moreover, if $X$ is a derived Artin stack with separated diagonal, then the substack $\M^{\O}_{g,n}(X)^{\mathrm{rep}}\subseteq \M^{\O}_{g,n}(X)$ of {\em representable} maps $f:C \to X$ is an open substack.
Indeed, the openness follows from the facts that 
the relative inertia $I_{C/X}:=C\times_{\Delta_{X/C},C\times_X C,\Delta_{X/C}}C \to C$ of a separated DM morphism $f:C \to X$ is finite unramified and that twisted curves are proper.
\end{remark}

In this paper, we focus on the case where $\mathcal{L} \in \Pic(\mathcal{C}_{g,n})$ is the {\em log-canonical line bundle}
\[\omega^{\log}:=\omega \left(\sum_{1\leq k\leq n}\Sigma_{g,n,k}\right) \in \Pic(\mathcal{C}_{g,n}). 
\]
Since $\Sigma_{g,n,k} \to \M_{g,n}^{\tw}$ are {\'e}tale,
the restrictions of $\omega^{\log}$ to the gerbe markings are trivial by 
\[
\omega^{\log}|_{\Sigma_{g,n,k}} \simeq \omega_{\Sigma_{g,n,k}/\M_{g,n}} \simeq \O_{\Sigma_{g,n,k}}.\]
Equivalently, we can form a commutative diagram
\begin{equation}\label{Eq:GLSM10}
\xymatrix{
\M^{\omega^{\log}}_{g,n}(X) \times_{\M^{\tw}_{g,n}} \Sigma_{g,n,k} \ar@{.>}[r] \ar@{^{(}->}[d] & X \ar[d] \ar[r] \cart & \pt \ar[d] \\
\M^{\omega^{\log}}_{g,n}(X) \times_{\M^{\tw}_{g,n}} \mathcal{C}_{g,n} \ar[r]^-{\univ} & X/\GG_m \ar[r] & B\GG_m.
}\end{equation}


\begin{definition}[Rigidified cyclotomic inertia stack]\label{Def:Inertia}
Let $X$ be a derived stack.
We define the {\em moduli stack of cyclotomic gerbes in $X$} as the mapping stack 
\[I(X)=\uMap_{B^2\mu}(\pt, X\times B^2\mu)
\]
where $B^2\mu=\bigsqcup_{r\in \Z_{>0}}B^2{\mu_r}$ and $B^2\mu_r=B(B\mu_r)$ is the classifying stack of the classifying stack $B\mu_r$ of the finite cyclic group $\mu_r$.
\end{definition}

More explicitly, a family of cyclotomic gerbes in $X$ parametrized by a derived stack $S$ (i.e. a map $S \to I(X)$) is a pair $(\Sigma,g)$ of
\begin{itemize}
\item a gerbe $\Sigma \to S$ banded by finite cyclic groups and 
\item a morphism of derived stacks $g:\Sigma \to X$.
\end{itemize}
If $X$ is a derived Artin stack of finite presentation, then
$I(X) \to B^2\mu$
is relatively derived Artin of finite presentation (Example \ref{Ex:TwMaps/Inertia}).
Although $I(X)$ itself is a derived {\em higher} Artin stack {\em locally} of finite presentation,
we can still consider the shifted symplectic structures.

\begin{remark} [Classical truncations]
If $X$ is a derived Artin stack with separated diagonal,
then the stack of cyclotomic gerbes defined in \cite[Def.~3.3.6]{AGV} is the {\em open} substack of $I(X)_{\cl}$ that consists of {\em representable} maps $g:\Sigma \to X$ from gerbes $\Sigma \to S$.
\end{remark}


\begin{definition} [Evaluation maps]\label{Def:ev}
Let $X$ be a derived stack with a $\GG_m$-action.
Denote by
\[\ev_k :\M^{\omega^{\log}}_{g,n}(X) \lra I(X) \]
the map that corresponds to the dotted arrow in \eqref{Eq:GLSM10}.
We define the {\em evaluation map} as
\[\ev_X=(\ev_1,\cdots,\ev_n,\un):\M_{g,n}^{\omega^{\log}}(X) \lra I(X)^{\times n} \times \M_{g,n}^{\tw}.\]
where $\un :  \M_{g,n}^{\omega_{\log}}(X) \to \M^{\tw}_{g,n}$ is the forgetful map. 
\end{definition} 

If $X$ is a derived Artin stack of finite presentation,
the cotangent complex of the evaluation map $\ev_X:\M^{\omega^{\log}}_{g,n}(X) \to I(X)^{\times n}\times \M^{\tw}_{g,n}$ can be computed as (Proposition \ref{Prop:Weil})
\begin{align}
\LL_{\M^{\omega^{\log}}_{g,n}(X)/ I(X)^{\times n}\times \M^{\tw}_{g,n}} 
&\simeq \pr_{1,*} \left( \univ^* (\LL_{[X/\GG_m]/B\GG_m}) \otimes \pr_2^*( \omega^{\log}) \right) [1] \label{Eq:GLSMRelCot}\\
&\simeq  \pr_{1,*} \circ \univ^*  \left( \LL_{[X/\GG_m]/B\GG_m} \otimes \O_X(1)  \right)[1].\nonumber
\end{align}
Comparing this with the cotangent complex of the forgetful map $\un : \M^{\omega^{\log}}_{g,n}(X) \to \M^{\tw}_{g,n}$ in \eqref{Eq:GLSMCot}, we are replacing the canonical line bundle $\omega$ by the log-canonical line bundle $\omega^{\log}$.

\subsection{Shifted symplectic geometry for GLSMs}

Recall from Definition \ref{41} that locked forms on a derived stack $X$ are functions on the deformation space $\D_{X/\pt}$.  
More precisely, the space of $d$-shifted locked $p$-forms on $X$ is
\[\sA^{p,\lc}(X,d) = \Map^{\GG_m}(\D_{X/\pt},\AA^1(p)[d+p]).\]
Given a $\GG_m$-action on $X$, we define the space of $d$-shifted {\em $\GG_m$-equivariant locked $p$-forms of weight $k$ on $X$} as 
\[\sA^{p,\lc}_{\GG_m}(X,d)(k)= \Map^{\GG_m\times\GG_m}(\D_{X/\pt},\AA^1(p,-k)[d+p]).\]
As before, {\em exact forms} are locked forms whose restrictions to $\pt\times \GG_m \subseteq \D_{X/\pt}$ vanish.

\begin{proposition}[Lagrangians]\label{Prop:AKSZ}
Let $X$ be a derived Artin stack with a $\GG_m$-action and a $\GG_m$-equivariant $(-1)$-shifted exact symplectic form of weight $1$.
\begin{enumerate}
\item The rigidified cyclotomic inertia stack $I(X)$ has a $(-1)$-shifted exact symplectic form.
\item The evaluation map
\[\ev_X:\M_{g,n}^{\omega^{\log}}(X) \lra I(X)^{\times n} \times \M_{g,n}^{\tw}\]
has a $(-1)$-shifted exact Lagrangian structure over $\M_{g,n}^{\tw}$.
\end{enumerate}
\end{proposition}

We expect that together with the Joyce conjecture (Conjecture \ref{Conj:Joyce}), Proposition \ref{Prop:AKSZ} may give us a cohomological field theory by vanishing cycle cohomology, or more generally, cohomological Donaldson-Thomas invariants \cite{KLDT,BBDJS}.
We will discuss this in \S\ref{ss:JoyceConj}.

Proposition \ref{Prop:AKSZ} follows from the twisted version of the AKSZ construction in \cite{PTVV,CHS,CS}.
Indeed, we have a canonical commutative square (see Construction \ref{Const:AKSZ})
\[\xymatrix{
\sA^{p,\lc}_{\GG_m}(X,d)(1) \ar[r]^-{I(-)} \ar[d] \ar@{}[rd]|{\M(-)} & \sA^{p,\lc}(I(X),d) \ar[d]^{\sum_k \ev^*_k}  \\
\pt \ar[r]^-0 & \sA^{p,\lc}(\M_{g,n}^{\omega^\log}(X)/\M_{g,n}^{\tw},d).
}\]
Given a symplectic form $\theta \in \sA^{2,\lc}_{\GG_m}(X,-1)(1)$,
the induced form $I\theta$ is the symplectic form on $I(X)$ and the induced path $\M\theta$ is the Lagrangian structure on $\M_{g,n}^{\omega^\log}(X)\to I(X)^{\times n} \times \M_{g,n}^{\tw}$ (see Proposition \ref{Prop:TwSymp/Lag}).

If we apply the AKSZ construction to {\em functions}, we obtain locked $1$-forms as follows. 
\begin{lemma}[Locked $1$-forms]\label{Lem:GLMSCosection}
Let $U$ be a derived stack with a $\GG_m$-action.
Then a $\GG_m$-equivariant function $w:U \to \AA^1$ of weight $1$ induces a canonical $(-1)$-shifted locked $1$-form
\[\sigma_w \in \sA^{1,\lc}(\M_{g,n}^{\omega^\log}(U)/I(U)^{\times n}\times \M_{g,n}^{\tw},-1),\]
whose underlying function is
\[\underline{w}:=\sum_{1\leq k\leq n} \pr_k^*(Iw) : I(U)^{\times n} \times \M^{\tw}_{g,n} \lra \AA^1.\]
\end{lemma}

\begin{proof}
Recall from Remark \ref{Rem:n1} that giving a $(-1)$-shifted $\underline{w}$-locked $1$-form on $\ev_U$ is equivalent to giving a null-homotopy of the function
\[\underline{w} \circ \ev_U :\M_{g,n}^{\omega^\log}(U) \lra I(U)^{\times n} \times \M^{\tw}_{g,n} \lra \AA^1.\]
We define $\sigma_w$ as the locked $1$-form that corresponds to $\M w : 0 \to \underline{w}\circ \ev_U$.
\end{proof}

\begin{remark}[Underlying cosections]
In the situation of Lemma \ref{Lem:GLMSCosection},
the underlying $1$-form $[\sigma_w]$ of $\sigma_w$ is the {\em tautological} one.
Indeed, by \eqref{Eq:GLSMRelCot}, we can identify:
\[\sA^1(\M^{\omega^{\log}}_{g,n}(U)/ I(U)^{\times n}\times \M^{\tw}_{g,n},-1) \simeq \Gamma(\M^{\omega^{\log}}_{g,n}(U) \times_{\M^{\tw}_{g,n}} \mathcal{C}_{g,n}, \univ^*(\LL_{U}(1))).\]
Under the above identification, we have
$[\sigma_w] \simeq \univ^*(dw)$.
See Proposition \ref{Prop:Weil} for a proof.
\end{remark}


The main result in this section is the computation of 
$\M_{g,n}^{\omega^\log}(X)$
when $X$ is a critical locus. 
Let $U$ be a smooth Artin stack with a $\GG_m$-action and a $\GG_m$-equivariant function $w:U \to \AA^1$ of weight $1$.
Then we can form an induced commutative square
\[\xymatrix@C+3pc{
\M_{g,n}^{\omega^{\log}}(\Crit_U(w)) \ar[r]^-{\ev_{\Crit_U(w)}} \ar[d]_{\M_{g,n}^{\omega^{\log}}(p)} & I(\Crit_U(w))^{\times n} \times \M_{g,n}^{\tw} \ar[d]^{I(p)^{\times n} \times \id}\\
\M_{g,n}^{\omega^{\log}}(U) \ar[r]^-{\ev_U} &  I(U)^{\times n} \times \M_{g,n}^{\tw},
}\]
where $p: \Crit_U(w) \to U$ is the canonical map.
There exists a canonical equivalence 
\begin{equation}\label{Eq:GLSM11}
I(\Crit_U(w)) \simeq \Crit_{I(U)}(Iw).
\end{equation}
of $(-1)$-shifted exact symplectic derived Artin stacks by Proposition \ref{Prop:Weil}.
By \cite[Thm.~A]{Park2}, the $(-1)$-shifted exact Lagrangian structure on
\[ \M_{g,n}^{\omega^{\log}}(\Crit_U(w)) \xrightarrow{\ev_{\Crit_U(w)}} I(\Crit_U(w))^{\times n} \times \M_{g,n}^{\tw} \xrightarrow[\eqref{Eq:GLSM11}]{\simeq}  \Crit_{I(U)^{\times n}\times \M^{\tw}_{g,n}/\M^{\tw}_{g,n}}(\underline{w})\]
is equivalent to a $(-2)$-shifted $\underline{w}$-locked symplectic form on the composition
\beq\label{v2} 
\M_{g,n}^{\omega^{\log}}(\Crit_U(w))  \xrightarrow{\ev_{\Crit_U(w)}} I(\Crit_U(w))^{\times n} \times \M_{g,n}^{\tw} \xrightarrow{I(p)^{\times n} \times \id} I(U)^{\times n}\times \M^{\tw}_{g,n}.\eeq
On the other hand, by Proposition \ref{Prop:TwCot}, the $(-1)$-shifted $\underline{w}$-locked $1$-form $\sigma_w$ on 
\beq\label{112} 
\M_{g,n}^{\omega^{\log}}(U) \xrightarrow{\ev_U}  I(U)^{\times n} \times \M_{g,n}^{\tw}\eeq 
induces a $(-2)$-shifted $\underline{w}$-locked symplectic form on the relative degeneracy locus
\beq\label{v1}\M_{g,n}^{\omega^{\log}}(U)(\sigma_w)  \lra \M^{\omega^{\log}}_{g,n}(U) \xrightarrow{\ev_U} I(U)^{\times n}\times \M^{\tw}_{g,n}.\eeq


We can identify \eqref{v2} with \eqref{v1} by the following. 
\begin{theorem}\label{Prop:GLSMmain}
Let $U$ be a smooth Artin stack with a $\GG_m$-action and let $w:U \to \AA^1$ be  a $\GG_m$-equivariant function of weight $1$.
Then we have a canonical equivalence 
\[\M_{g,n}^{\omega^{\log}}(\Crit_U(w)) \simeq \M_{g,n}^{\omega^{\log}}(U)(\sigma_w)\]
of $\underline{w}$-locked $(-2)$-shifted symplectic fibrations over $I(U)^{\times n} \times \M_{g,n}^{\tw}$.
\end{theorem}

We will prove Theorem \ref{Prop:GLSMmain} in Appendix \ref{Sec:Weil}.

\subsection{GLSM invariants}
For a higher derived Artin stack $S$, if $S$ is 
a derived Artin stack over $(B^2\mu)^{\times n}$, we define the {\em Chow group} of $S$ as
\[A_*(S)=A_*\left(S\times_{(B^2\mu)^{\times n}} \pt\right).\]
This is well defined since the Chow group of a gerbe banded by finite groups is isomorphic to the Chow group of the base.
Even when $S$ is only locally of finite type, we can still define the Chow group as the inverse limit of the Chow groups of all open substacks of finite type. 

Now that we have a $(-1)$-shifted locked 1-form by Lemma \ref{Lem:GLMSCosection} and a $(-2)$-shifted locked symplectic structure on \eqref{v2}, we have the cosection localized virtual cycle $[\M_{g,n}^{\omega^{\log}}(U)^{\alpha}, \sigma_w]^\vir$
and the virtual Lagrangian cycle $[\M_{g,n}^{\omega^{\log}}(\Crit_U(w))^{\alpha}]^\lag$ for any open DM substack $\M_{g,n}^{\omega^{\log}}(U)^{\alpha}$ of $\M_{g,n}^{\omega^{\log}}(U)$.  
Combining Theorem \ref{Prop:GLSMmain} and Theorem \ref{Thm:main}, we obtain the following comparison result. 

\begin{corollary}\label{Cor:2}
Let $U$ be a smooth Artin stack	with a $\GG_m$-action and a $\GG_m$-equivariant function $w:U \to \AA^1$ of weight $1$.
Consider an open substack
$\M_{g,n}^{\omega^{\log}}(U)^{\alpha} \subseteq \M_{g,n}^{\omega^{\log}}(U)$
that is DM and let $\M_{g,n}^{\omega^{\log}}(\Crit_U(w))^{\alpha}=\M_{g,n}^{\omega^{\log}}(\Crit_U(w)) \cap \M_{g,n}^{\omega^{\log}}(U)^{\alpha}$.
Then we have the equality 
\[[\M_{g,n}^{\omega^{\log}}(U)^{\alpha}, \sigma_w]^\vir = [\M_{g,n}^{\omega^{\log}}(\Crit_U(w))^{\alpha}]^\lag : A_*(\underline{w}^{-1}(0)) \to A_*(\M_{g,n}^{\omega^{\log}}(\Crit_U(w))^{\alpha}).\]
\end{corollary}

In particular, the GLSM invariants defined by cosection localization in \cite{FJR} and the GLSM invariants defined by shifted symplectic structures in \cite{CZ} coincide when both make sense. 

\begin{remark}
We note that Corollary \ref{Cor:2} also works in the {\em equivariant} setting.
Indeed, in the situation of Corollary \ref{Cor:2}, if there is an additional action on $U$ of a linear algebraic group $H$ such that it commutes with the $\GG_m$-action and $w$ is $H$-invariant, then the AKSZ construction extends to $H$-invariant forms and the equivalence in Theorem \ref{Prop:GLSMmain} can be enhanced to an $H$-equivariant equivalence of derived stacks that preserves the $H$-invariant $\underline{\omega}$-locked $(-2)$-shifted symplectic forms.
Consequently, Theorem \ref{Thm:main} gives the equality
\[[\M_{g,n}^{\omega^{\log}}(U)^{\alpha}, \sigma_w]^\vir = [\M_{g,n}^{\omega^{\log}}(\Crit_U(w))^{\alpha}]^\lag : A^H_*(\underline{\omega}^{-1}(0)) \to A^H_*(\M_{g,n}^{\omega^{\log}}(\Crit_U(w))^{\alpha})\]
in the $H$-equivariant Chow groups.

\end{remark}


\begin{remark}
Our construction of the cosection localized virtual cycle $[\M_{g,n}^{\omega^{\log}}(U)^{\alpha}, \sigma_w]^\vir$ is more general than the virtual cycle in \cite{FJR} because we don't assume the factorization \cite[Def.~6.1.2]{FJR} which is quite restrictive. 
In fact, we always have a relative cosection for the {\em total evaluation map} $$\ev:\M_{g,n}^{\omega^{\log}}(U)^{\alpha} \lra I(U)^{\times n}\times \M^{\tw}_{g,n}.$$ 
The assumption in \cite[Lem.~6.1.1, Def.~6.1.2]{FJR} was introduced to lift the relative cosection to 
an absolute cosection in \cite[Lem.~6.1.5]{FJR} which implies the cone reduction by Lemma \ref{v3}. 
In our relative setup, we no longer need this assumption. 
\end{remark}
\begin{remark}\label{Rem:CZ}
Our construction of the virtual Lagrangian cycle $[\M_{g,n}^{\omega^{\log}}(\Crit_U(w))^{\alpha}]^\lag$ is more general than the virtual pullback in \cite{CZ} because our construction works for an arbitrary smooth Artin stack $U$ with a $\GG_m$-action and a function $w$, not necessrarily a quotient stack $[V/G]$ while the approach in \cite{CZ} depends on the presentation $[V/G]$ of $U$. 
Moreover,  we allow twisted prestable curves as domain curves. 
Also note that the  derived stack $\M^{\omega^{\log}}_{g,n}(\Crit_V(\widehat{w})/G)$ in \cite{CZ} is different from ours $\M^{\omega^{\log}}_{g,n}(\Crit_V(\widehat{w})/\!/G)$, because 
\[\Crit_V(\widehat{w})/G\not\simeq \Crit_V(\widehat{w})/\!/G \simeq\Crit_{V/G}(w),\]
as derived stacks although they are equivalent as classical stacks.
Here $\Crit_V(\widehat{w})/\!/G$ is the {\em symplectic quotient}.
Consequently, Cao-Zhao get a $(-2)$-shifted symplectic structure on 
\[\M^{\omega^{\log}}_{g,n}(\Crit_V(\widehat{w})/G) \lra \M^{\omega^{\log}}_{g,n}(BG)\times_{I(BG)^{\times n}} I(V/G)^{\times n},\] 
induced from the $(-1)$-shifted symplectic structure on $\Crit_V(\widehat{w})/G \to BG$,
while we get a $(-2)$-shifted symplectic structure on
\[\M^{\omega^{\log}}_{g,n}(\Crit_V(\widehat{w})/\!/G) \lra \M^{\tw}_{g,n}\times I(V/G)^{\times n},\]
induced from the $(-1)$-shifted symplectic structure on $\Crit_V(\widehat{w})/\!/G$.
However, the virtual Lagrangian cycles are comparable.
Indeed, consider the smooth morphism 
\[\nu:\M^{\omega^{\log}}_{g,n}(BG)\times_{I(BG)^{\times n}} I(V/G)^{\times n}\lra \M^{\tw}_{g,n}\times I(V/G)^{\times n}.\]
The $(-1)$-shifted locked $1$-form $\sigma_w$ induces a corresponding form $\sigma_w^{\CZ}$ by 
\[\sA^{1,\lc}(\M^{\omega^{\log}}_{g,n}(V/G)/ I(V/G)^{\times n}\times \M^{\tw}_{g,n},-1) \lra \sA^{1,\lc}(\M^{\omega^{\log}}_{g,n}(V/G)/ \M^{\omega^{\log}}_{g,n}(BG)\times_{I(BG)^{\times n}} I(V/G)^{\times n},-1). \] 
As a variant of Theorem \ref{Prop:GLSMmain}, using the trick Lemma \ref{Lem:Weil1}, Proposition \ref{Prop:Weil} also implies
\[\M^{\omega^{\log}}_{g,n}(\Crit_V(\widehat{w})/G) \simeq \M^{\omega^{\log}}_{g,n}(V/G)(\sigma_w^{\CZ})\]
as $(-2)$-symplectic fibrations over $\M^{\omega^{\log}}_{g,n}(BG)\times_{I(BG)^{\times n}} I(V/G)^{\times n}$.
Hence we get a commutative diagram 
\[\xymatrix{
 A_*(\M^{\tw}_{g,n}\times I(V/G)^{\times n} (\underline{w})) \ar[r]^-{} \ar[d]^{\nu^*} & A_* (\M^{\omega^{\log}}_{g,n}(\Crit_V(\widehat{w})/G) \simeq \M^{\omega^{\log}}_{g,n}(V/G)(\sigma_w)) \ar@{=}[d]  \\
 A_*(\M^{\omega^{\log}}_{g,n}(BG)\times_{I(BG)^{\times n}} I(V/G)^{\times n} (\underline{w})) \ar[r]^-{\CZ} & A_* (\M^{\omega^{\log}}_{g,n}(\Crit_V(\widehat{w})/\!/G) \simeq \M^{\omega^{\log}}_{g,n}(V/G)(\sigma_w^{CZ}))
}\]
from the functoriality of localized virtual pullbacks \cite[Thm.~2.10]{CKL}. 
Eventually, Cao-Zhao consider the invariants after taking $\nu^*$ and hence our invariants recover theirs.
\end{remark}


\subsection{Vanishing cycles}\label{ss:JoyceConj}

Finally, we discuss what we expect for a general $(-1)$-shifted exact symplectic stack $X$.
Given an orientation of $X$ (i.e. a square-root of $\det(\LL_X)$ in $\Pic(X)$),
denote by $\phi_X\in \Perv(X)$ the perverse sheaf constructed in \cite{KLDT,BBDJS}.

\begin{conjecture}[{Joyce, cf.~\cite[Conj.~5.22]{AB}}]\label{Conj:Joyce}
Let
\[\xymatrix{
& L \ar[ld]_s \ar[rd]^t & \\ M && N.
}\]
be a $(-1)$-shifted exact oriented Lagrangian correspondence. 
Then there is a canonical map
\[[L]^\lag : s^*\phi_M[\dim\,L]\lra t^!\phi_N \textin \D^b_c(L),\]
functorial along compositions of Lagrangian correspondences.
\end{conjecture}

By \cite[Thm.~A]{Park2}, the $(-1)$-shifted exact Lagrangian 
\[\ev_X:\M_{g,n}^{\omega^{\log}}(X) \lra I(X)^{\times n} \times \M_{g,n}^{\tw}\quad\text{over}\quad\M_{g,n}^{\tw}\]
in Proposition \ref{Prop:AKSZ} 
is equivalent to a $(-1)$-shifted exact Lagrangian correspondence
\begin{equation}\label{Eq:GLSM12}
\xymatrix{
& \M_{g,n}^{\omega^{\log}}(X)  \ar[ld]_{s:=(\ev_1,\cdots,\ev_n)} \ar[rd]^t & \\  I(X)^{\times n}  && \T^*_{\M_{g,n}^{\tw}}[-1].
}\end{equation}
Assuming the Joyce conjecture and the existence of orientation of the Lagrangian, we have 
\[[\M_{g,n}^{\omega^{\log}}(X)]^\lag : s^* \phi_{I(X)^{\times n}} \lra t^! \phi_{\T^*_{\M_{g,n}^{\tw}}[-1]}[-d_X] \simeq t^! \Q_{\M_{g,n}^{\tw}}[-d_X+d_{\pt} ] 
\simeq \omega_{\M_{g,n}^{\omega^{\log}}(X)}[-d_X-d_{\pt}],\]
where $d_X=\dim\,{\M_{g,n}^{\omega^{\log}}(X)}$.

\begin{corollary} \label{111} 
Assume Conjecture \ref{Conj:Joyce}.
Let $X$ be a derived Artin stack with a $\GG_m$-action, an oriented exact $(-1)$-shifted symplectic form of weight $1$ and an orientation of \eqref{Eq:GLSM12}.
Then we have a canonical map
\[[\M_{g,n}^{\omega^{\log}}(X)]^\lag : H^*(I(X), \phi_{I(X)})^{\otimes n} \lra H^{\BM}_{d_X+d_{\pt}-*}(\M_{g,n}^{\omega^{\log}}(X),\Q).\]
If $\M_{g,n}^{\omega^{\log}}(X)^{\alpha} \subseteq \M_{g,n}^{\omega^{\log}}(X)$ is an open substack that is proper, then we have a map
\[(\mathrm{st}\circ \un)_* \circ [\M_{g,n}^{\omega^{\log}}(X)^{\alpha}]^\lag : H^*(I(X), \phi_{I(X)})^{\otimes n} \lra H^{*+d_X-d_{\pt}}(\overline{\M}_{g,n},\Q)\]
where $\overline{\M}_{g,n} \subseteq \M_{g,n}^{\tw}$ is the proper DM stack of stable curves and $\mathrm{st} : \M^{\tw}_{g,n} \lra \overline{\M}_{g,n}  $ is the stabilization map.
\end{corollary}

A $K$-theoretic version of Corollary \ref{111} for $X=\Crit_U(w)$ is constructed in \cite{CTZ}.

\bigskip
\appendix
\section{Weil restrictions of critical loci}\label{Sec:Weil}

In this section, we show that Weil restrictions preserve twisted cotangent bundles.
Theorem \ref{Prop:GLSMmain} will then follow as an example.

\subsection{Weil restrictions}

For any derived stack $B$, denote by $\dSt_B$ the $\infty$-category of derived stacks over $B$.
\begin{definition}[Weil restriction]
Let $g:D \to B$ be a morphism of derived stacks.
We define the {\em Weil restriction} functor
\[\Res_{D/B} : \dSt_D \lra \dSt_B\]
as the right adjoint of the pullback functor $(-)\times_B D :\dSt_B \to \dSt_D$.
\end{definition}

For any $g:D \to B$ and $V \in \dSt_D$, we can form a correspondence
\[\xymatrix{
&\Res_{D/B}(V) \times_B D \ar[ld]_-{\pr_1} \ar[rd]^{\univ} & \\ \Res_{D/B}(V) && V ,
}\]
where the map $\univ$ is given by the counit of the adjunction $(-)\times_BD \dashv \Res_{D/B}$.


\begin{lemma}[Base change]\label{Lem:Weil1}
Let $g:D \to B$ be a morphism of derived stacks and
$V \in \dSt_D$.
For any $Y \in \dSt_V$, we have a canonical equivalence of derived stacks
\[\Res_{D/B}(Y) \simeq \Res_{\Res_{D/B}(V)\times_BD/\Res_{D/B}(V)} \left(Y\times_{V,\univ} \left(\Res_{D/B}(V)\times_B D\right)\right).\]
\end{lemma}
\begin{proof}
It follows immediately from the definition of Weil restrictions.
\end{proof}

\begin{lemma}[Total spaces]\label{Lem:Weil2}
Let $g: D \to B$ be a morphism of derived Artin stacks
and $E\in\Perf(D)$ be a perfect complex on $D$.
Then we have a canonical equivalence of derived stacks
\[\Res_{D/B}\left(\Tot_D(E)\right) \simeq \Tot_B \left(f_*E\right) . \]
\end{lemma}
\begin{proof}
The pushforward $g_*:\QCoh(D) \to \QCoh(B)$ has the base change formula by \cite[Cor.~1.4.5]{DG}.
Then the equivalence follows from the definition.
\end{proof}

When $g:D \to B$ is a proper flat morphism of classical Artin stacks,
if $V$ is a derived Artin stack of finite presentation over $D$,
then $\Res_{D/B}(V)$ is a derived higher Artin stack of finite presentation over $B$ by \cite[Thm.~5.1.1, Rem.~5.1.3]{HP}.
Moreover, the tangent complex is:
\begin{equation}\label{Eq:WeilTangent}
\TT_{\Res_{D/B}(Y)/B} \simeq \pr_{1,*}\circ \univ^*(\TT_{Y/D}).
\end{equation}
Indeed, this follows from Lemma \ref{Lem:Weil1}, Lemma \ref{Lem:Weil2}, and the canonical equivalences:
\[\T_{\Res_{D/B}(Y)/B}[1] \simeq \uMap_B(\AA^1_B[-1],\Res_{D/B}(Y)) \simeq \Res_{D/B}(\uMap_D(\AA^1_D[-1],Y)) \simeq \Res_{D/B}(\T_{Y/D}[1]).\]

\begin{example}[Twisted maps/inertia]\label{Ex:TwMaps/Inertia}
Let $X$ be a derived stack with a $\GG_m$-action.
\begin{enumerate}
\item The moduli stack of $\cL$-twisted maps to $X$ (Definition \ref{Def:TwMaps}) can be described as:
\[\M^{\cL}_{g,n}(X)\simeq \Res_{\mathcal{C}_{g,n}/\M^{\tw}_{g,n}}\left((X/\GG_m)\times_{B\GG_m,\cL}\cC_{g,n}\right).\]
\item The moduli stack of cyclotomic gerbes in $X$ (Definition \ref{Def:Inertia}) can be described as:
\[I(X)\simeq \Res_{\pt/B^2\mu}(X).\]
\end{enumerate}
If $X$ is a derived Artin of finite presentation, then so is $\M^{\cL}_{g,n}(X) \to \M^{\tw}_{g,n}$ and $I(X) \to B^2\mu$.
\end{example}

\subsection{Dualizing complexes}

Recall from \cite[Def.~2.1]{PTVV} and \cite[Def.~B.10.16]{CHS} that a morphism $g:D \to B$ is {\em $\O$-compact} if universally, $g_*:\QCoh(D) \to \QCoh(B)$ preserves perfect complexes and colimits.
For an $\O$-compact morphism $g:D \to B$, denote by:
\begin{itemize}
\item $g^{\times}:\QCoh(B) \to \QCoh(D)$ the right adjoint of $g_* :\QCoh(D) \to \QCoh(B)$;
\item $\omega_{D/B}=g^{\times}(\O_B)\in \QCoh(D)$ the {\em $*$-dualizing complex}. 
\end{itemize}


\begin{definition}[$*$-Gorenstein]
A {\em $*$-Gorenstein} morphism $g:D \to B$ is an $\O$-compact morphism of derived stacks such that $\omega_{D/B}\in \QCoh(D)$ is a shifted line bundle.
\end{definition}

\begin{definition}[Boundary]
A {\em boundary structure} on a $*$-Gorenstein morphism $i:D \to C$ consists of 
\begin{enumerate}
\item [(D1)] a shifted line bundle $\widetilde{\omega}_{D/C} \in \QCoh(C)$ with $i^*\widetilde{\omega}_{D/C} \simeq\omega_{D/C} \in \QCoh(D)$;
\item [(D2)] a cofiber sequence
$
\partial_{D/C}: \xymatrix{\widetilde{\omega}_{D/C} \ar[r]^-{i^*\dashv i_*} & i_*\omega_{D/C} \ar[r]^-{i_* \dashv i^{\times}} & \O_C}
$ in $\QCoh(C)$.
\end{enumerate}
When $f: C \to B$ is a $*$-Gorenstein morphism, we let 
$\omega_{C/B}^{\log}=\omega_{C/B}\otimes \widetilde{\omega}_{D/C}.$
\end{definition}

We say that $i:D \to C$ is an {\em effective divisor} if it is equipped with a line bundle $L \in \Pic(C)$, a section $s\in \Gamma(C,L)$ and an equivalence $(i:D \to C) \simeq (\Zero(s) \to C)$.

\begin{lemma}\label{Lem:5}
Let $f:C \to B$ and $i:D \to C$ be morphisms of derived Artin stacks.
\begin{enumerate}
\item If $f:C \to B$ is a proper quasi-smooth morphism, then it is a $*$-Gorenstein morphism.
\item If $i:D \to C$ is an effective divisor, then it has a canonical boundary structure.
\end{enumerate}
\end{lemma}

\begin{proof}
We first note that any proper quasi-smooth morphism $f:C \to B$ of derived Artin stacks is $\O$-compact.
Indeed, $f_*:\QCoh(C) \to \QCoh(B)$ preserves colimits and has the base change formula by \cite[Cor.~1.4.5]{DG} and preserves perfect complexes by \cite[Thm.~6.1.3.2]{LurSAG}.

(2) 
For an effective divisor $i:D\to C$ of $s\in \Gamma(C,L)$, we have a canonical cofiber sequence
\[\xymatrix{
L\dual \ar[r]^-{s} & \O_C \ar[r]^-{i^*\dashv i_*} & i_* \O_D .
}\]
Under the adjunction $i_* \dashv i^{\times}$, the map $i_*\O_D \to L\dual[1]$ corresponds to a map
\[\O_D \lra L\dual \otimes \omega_{D/C}[1],\]
which is an equivalence by \cite[Prop.~4.2]{Khan}.
Hence $i:D \to C$ is a $*$-Gorenstein morphism which has a boundary structure with $\widetilde{\omega}_{D/C}\simeq L[-1].$

(1) Since $f^{\times}$ commutes with base changes by \cite[Prop.~6.4.2.1]{LurSAG}, we may assume that $B=\pt$.
Since $C$ is proper quasi-smooth, being $*$-Gorenstein is equivalent to being Gorenstein (in the sense of \cite[Def.~6.6.5.1]{LurSAG}) by \cite[Prop.~6.6.6.7]{LurSAG}.
Since being Gorenstein is {\'e}tale-local by \cite[Rem.~6.6.1.3]{LurSAG}, we may assume that $C$ is affine. 
Then we can find a smooth scheme $U$, a trivial vector bundle $E\simeq \O_U^{\oplus r}$ on $U$ and a section $s\in \Gamma(U,E)$ such that $C \simeq \Zero(s)$. 
Since Gorenstein morphisms (in the sense of \cite[Def.~6.6.6.1]{LurSAG}) are stable under compositions by \cite[Cor.~6.6.6.6]{LurSAG},
it suffices to show that the inclusion $\Zero(s) \hookrightarrow U$ is Gorenstein.
Since $\Zero(s) \hookrightarrow U$ can be factored by effective divisors, the statement follows from (2).
\end{proof}

\begin{example}[Twisted curves]\label{Ex:Dualizing}
Let $\M^{\tw}_{g,n}$ be the moduli stack of twisted curves.
The universal curve $\cC_{g,n} \to  \M^{\tw}_{g,n}$ is a $*$-Gorenstein morphism 
and the gerbe marking $\bigsqcup_k \Sigma_{g,n,k} \to \cC_{g,n}$ is a boundary (Lemma \ref{Lem:5}).
The {\em canonical line bundle} (resp. {\em log-canonical line bundle}) is:
\[\omega:=\omega_{\cC_{g,n}/\M^{\tw}_{g,n}}[-1], \quad (\text{resp. }\omega^{\log}:=\omega_{\cC_{g,n}/\M^{\tw}_{g,n}}^{\log}[-1] ).\]
See also \cite[Cor.~3.2]{HallPriver} for defining canonical line bundles of proper lci morphisms of classical Artin stacks and \cite[Prop.~1.3]{HR} for comparing $\QCoh$ with the classical derived categories.
\end{example}

\subsection{Twisted forms}

Following \cite[\S1]{Park2}, we form the {\em de Rham algebra} of $S\in \dSt_B$ as:
\[\DR(S/B) := (\D_{S/B}/\GG_m \to B\times \AA^1/\GG_m)_*(\O_{\D_{S/B}/\GG_m}) \in \QCAlg_{B\times \AA^1/\GG_m}=:\QCAlg_B^\fil.\]
For
$\cE \in \Perf(B)$,
the {\em derived stack of $\cE$-twisted locked/ordinary $p$-forms} is:
\begin{align*}
\sA^{p,\lc}_B(\cE)\in \dSt_B : (S\to B) \mapsto \Map_{\QCoh(B)}\left(\O_B,\Fil^p\DR(S/B)[p]\otimes \cE\right),\\
\sA^{p}_B(\cE)\in \dSt_B : (S\to B) \mapsto \Map_{\QCoh(B)}\left(\O_B,\Gr^p\DR(S/B)[p]\otimes \cE\right).
\end{align*}
If $\cE$ is a shifted line bundle and $S \to B$ is a finitely presented morphism of derived Artin stacks,
we say that $\theta \in \sA^{2,\lc}(S/B,\cE)$ is a {\em $\cE$-twisted locked symplectic} if it induces an equivalence:
\[\theta^{\#}:\TT_{S/B} \to \LL_{S/B}\otimes \cE.\]

\begin{example}[Twisted cotangent bundles]\label{Ex:TwCot}
Let $p:U \to C$ be a finitely presented morphism of derived Artin stacks,
$\cE \in \Perf(C)$
and $\alpha \in \sA^{1,\lc}(U/C,\cE)$.
\begin{enumerate}
\item The {\em $\cE$-twisted contagent bundle} is the total space:
\[\T^*_{U/C}(\cE)=\Tot_U(\LL_{U/C} \otimes p^*\cE).\]
The tautological section gives us the {\em Liouville $1$-form}
$\lambda_{U/C}(\cE) \in \sA^1(\T^*_{U/C}(\cE)/C, \cE).$
\item The {\em $\cE$-twisted $\alpha$-twisted cotangent bundle} is the fiber product:
\begin{equation}\label{Eq:GLSM4}
\xymatrix{
\T^*_{U/C,\alpha}(\cE[-1]) \ar[r] \ar[d] \cart & U \ar[d]^0 \\ U \ar[r]^-{\Gamma_{\alpha}} & \T^*_{U/C}(\cE).
}\end{equation}
As in Proposition \ref{Prop:TwCot}, we have an $\cE[-1]$-twisted locked symplectic form
\[\theta_{U/C,\alpha}(\cE) \in \sA^{2,\lc}(\T^*_{U/C,\alpha}(\cE[-1])/C, \cE[-1]).\]
\end{enumerate}
\end{example}

\subsection{Twisted AKSZ}

The {\em AKSZ construction} is based on the following projection formula.

\begin{lemma}[Projection formula]
Let $g:D \to B$ be a $\O$-compact morphism and $\cE \in \Perf(D)$ be a perfect complex.
The canonical map
\[\Res_{D/B}\left(\sA^{p,\lc}_D(\cE)\right)\xleftarrow{\simeq} \sA^{p,\lc}_B(g_*\cE)\]
is an equivalence of derived stacks over $B$.
\end{lemma}
\begin{proof}
For any derived stack $S$ over $B$, we have a canonical map
\[g_*\left(\DR(S\times_BD/D)\otimes \cE\right) \simeq g_*\left(g^*\DR(S/B)\otimes \cE\right)  \xleftarrow{} \DR(S/B)\otimes g_*\cE \]
given by the adjunction $g^*\dashv g_*$.
Since the $\O$-compact morphism $g:D \to B$ has the projection formula \cite[Prop.~3.10]{BFN}, the above map is an equivalence.
\end{proof}

Consequently, we have two induced structures.
\begin{enumerate}
\item (Integration) For a $*$-Gorenstein morphism $g: D\to B$, we have a map
\[\int_{D/B} : \Res_{D/B}\left(\sA_D^{p,\lc}(\omega_{D/B}[d])\right) \simeq \sA^{p,\lc}_B(g_*\omega_{D/B}[d]) \xrightarrow{g_* \dashv g^{\times}} \sA^{p,\lc}_B[d].\]
\item (Boundary) For a boundary $i:D \to C$, we have a fiber square 
\[\xymatrix{\sA^{p,\lc}_C(\widetilde{\omega}_{D/C}[d]) \ar[r] \ar[d] \ar@{}[rd]|{\partial_{D/C}} & \Res_{D/C}\left(\sA^{p,\lc}_D(\omega_{D/C}[d])\right) \ar[d]^{\int_{D/C}} 
\\
C \ar[r]^-{0}&  \sA^{p,\lc}_C(\O_C[d]). 
}\]
\end{enumerate}

\begin{construction}[AKSZ]
\label{Const:AKSZ}
Let $f:C \to B$ be a $*$-Gorenstein morphism with a boundary $i: D \to C$.
\begin{enumerate}
\item Let $Y \in \dSt_D$ and $\beta \in \sA^{p,\lc}(Y/D,\omega_{D/B}[d])$.
We define the induced locked form
\[\AKSZ_{D/B}(\beta)=\int_{D/B}\Res_{D/B}(\beta) \in \sA^{p,\lc}(\Res_{D/B}(Y)/B,d)\]
as the composition
\[\Res_{D/B}(Y) \xrightarrow{\Res_{D/B}(\beta)} \Res_{D/B}\left(\sA^{p,\lc}_D(\omega_{D/B}[d])\right) \xrightarrow{\int_{D/B}} \sA^{p,\lc}_B[d].\]
\item Let $X\in \dSt_C$ and $\alpha \in \sA^{p,\lc}(X/C,\omega_{C/B}^\log[d])$.
We define the induced path
\[\AKSZ_{C/B}(\alpha) : 0 \lra \AKSZ_{D/B}(\beta)|_{\Res_{C/B}(X)} \textin \sA^{p,\lc}(\Res_{C/B}(X)/B,d)\]
via the commutative diagram
\[\xymatrix{
\Res_{C/B}(X) \ar[d]_{\Res_{C/B}(\alpha)} \ar[r] & \Res_{D/B}(Y) \ar[d]_{\Res_{D/B}(\beta)} \\
\Res_{C/B}\left(\sA^{p,\lc}_C(\omega_{C/B}^\log[d])\right) \ar[r] \ar[d]\ar@{}[rd]|{\partial_{D/C}} & \Res_{D/B}\left(\sA^{p,\lc}_D(\omega_{D/B}[d])\right) \ar[rd]^{\int_{D/B}}\ar[d]_{\int_{D/C}} \\
0 \ar[r]& \Res_{C/B}\left(\sA^{p,\lc}_C(\omega_{C/B}[d])\right) \ar[r]_-{\int_{C/B}}
& \sA^{p,\lc}_B[d],
}\]
where $Y:=X\times_CD$ and $\beta:=\alpha|_{Y/D} \in \sA^{p,\lc}(Y/D,\omega_{D/B}[d])$ are the restrictions.

If $p>0$, then we also have an induced locked form
\[\AKSZ_{C/B}(\alpha)_{/}:=\AKSZ_{C/B}(\alpha)_{/\Res_{D/B}(Y)} \in \sA^{p,\lc}(\Res_{C/B}(X)/\Res_{D/B}(Y),d-1).\]
\end{enumerate}
\end{construction}

\begin{proposition}\label{Prop:TwSymp/Lag}
Let $f:C \to B$ be a proper flat lci morphism of classical Artin stacks with an effective Cartier divisor $i:D \to C$ such that $f \circ i:D \to B$ is flat.
\begin{enumerate}
\item If $\beta$ is a $\omega_{D/B}[d]$-twisted locked symplectic form on a morphism of derived Artin stacks $Y \to D$, then $\AKSZ_{D/B}(\beta)$ is a $d$-shifted locked symplectic form on
\[\Res_{D/B}(Y) \lra B.\]
\item If $\alpha$ is a $\omega_{C/B}^{\log}[d]$-twisted locked symplectic form on a morphism of derived Artin stacks $X \to C$, then $\AKSZ_{C/B}(\alpha)$ is a $(d-1)$-shifted locked Lagrangian structure 
\[\Res_{C/B}(X) \lra \Res_{D/B}(X|_D).\]
\end{enumerate}
\end{proposition}

\begin{proof}
The non-degeneracy depends only on the underlying (closed) $2$-forms.
Hence the statements follow from \cite[Thm.~2.34]{CS}.
\end{proof}

\subsection{Weil restrictions of twisted cotangents}

Our main result in this section is:

\begin{proposition} 
\label{Prop:Weil}
Let $f:C \to B$ be a proper flat lci morphism of classical Artin stacks with an effective Cartier divisor $i:D \to C$ such that $f \circ i:D \to B$ is flat.
Let $U$ be a derived Artin stack of finite presentation over $C$ and $\alpha \in \sA^{1,\lc}(U/C,\omega_{C/B}^\log[d])$.
Then we have a canonical equivalence of locked Lagrangians
\begin{equation}\label{Eq:GLSM6}
\xymatrix{
\Res_{C/B}\left(\T^*_{U/C,\alpha}(\omega_{C/B}^\log[d-1])\right) \ar[r] \ar@{.>}[d]^{\simeq}_{\Psi_{U/C,\alpha}} &  \Res_{D/B}\left(\T^*_{V/D,\beta}(\omega_{D/B}[d-1])\right)  \ar@{.>}[d]^{\simeq}_{\Phi_{V/D,\beta}}\\
\T^*_{\Res_{C/B}(U)/\Res_{D/B}(V),\AKSZ_{C/B}(\alpha)_{/}}[d-2] \ar[r] & \T^*_{\Res_{D/B}(V)/B,\AKSZ_{D/B}(\beta)}[d-1],
}\end{equation}
where $V:=U\times_C B$ and $\beta:=\alpha|_{V/D} \in \sA^{1,\lc}(V/D,\omega_{D/B}[d])$ are the restrictions
and the lower Lagrangian is the one in Lemma \ref{Lem:Conormal} below.
\end{proposition}


\begin{lemma}[Twisted conormal Lagrangian]\label{Lem:Conormal}
Let $R \to Q \to B$ be finitely presented morphisms of derived Artin stacks,
$\delta \in \sA^{1,\lc}(Q/B,d)$
and $\gamma : 0 \lra \delta|_R$ be a path in $\sA^{1,\lc}(R/B,d)$.
Then we have a canonical $d$-shifted locked Lagrangian
\[\T^*_{R/Q,\gamma_{/}}[d-1] \lra \T^*_{Q/B,\delta}[d],\]
where $\gamma_{/}\in \sA^{1,\lc}(R/B,d-1)$ corresponds to the loop $\gamma_{/Q} : 0 \to \delta|_{R/Q} \simeq 0$ in $\sA^{1,\lc}(R/B,d)$.
\end{lemma}

\begin{proof}
We have a canonical commutative diagram
\begin{equation}\label{Eq:GLSM21}
    \xymatrix{
\T^*_{R/Q,\gamma_{/}}[d-1] \ar[r] \ar[d] \cart &  \T^*_{Q/B,\delta}[d]\times_Q R \ar[r]^-{\pr_1} \ar[d] & \T^*_{Q/B,\delta}[d] \ar[d]^{\theta_{Q/B,\delta}[d]}
 \\
R \ar[r]^-0 & \T^*_{R/B}[d] \ar[r]^-{d\lambda_{R/B}[d]} & \sA^{2,\lc}_C[d]
}\end{equation}
where the left square is cartesian and the right square is a locked Lagrangian correspondence.
Hence the lemma follows from the composition of Lagrangian correspondences.
\end{proof}

We note that the induced $(d-1)$-shifted locked symplectic form on the composition
\[\T^*_{R/Q,\gamma_{/}}[d-1] \lra \T^*_{Q/B,\delta}[d] \lra Q,\]
from the above Lagrangian structure and the Lagrangian fibration structure of $\T^*_{Q/B,\delta}[d] \to Q$, is the canonical one for the twisted cotangent bundles. 

\begin{proof}[Proof of Proposition \ref{Prop:Weil}]
For simplicity, we suppress the notation $/B$.

We first consider the case when $\alpha \simeq 0$.
By \cite[Prop.~3.9]{CS}, we have a canonical equivalence $\Phi_{V/D}$ that fits into a commutative square
\begin{equation}\label{Eq:GLSM1}
\xymatrix@C+4pc{
\Res_{D}\left(\T^*_{V/D}(\omega_{D}[d])\right)  \ar@{.>}[d]^{\simeq}_{\Phi_{V/D}} \ar[r]^-{\Res_{D}(\lambda_{V/D}(\omega_D[d]))} & \Res_{D}\left(\sA^{1}_{D}(\omega_{D}[d])\right)  \ar[d]^{\int_{D}}\\
 \T^*_{\Res_{D}(V)}[d] \ar[r]^-{\lambda_{\Res_{D}(V)}[d]} & \sA^1_B[d].
}	
\end{equation}

Then we can find a canonical equivalence $\Psi_{U/C}$ that fits into a morphism of fiber sequences
\begin{equation}\label{Eq:GLSM2}
\xymatrix{
\Res_{C}\left(\T^*_{U/C}(\omega_{C}^\log[d])\right) \ar[r] \ar@{.>}[d]^{\simeq}_{\Psi_{U/C}} &  \Res_{C}\left(\T^*_{U/C}(i_*\omega_{D}[d])\right)  \ar[d]^{\simeq}_{\Phi_{V/D}|_{\Res_C(U)}} \ar[r] &  \Res_{C}\left(\T^*_{U/C}(\omega_{C}[d])\right)   \ar[d]^{\simeq}_{\Phi_{U/C}} \\
\T^*_{\Res_{C}(U)/\Res_{D}(V)}[d-1] \ar[r] & \T^*_{\Res_{D}(V)}[d]|_{\Res_{C}(U)} \ar[r]  & \T^*_{\Res_{C}(U)}[d] 
}
\end{equation}
where the upper sequence is induced by the boundary structure $\partial_{D/C}$ of $i:D \to C$
and the middle vertical arrow is induced by the canonical equivalence
\[\Res_{C}\left(\T^*_{U/C}(i_*\omega_{D}[d])\right) \simeq \Res_{D}\left(\T^*_{V/D}(\omega_{D}[d])\right)\times_{\Res_{D}(V)}\Res_{C}(U) .\]
Here the commutativity of the right square follows from the functoriality of $\Phi$.

We then claim that the path
\[\AKSZ_{C}\left(\lambda_{U/C}(\omega_{C}^{\log}[d])\right) : 0 \lra \AKSZ_{D}\left(\lambda_{V/D}(\omega_{D}[d])\right)|_{\Res_{C}\left(\T^*_{U/C}(\omega_{C}^\log[d])\right)}\]
is equivalent to the canonical path given by the lower sequence in \eqref{Eq:GLSM2}
\[0 \lra \lambda_{\Res_{D}(V)}[d]|_{\T^*_{\Res_{C}(U)/\Res_{D}(V)}[d-1]},\]
under the vertical equivalences in \eqref{Eq:GLSM2}.
Indeed, the boundary structure $\partial_{D/C}$ of $i:D \to C$ gives us a morphism of fiber sequences
\[\xymatrix{
\Res_{C}\left(\T^*_{U/C}(\omega_{C}^\log[d])\right) \ar[r] \ar[d]^{\Res_{C}\left(\lambda_{U/C}(\omega_{C}^{\log}[d])\right)} &  \Res_{D}\left(\T^*_{V/D}(\omega_{D}[d])\right)\times_{\Res_{D}(V)}\Res_{C}(U)  \ar[d]^{\Res_{D}\left(\lambda_{U/C}(\omega_{C}^{\log}[d])\right)} \ar[r] &  \Res_{C}\left(\T^*_{U/C}(\omega_{C}[d])\right)   \ar@{.>}[d] \\
\Res_{C}\left(\sA^1_{C}(\omega_{C}^\log[d])\right) \ar[r] \ar[d] &  \Res_{D}\left(\sA^1_{D}(\omega_{D}[d])\right)  \ar[r] \ar[d]^{\int_{D}} &  \Res_{C}\left(\sA^1_{C}(\omega_{C}[d])\right) \ar[d]^{\int_{C}} \\
B \ar[r]^-0 & \sA^1_B[d] \ar@{=}[r] & \sA^1_B[d].
}\]
On the other hand, we have a canonical commutative diagram
\[\xymatrix{
\T^*_{\Res_{C}(U)/\Res_{D}(V)}[d-1] \ar[r] & \T^*_{\Res_{D}(V)}[d]|_{\Res_{C}(U)} \ar[r] \ar[d]  & \T^*_{\Res_{C}(U)}[d] \ar[d]^{\lambda_{\Res_C(U)}[d]} \\
& \T^*_{\Res_D(V)}[d] \ar[r]^{\lambda_{\Res_D(V)}[d]} & \sA^1_B[d].
}\]
Then the claim follows from the commutative square \eqref{Eq:GLSM1} for $U \to C$.

We now consider the general case $\alpha\not\simeq 0$.
We first construct an equivalence of derived stacks
\[\Phi_{V/D,\beta}:\Res_D\left(\T^*_{V/D,\beta}(\omega_D[d-1])\right) \xrightarrow{\simeq}  \T^*_{\Res_D(V),\AKSZ_D(\beta)}[d-1].\]
Indeed,
applying $\Res_{D}$ to \eqref{Eq:GLSM4} for $(V \to D, \beta, \omega_D[d])$,
we can form a fiber square
\[\xymatrix{
\Res_D\left(\T^*_{V/D,\beta}(\omega_D[d-1])\right) \ar[r] \ar[d] \cart & \Res_D(V) \ar[d]^0 \\ \Res_D(V) \ar[r]^-{\Res_D(\Gamma_{\beta})} & \Res_D\left(\T^*_{V/D}(\omega_D[d])\right).
}\]
Then the claimed equivalence follows from the equivalences
\[\Res_D\left(\T^*_{V/D}(\omega_D[d])\right) \xrightarrow[\simeq]{\Phi_{V/D}} \T^*_{\Res_D(V)}[d] \and \Phi_{V/D}\circ\Res_D(\Gamma_{\beta}) \simeq \Gamma_{\AKSZ_D(\beta)},\]
induced by \eqref{Eq:GLSM1}.

We then claim that the above equivalence $\Phi_{V/D,\beta}$ is a locked symplecto-morphism.
Indeed, the locked symplectic form on $\T^*_{V/D}(\omega_D[d])$ 
is given by the commutative diagram
\begin{equation}\label{Eq:GLSM5}
\xymatrix@C+1pc{
\T^*_{V/D,\alpha}(\omega_D[d-1]) \ar[r] \ar[d] \cart & V \ar[d]^{\Gamma_\beta} \ar[r]^-{\beta} & \sA^{1,\lc}_D(\omega_D[d]) \ar[d] \ar[r] \cart& D \ar[d]^0 \\ V \ar[r]^-{0} & \T^*_{V/D}(\omega_D[d]) \ar[r]^-{\lambda_{V/D}(\omega_D[d])} & \sA^1_D(\omega_D[d]) \ar[r]^-d & \sA^{2,\lc}_D(\omega_D[d]).
}\end{equation}
Applying $\Res_D$ to \eqref{Eq:GLSM5}, we obtain the claimed equivalence
\begin{equation}\label{Eq:GLSM8}
\Phi_{V/D,\beta}^*(\theta_{\Res_D(V),\AKSZ_D(\beta)}[d])\simeq \AKSZ_D\left(\theta_{V/D,\beta}(\omega_D[d])\right),\end{equation}
from the canonical morphism of fiber sequences
\begin{equation}\label{Eq:GLSM7}
\xymatrix{
\Res_D\left(\sA^{1,\lc}_D(\omega_D[d])\right) \ar[r] \ar[d]^{\int_D} & \Res_D\left(\sA^{1}_D(\omega_D[d])\right) \ar[r] \ar[d]^{\int_D} & \Res_D\left(\sA^{2,\lc}_D(\omega_D[d])\right) \ar[d]^{\int_D} \\
\sA^{1,\lc}_B[d] \ar[r] & \sA^1_B[d] \ar[r] & \sA^{2,\lc}_B[d].
}\end{equation}

Our next step is to construct an equivalence of derived stacks
\[\Psi_{U/C,\alpha}:\Res_C\left(\T^*_{U/C,\alpha}(\omega_C^\log[d-1])\right) \xrightarrow{\simeq} \T^*_{\Res_C(U)/\Res_C(V),\AKSZ_C(\alpha)_{/}}[d-2].\]
Indeed, applying $\Res_C$ to \eqref{Eq:GLSM4}, we obtain a fiber square
\[\xymatrix{
\Res_C\left(\T^*_{U/C,\alpha}(\omega_C^\log[d-1])\right) \ar[r] \ar[d] \cart & \Res_C(U) \ar[d]^0 \\ \Res_C(U) \ar[r]^-{\Res_C(\Gamma_{\alpha})} & \Res_C\left(\T^*_{U/C}(\omega_C^\log[d])\right).
}\]
Hence the claimed equivalence follows from the canonical equivalences
\[\Res_C\left(\T^*_{U/C}(\omega_C^\log[d])\right) \xrightarrow[\simeq]{\Psi_{U/C}} \T^*_{\Res_C(U)/\Res_C(V)}[d-1], \quad \Psi_{U/C}\circ\Res_C(\Gamma_{\alpha}) \simeq \Gamma_{\AKSZ_C(\alpha)_{/}},\]
which can be deduced from \eqref{Eq:GLSM2}.

Our final step is to show that $\Psi_{U/C,\alpha}$ preserves the Lagrangian structures under the symplecto-morphism $\Phi_{V/D,\beta}$.
Indeed, the left commutative square in \eqref{Eq:GLSM2} gives us the desired commutative square \eqref{Eq:GLSM6}.
Replacing $(D,\omega_D)$ by $(C,\omega_C^\log)$ in \eqref{Eq:GLSM7},
we obtain the desired equivalence
\[\eqref{Eq:GLSM8}\circ \Psi_{U/C,\alpha}^*(\eta)\simeq \AKSZ_{C}(\alpha)\]
where $\eta$ is the Lagrangian structure induced by the commutative square \eqref{Eq:GLSM21}.
\end{proof}

Now the computation of the twisted maps to critical loci follows immediately.

\begin{proof}[Proof of Theorem \ref{Prop:GLSMmain}]
Applying Proposition \ref{Prop:Weil} to Example \ref{Ex:TwMaps/Inertia},
we obtain a canonical equivalence of Lagrangians
\begin{multline*}
    \left(\M_{g,n}^{\omega^{\log}}(\Crit_U(w)) \to \prod_{1\leq k\leq n}\left(I(\Crit_U(w))\times_{B^2\mu,\Sigma_k}\M^{\tw}_{g,n}\right)\right)  
    \\\simeq \left(\M_{g,n}^{\omega^{\log}}(U)(\sigma_w)  \to \prod_{1\leq k\leq n}\left(\Crit_{I(U)}(Iw)\times_{B^2\mu,\Sigma_k}\M^{\tw}_{g,n}\right)\right).\end{multline*}
Since the diagonal of $B^2\mu_{r}$ is {\'e}tale, we may replace:
$(-)\times_{B^2\mu,\Sigma_k}(-)\leadsto (-)\times(-)$.
\end{proof}


\end{document}